\documentclass{article}
\usepackage[utf8]{inputenc}   
\usepackage[left=2.5cm,right=2.5cm,top=2.5cm,bottom=2.5cm]{geometry}
\usepackage[T1]{fontenc}                
\usepackage[english]{babel}              
\usepackage{graphicx}
\usepackage{amsmath,amsfonts,amssymb,amsthm}  
\usepackage{xspace}
\usepackage{geometry} 
\usepackage[overload]{empheq}
\usepackage{dsfont} 
\usepackage{stmaryrd}
\usepackage{tikz}      
\usepackage{enumerate}
\usepackage{subfig}
\usepackage{float}

\usepackage{hyperref}
\hypersetup{
colorlinks = true,
urlcolor = blue,
linkcolor = blue,
citecolor = blue,
}

\newcommand{\black}{\color{black}}

\newtheorem{theorem}{Theorem}
\newtheorem{definition}{Definition} 
\newtheorem{proposition}{Proposition}
\newtheorem{lemma}{Lemma}
\newtheorem{corollary}{Corollary}

\usepackage{amsmath,amssymb,euscript ,yfonts,psfrag,latexsym,dsfont,graphicx,bbm,color,amstext,wasysym,epsfig,subfig,parskip,textcomp}

\usepackage{pgfplots}
\usetikzlibrary{arrows}

\usepackage{natbib}
\bibpunct{(}{)}{;}{a}{,}{,}
\usepackage{hyperref}
\hypersetup{
colorlinks = true,
urlcolor = blue,
linkcolor = blue,
citecolor = blue,
}

\newtheorem{assumption}{Assumption}
\DeclareMathOperator*{\argmax}{arg\,max}

\DeclareMathOperator*{\argmin}{arg\,min}
\title{\textbf{Smooth transport map via diffusion process}}
\usepackage{authblk}
\author{Arthur St\'ephanovitch}
\affil{\small D\'epartement de Math\'ematiques et Applications\\
       Ecole Normale Sup\'erieure, Université PSL, CNRS\\
       F-75005 Paris, France,\\ \texttt{stephanovitch@dma.ens.fr}}

\makeatletter
\def\thm@space@setup{%
    \thm@preskip=8pt plus 2pt minus 2pt 
    \thm@postskip=1pt plus 2pt minus 2pt 
}
\makeatother
\date{}

\begin{document}
\maketitle

\begin{abstract}
\noindent We extend the classical regularity theory of optimal transport to non-optimal transport maps generated by heat flow for perturbations of Gaussian measures. Considering probability measures of the form $d\mu(x) = \exp\left(-\frac{|x|^2}{2} + a(x)\right)dx$ on $\mathbb{R}^d$ where $a$ has Hölder regularity $C^\beta$ with $\beta\geq 0$; we show that the Langevin map transporting the $d$-dimensional Gaussian distribution onto $\mu$ achieves Hölder regularity $C^{\beta + 1}$, up to a logarithmic factor. We additionally present applications of this result to functional inequalities and generative modeling.
    
\end{abstract}

\tableofcontents

\section{Introduction}

\subsection{Context}
\paragraph*{Regularity of optimal transport maps}

Given two probability densities $\mu_1$ and $\mu_2$ on $\mathbb{R}^d$ with finite second moments, the quadratic optimal transport problem consists in solving:
\begin{equation}\label{eq:OT}
    \inf_{T_{\#} \mu_1=\mu_2} \int \|x-T(x)\|^2 d\mu_1(x),
\end{equation}
where the transport constraint $T_{\#} \mu_1=\mu_2$ imposes that for all Borel subset $A$ of $\mathbb{R}^d$,
$$
\int_{T^{-1}(A)} d\mu_1(x) = \int_A d\mu_2(x).
$$

Thanks to the work of \cite{Brenier1991PolarFA} and \cite{McCann1995ExistenceAU}, it is now well known that there exists a unique solution $T$ to \eqref{eq:OT} and that this map can be written as $T=\nabla u$ with $u$ convex being solution in the weak sense to the Monge-Ampère equation
\begin{equation}
    \text{det}(\nabla^2 u)\mu_1(\nabla u) = \mu_2.
\end{equation}

The regularity theory for the optimal transport map $T$ initiated by \cite{Caffarelli92} has stimulated an intensive line of research (see the textbook \cite{figalli2017monge}) with many problems still opened to date. The classical statement of this theory (originally from \cite{Caffarelli92}) is the following:

\begin{theorem}[Informal version of Theorem 12.50 in \cite{villani2009optimal}]\label{theorem:caffarellisregularity}
    Let $\mu_1$ and $\mu_2$ be two probability measures being absolutely continuous with respect to the $d$-dimensional Lebesgue measure and that are supported on open bounded sets $\mathcal{X}$ and $\mathcal{Y}$ respectively. Provided that
their respective densities are bounded away from zero and infinity on their supports and that $\mathcal{Y}$ is convex; if the densities of $\mu_1$ and $\mu_2$ have Hölder regularity $\beta>0$ (with $\beta$ a non-integer), then the optimal transport map has Hölder regularity $\beta+1$.
\end{theorem}
 Unfortunately, Caffarelli’s boundedness assumptions on the domains and densities are restrictive as they exclude many important cases, such as Gaussian measures. Recent works, like \cite{Dario2019} and \cite{FigalliJhaveri2023}, extend the classical result to unbounded domains and unbounded densities, but obtain only local regularity. The main obstacle to achieving global regularity is that most methods boil down to showing that the ratio $\frac{\mu_2(x)}{\mu_1(\nabla u(x))}$ is bounded above and below, which then allows the application of classical regularity theory for elliptic equations \citep{de2014monge}. However, obtaining boundedness for this ratio is highly challenging as even the Gaussian Lipschitz case remains unresolved (Conjecture 1 in \cite{fathi2024transportation}). This difficulty has led researchers to investigate alternative approaches for constructing transport maps that may more readily be proven smooth.

\paragraph*{Transport maps from heat flows}
A recent advance in this direction is the transport map introduced by \cite{kim2012generalization}, which is constructed using a heat-diffusion process. The original motivation was to extend Caffarelli's contraction theorem \citep{caffarelli2000monotonicity} stating the existence of a $1$-Lipschitz map pushing forward the
Gaussian measure onto a more log-concave one. This result allows to transfer the  functional inequalities satisfied by the Gaussian (like Poincaré and log-Sobolev) to the transported measure.
 While Caffarelli's theorem relies on the optimal transport map, making it difficult to achieve regularity in a wide range of settings, the Kim-Milman map offers a more accessible route to obtaining Lipschitz estimates. A sketch of the map's construction is as follows:\\
For all $t>0$ and $f:\mathbb{R}^d\rightarrow \mathbb{R}$ bounded, define  the Ornstein-Uhlenbeck semigroup by
$$P_tf(x)=\int f(xe^{-t}+\sqrt{1-e^{-t}}z)d\gamma_d(z),$$
with $\gamma_d$ the standard $d$-dimensional Gaussian measure. Then, we can show that under appropriate conditions the flow $S_t:\mathbb{R}^d\rightarrow \mathbb{R}^d$ defined by
\begin{align*}
d_tS_t(x)&=-\nabla (\log P_tf)(S_t(x))\\
S_0(x) & = x,
\end{align*}
satisfies that $(S_t)_{\# }f=f_t d\gamma_d$ with $f_0=f$ and $f_t\rightarrow 1$ when $t\rightarrow 1$. Therefore, the limit map \emph{$(S_\infty)^{-1}$} pushes forward $\gamma_d$ onto $fd\gamma_d$.

Following this idea, several recent works build transport maps using diffusion process and prove their Lipschitz continuity under suitable conditions. We can notably cite the  Brownian transport map  from \cite{mikulincer2024brownian} (pushes forward the Wiener process to Euclidean semi log concave measures), the  Langevin transport map  from \cite{fathi2024transportation} (pushes forward log concave measures to log-Lipschitz perturbations in the Euclidean and manifolds settings) and the  Föllmer transport map  from \cite{dai2023lipschitz}  (pushes forward the Gaussian measure to semi log concave measures).  In this paper, we focus on the Langevin transport map which is built using the Ornstein–Uhlenbeck process, a flow widely used in practice \citep{song2020score,bjork2009arbitrage,davies1973harmonic}. We actually show that the "Föllmer transport map" from \cite{dai2023lipschitz} is equal to the Langevin map, which allows for a simple formulation using flows that evolve over finite time. 
Although we focus on the Langevin transport map, we believe that our regularity results could be extended to other flows given that they have a similar behavior as explained in \cite{gao2023Gaussian}. 

\subsection{Contributions}
In this work, we explore the regularity of transport maps constructed via diffusion processes, extending the classical theory beyond Lipschitz continuity. While Caffarelli's seminal regularity results for Monge–Ampère equations have provided foundational insights into optimal transport maps, obtaining regularity in unbounded settings remains a significant challenge. Here, we address this gap by demonstrating that transport maps generated through diffusions can indeed exhibit higher-order regularity, leading to the construction of the first smooth transport map within this framework. These results open new avenues for understanding the structure and properties of transport maps in broader settings. In Section \ref{sec:preandmain} we develop the framework in full generality, where we emphasize
the following key contributions:
\subsubsection*{1. Higher order regularity of transport maps (Theorems \ref{theo:thetheo} and \ref{theo:thetheo2})}  Mirroring the classical regularity theory of optimal transport (Theorem \ref{theorem:caffarellisregularity}), we prove that the Langevin transport map between the $d$-dimensional Gaussian distribution and a log $\beta$-Hölder perturbation, is of Hölder regularity $\beta+1$. Furthermore, we obtain a Lusin-type result for the transport of the Gaussian to a class of measures supported on the ball. It is shown that a set of mass $1-\epsilon$ is transported by a $(\beta+1)$-Hölder map having a norm controlled by a logarithmic power of $\epsilon$.

\subsubsection*{2. Applications}
The existence of smooth transport maps enables a range of applications, which we present in the subsequent sections.

{\small\textbf{Transfer of functional inequalities (Corollary \ref{coro:logsogen}) } }Theorem \ref{theo:thetheo} allows to transfer the functional inequalities involving higher order derivatives from the Gaussian measure to the transported one. As applications we extend the class of measures satisfying the generalized log-Sobolev inequality \citep{rosen1976sobolev}.

{\small\textbf{Applications to generative models (Theorem \ref{theo:minimaxgan} and Corollary \ref{coro:diffusregu})}} We provide applications of Theorem \ref{theo:thetheo} to the estimation of densities by Wasserstein Generative Adversarial Networks (WGAN) and score-based generative models. Firstly,  the existence of smooth transports maps allows to show the optimality of the WGAN estimator within the (widely used) Gaussian setting. Secondly, given the strong correlation between the score function in diffusion models and the transport map's velocity field, we show that higher regularity of the target distribution transfers to  higher regularity of the score function.

\subsection{Organization of the paper}
The paper is organized as follows. In Section \ref{sec:preandmain}, we present the main results in detail, including  necessary notations and the construction of the Föllmer/Langevin transport map. Section \ref{sec:nablak} is dedicated to proving bounds on the derivatives of the transport map, whereas the non-integer regularity is treated in the appendix Section \ref{sec:holdregulastderiv}. Finally, we present in Section \ref{sec:app} the applications to the extension of the log-Sobolev inequality as well as to  density estimation via generative methods.

\subsection*{Acknowledgments}
The author would like to thank Max Fathi for his insightful discussions and valuable advice, which greatly contributed to the development of this work.

\section{Preliminaries and main results}\label{sec:preandmain}
\subsection{Hölder spaces}
We begin by precisely defining the notion of regularity we describe. For $\eta >0$, $\mathcal{X},\mathcal{Y}$ two open subsets of Euclidean spaces and $f=(f_1,...,f_p)\in C^{\lfloor \eta \rfloor}(\mathcal{X},\mathcal{Y})$ the set of $\lfloor \eta \rfloor:=\max \{k\in \mathbb{N}_0 |\ k\leq \eta\}$  times differentiable functions (almost everywhere), denote $\partial^\nu f_i := \frac{\partial^{|\nu|}f_i}{\partial x_1^{\nu_1}...\partial x_d^{\nu_d}}$ the partial differential operator  for any multi-index $\nu = (\nu_1,...,\nu_d)\in \mathbb{N}_0^d$ with $|\nu|:=\nu_1+...+\nu_d\leq \lfloor \eta \rfloor$. For $\alpha\geq 0$, write 
$$
\|f_i\|_{\eta-\lfloor \eta \rfloor,\alpha}:=\sup \limits_{x\neq y}\ \frac{f_i(x)-f_i(y)}{\|x-y\|^{\eta - \lfloor \eta \rfloor}} \log\big((\|x-y\|\wedge\frac{1}{2})^{-1}\big)^{\alpha}
$$
and let 
\begin{align*}
\mathcal{H}^{\eta,\alpha}_K(\mathcal{X},\mathcal{Y}):=\Big\{& f  \in C^{\lfloor \eta \rfloor}(\mathcal{X},\mathcal{Y}) \ \big| \max \limits_{i} \sum \limits_{|\nu|\leq \lfloor \eta \rfloor} \|\partial^\nu f_i\|_{L^\infty(\mathcal{X},\mathcal{Y})} + \sum \limits_{|\nu|  = \lfloor \eta \rfloor} \|\partial^\nu f_i\|_{\eta-\lfloor \eta \rfloor,\alpha}   \leq K\Big\}
\end{align*}
denote the ball of radius $K>0$ of the Hölder space $\mathcal{H}^{\eta,\alpha}(\mathcal{X},\mathcal{Y})$, the set of functions $f:\mathcal{X}\rightarrow \mathcal{Y}$ of regularity $\eta$ with additional weak regularity $\alpha$. We have in particular that for all $\epsilon>0$,
$$\mathcal{H}^{\eta+\epsilon}(\mathcal{X},\mathcal{Y})\xhookrightarrow{} \mathcal{H}^{\eta,\alpha}(\mathcal{X},\mathcal{Y}) \xhookrightarrow{} \mathcal{H}^{\eta}(\mathcal{X},\mathcal{Y}),$$
where we write $A\xhookrightarrow{} B$ if the function space $A$ compactly injects in the function space $B$ and $\mathcal{H}^{\eta}(\mathcal{X},\mathcal{Y}):=\mathcal{H}^{\eta,0}(\mathcal{X},\mathcal{Y})$ denote the classical Hölder space. In the case where $\eta$ is not an integer, the space $\mathcal{H}^{\eta,\alpha}$ is a Zygmund type space \citep{Haroske2000}, which coincides with the Besov space $\mathcal{B}^{\eta,\alpha}_{\infty,\infty}$ \citep{haroske2006envelopes}.

\subsection{Additional notation}
In the following, $\mu$ is used
to denote
both the measure and the corresponding density when the probability measure $\mu$ is absolutely continuous with respect to $\lambda^d$ the $d$-dimensional Lebesgue measure. We write $\langle \cdot, \cdot \rangle$ the dot product on $\mathbb{R}^p$ and $\|\cdot\|$ the Euclidean norm. For $a,b\in \mathbb{R}$, $a\wedge b$ and $a\vee b$ denote the minimum and maximum value between $a$ and $b$ respectively. We write $\text{Lip}_1$ for the set of 1-Lipschitz functions. The support of a function $f$ is denoted by $supp(f)$. We  denote by $\text{Id}$ the identity from a Euclidean space to itself. With a slight abuse of notation, the probability density function of a probability measure $\mu$ on $\mathbb{R}^d$ (being absolutely continuous with respect to the Lebesgue measure) will also be denoted by  the map $x\mapsto \mu(x)$. The cardinal of a set $A$ will be written as $|A|$.

For any map $f:\mathbb{R}^d\rightarrow \mathbb{R}^l$ we denote by $\nabla^k f$ the $k$-th differential of $f$ and by $\|\nabla f(x)\|$ its operator norm at point $x$. We denote by $(\nabla f(x))^\top$ its transpose matrix and if $k\leq l$, $\lambda_{\max}((\nabla f(x) )^\top \nabla f(x))$ corresponds to the largest eigenvalue of the matrix $(\nabla f(x) )^\top \nabla f(x)$. For $w\in \mathbb{R}^d$, we write $\nabla^k f(x) w^k: = \nabla^k f(x) (w,...w)$  with $w^k=(w,...w) \in \mathbb{R}^{d\times k}$. For a univariate map $f\in C^1(\mathbb{R},\mathbb{R})$, we denote by $f^{(k)}$ the $k$-th derivative of $f$. We will denote $\nabla \cdot f(x):= \sum_{i=1}^d \partial_i f(x)$ the divergence of $f$ at the point $x$. For a probability measure $\mu$ on $\mathbb{R}^d$, the $p$-Lebesgue space with respect to the integration of $\mu$ (with $p\in \mathbb{N}_{>0}$) will be denoted as $L^p_\mu(\mathbb{R}^d)=\{f:\mathbb{R}^d\rightarrow \mathbb{R}| \int (f(x))^p d\mu(x)<\infty\}$. Let $\Omega \subset \mathbb{R}^n$ be an open set, and let $p \in \mathbb{N}$ and $1 \leq q < \infty$. Likewise, the Sobolev 
 space is defined as $
W^{p,q}(\mathbb{R}^d) = \{ f \in L^q(\mathbb{R}^d) |\sum_{|\alpha|\leq p} \|\partial^\alpha f\|_{L^q_{\lambda^d}} <\infty\}
$ and its local analog $
W^{p,q}_{loc}(\mathbb{R}^d) = \{ f |\forall \mathbb{K}\subset \mathbb{R}^d \text{ compact },\sum_{|\alpha|\leq p} \|\partial^\alpha f\|_{L^q_{\lambda^d}(\mathbb{K})} <\infty\}
$. We write $\mathbb{E}_\mu[f(X)]$ for the expectation of 
the map $f:\mathbb{R}^d\rightarrow \mathbb{R}$, with respect to $\mu$ and $\text{Var}_\mu[f(X)]$ for its variance. For a map $T:\mathbb{R}^d\rightarrow \mathbb{R}^d$, we write $T_{\# }\mu$ for the push forward of the measure $\mu$ by $T$. For $w\in \mathbb{R}^d$, $k\in \mathbb{N}_{>0}$ and $F:\left(\mathbb{R}^d\right)^k\rightarrow \mathbb{R}^d$ a multilinear map, we write $Fw^k:=F(w,....,w)$ the value of $F$ taken at the vector $(w,...,w)\in \left(\mathbb{R}^d\right)^k$.

\subsection{Main results}
Let $d\in \mathbb{N}_{>0}$ and $\gamma_d$ denote the standard $d$-dimensional Gaussian measure. We are interested in finding a transport map from  $\gamma_d$ to a probability density function $p:\mathbb{R}^d\rightarrow \mathbb{R}$ of the form \begin{equation}\label{eq:p}
p(x):=\exp\left(-\frac{\|x\|^2}{2}+a(x)\right),
\end{equation}
such that the regularity of the transport map depends on the regularity of $a$.
Let  $\beta\geq 0$ represent the regularity of the map $a$ in \eqref{eq:p} and $\theta>1$ represent its additional weak regularity, i.e. we suppose that $a$ belongs to $ \mathcal{H}^{\beta,\theta}_K(\mathbb{R}^d,\mathbb{R})$ for a certain $K>0$. Throughout,  $C,C_2>0$ correspond to constants that can vary from line to line and that only depend on $\theta,\beta,d$ and $K$. Let us state the main result of the paper.

\begin{theorem}\label{theo:thetheo}
Let $p:\mathbb{R}^d\rightarrow \mathbb{R}$ be a probability density function of the form $$p(x)=\exp\left(-\frac{\|x\|^2}{2}+a(x)\right),$$
such that $a \in \mathcal{H}^{\beta,\theta}_K(\mathbb{R}^d,\mathbb{R})$ with $\beta\geq 0$, $\theta>1$ and $K>0$. Then, there exists a map $T\in C^{\lfloor \beta \rfloor+1}(\mathbb{R}^d,\mathbb{R}^d)$ satisfying $\nabla T\in  \mathcal{H}^{\beta}_C  (\mathbb{R}^d)$ and $T_{\#}\gamma_d=p$.
\end{theorem}

Theorem \ref{theo:thetheo} asserts that there exists a transport map of Hölder regularity $\beta+1$ from the Gaussian measure to a log-$\beta$ smooth perturbation. To the best of our knowledge, this is the first result providing high-order regularity of transport in the unbounded setting. The additional logarithmic factor $\theta>1$ is expected for heat flows as the solution of the linearized equation has also a logarithmic factor in its regularity (see Proposition 2.4. in \cite{gallouet2018regularity}).  The  map $T$ exhibited in Theorem \ref{theo:thetheo} is the
classical Langevin transport map generated by heat flows \citep{kim2012generalization,fathi2024transportation,mikulincer2023lipschitz}.
We first prove Section \ref{sec:follandlang} that the Föllmer transport map introduced in \cite{dai2023lipschitz} is actually equal to the Langevin map.
This result allows to obtain a forward formulation in finite time of the Langevin transport map.  Using this formulation, the proof of Theorem \ref{theo:thetheo} is then divided into two parts: Section \ref{sec:nablak}
focuses on establishing bounds for the derivatives, while the non-integer regularity is addressed in Appendix Section \ref{sec:holdregulastderiv}.

We prove that the inverse of the transport map from Theorem \ref{theo:thetheo} is actually also $(\beta+1)$-Hölder, which allows to obtain $C^{\beta+1}$ diffeomorphisms between Gaussian deformations.

\begin{corollary}
\label{theo:diffeo}
Let $p,q:\mathbb{R}^d\rightarrow \mathbb{R}$ be probability density functions of the form $$p(x)=\exp\left(-\frac{1}{2}\|\Sigma_p^{-1/2}(x-\mu_p)\|^2+a(x)\right) \quad \text{ and } \quad q(x)=\exp\left(-\frac{1}{2}\|\Sigma_q^{-1/2}(x-\mu_q)\|^2+b(x)\right),$$
with $\mu_p,\mu_q\in \mathbb{R}^d$ and $\Sigma_p,\Sigma_q$ symmetric  matrices satisfying $K^{-1}\text{Id}\preceq \Sigma_i\preceq K\text{Id}$ for $i\in \{p,q\}$.
Suppose that $a,b \in \mathcal{H}^{\beta,\theta}_K(\mathbb{R}^d,\mathbb{R})$ with $\beta\geq 0$, $\theta>1$ and $K>0$. Then, there exists a map $T\in C^{\lfloor \beta \rfloor+1}(\mathbb{R}^d,\mathbb{R}^d)$ satisfying $\nabla T,\nabla T^{-1}\in  \mathcal{H}^{\beta}_C  (\mathbb{R}^d)$ and $T_{\# }p=q$. 
\end{corollary}

The proof of Corollary \ref{theo:diffeo} can be found in Section \ref{sec:diffeog}.  As a by-product of Theorem \ref{theo:thetheo}, we also obtain a result for densities that are supported on the ball. These densities can be written as 
\begin{equation}\label{eq:pbis}
p(x)=\exp\left(-u(\|x\|^2)+a(x)\right),
\end{equation}with $u$ satisfying the following assumption.
\begin{assumption}\label{assum:u}
    The map $u\in C^{(\lfloor \beta\rfloor +1)\vee 2}([0,1),[0,\infty))$ satisfies $u(t)\geq (1-t)^{-\frac{1}{K}}$ and $0\leq u^{(l)}(t)\leq (1-t)^{-(K+l)}$ for all $t\in [0,1)$ and $l\in \{1,...,(\lfloor \beta\rfloor +1)\vee 2\}$.
\end{assumption}
Assumption \ref{assum:u} imposes that the probability density $p$ defined in \eqref{eq:pbis} decays at a controlled rate near the boundary of the ball. In particular the density is not bounded from below, which means that it does not fulfill the conditions required for the classical regularity results established by Caffarelli (Theorem \ref{theorem:caffarellisregularity}). The adaptation of Theorem \ref{theo:thetheo} to the compact case is the following.

\begin{theorem}\label{theo:thetheo2}
Let $p:B^d(0,1)\rightarrow \mathbb{R}$ be a probability density function of the form $$p(x)=\exp\left(-u(\|x\|^2)+a(x)\right),$$
with $a \in \mathcal{H}^{\beta,\theta}_K(\mathbb{R}^d,\mathbb{R})$, $\beta>0$, $\theta>1$, $K>0$
and $u$ satisfying Assumption \ref{assum:u}.
     Then, there exists $T\in C^{\lfloor \beta \rfloor+1}(\mathbb{R}^d,B^d(0,1))$ such that  $ T\in  \mathcal{H}^{1}_C  (\mathbb{R}^d)$ and $T_{\#}\gamma_d=p$. Furthermore, for all $\epsilon\in (0,1)$, the restriction of $T$ to $B^d(0,\log(\epsilon^{-1}))$ belongs to $ \mathcal{H}_{C\log(\epsilon^{-1})^{C_2}}^{\beta+1}\Big(B^d(0,\log(\epsilon^{-1})),B^d(0,1)\Big)$.
\end{theorem}
This Lusin-type result is weaker than Theorem \ref{theo:thetheo} as the bound $\log(\epsilon^{-1})^{C_2}$ on the high-order derivatives of the transport map diverges for points at infinity. Nevertheless, it guarantees that a set of mass $1-\epsilon$ is transported by a $(\beta+1)$-Hölder map having a norm controlled by a logarithmic power of $\epsilon$. The log factor arises from the growth of the derivatives of the function $r=p/\gamma_d$ near the boundary of the ball.  Note that in the case $\beta \in (0,1)$, the result $ T\in  \mathcal{H}^{1}_C  (\mathbb{R}^d)$ improves on the Lipschitz result of \cite{dai2023lipschitz} as we do not require semi-log-concavity. 
The proof of Theorem \ref{theo:thetheo2} can be found in the Appendix Section \ref{sec:prooftheo2}, sharing the same reasoning than for Theorem \ref{theo:thetheo} (developed in Section
\ref{sec:generalreas}).

\subsection{Transport maps from diffusion processes}
The map $T$ used in Theorems \ref{theo:thetheo} and \ref{theo:thetheo2} is the Langevin transport map generated by heat flows \citep{fathi2024transportation}. It is built using the backward flow of the Langevin dynamic (see Section \ref{sec:follandlang}). To achieve a simpler formulation, we first prove that
the "Föllmer transport map" from \cite{dai2023lipschitz} is equal to the Langevin map. This enables to have a forward formulation over finite time of the transport map.
\subsubsection{The Föllmer transport map}

Let us recall the construction of the Föllmer transport map for  a probability density $p$ on $\mathbb{R}^d$ being absolutely continous with respect to the $d$-dimensional Gaussian measure and having finite 2nd moment.

Let $\epsilon\in (0,1)$, the Föllmer diffusion process $(\overline{F}_t)_{t\in [0,1-\epsilon]}$ is defined by the  Itô Stochastic Differential Equation (SDE):
\begin{equation*}
    d\overline{F}_t=-\frac{1}{1-t}\overline{F}_tdt+\sqrt{\frac{2}{1-t}}dB_t,\ \ \overline{F}_0\sim p,
\end{equation*}
with $B_t$ a Brownian motion in $\mathbb{R}^d$.
By Theorem 2.1 in \cite[Chapter 9]{revuz2013continuous},  the diffusion process $\overline{F}_t$  has a
unique strong solution on $[0,1-\epsilon]$.  Moreover, the probability distribution $\mu_t$ of $\overline{F}_t$ satisfies
\begin{equation}\label{eq:laxXt}
\mu_t\overset{d}{=} (1-t)X+\sqrt{t(2-t)}Y, 
\end{equation}
with $X\sim p$ and $Y\sim \gamma_d$ independent random variables.  Furthermore, from Theorem 5.14. in \cite{santambrogio2015optimal}, there exists a time-dependent vector field $W_t$ such that the marginal distribution flow $(\mu_t)_{t\in [0,1-\epsilon]}$ satisfies the continuity equation:
\begin{equation}\label{eq:continuityeq}
\partial_t \mu_t = \nabla \cdot (\mu_t W_t),\quad \mu_0=p.
\end{equation}
Knowing the density of $\mu_t$, simple calculation gives  $W_t=V(1-t,\cdot)$, $t$ a.e. with
\begin{equation}\label{eq:V}
    V(t,x):=\frac{1}{t} \nabla \log Q_t r(x), \ \ r(x):=\frac{p(x)}{\gamma_d(x)}
\end{equation}
and 
\begin{equation}\label{eq:Qtr}
    Q_t r(x):= \int \varphi^{t,x}(y)r(y)d\lambda^d(y)= \int r(tx+\sqrt{1-t^2}z)d\gamma_d(z),
\end{equation}
with $\varphi^{t,x}$ the density of the $d$-dimensional Gaussian measure with mean $tx$ and covariance $(1-t^2)\text{Id}$. Now, for $t\in (0,1)$, we have
\begin{align*}
    V(t,x) = & \frac{1}{1-t^2}\int (y-tx)dp^{t,x}(y),
\end{align*}
with
$$p^{t,x}(y):=\frac{ 1}{Q_tr(x)}\ \varphi^{t,x}(y)r(y).$$
Then, by dominated convergence we obtain 
$$\lim_{t\rightarrow 0} V(t,x) = \int y dp(y),$$
so the velocity field $V$ is well defined for all $t\in [0,1).$ Under appropriate conditions (see Propositions \ref{prop:condiwelldifned} or \ref{prop:equalitytmaps}), we can define a flow of maps $\left(x\mapsto X_t(x)\right)_{t\in [0,1)}$ with $X_t$ solution to the equation
\begin{equation}\label{eq:PDE}
\partial_t X_t(x) =V(t,X_t(x)), \quad X_0(x)=x.
\end{equation}
Then, from \eqref{eq:continuityeq}  we have that $$(X_{t})_{\#}\gamma_d\overset{d}{=}\mu_{1-t},$$
in particular the map $X_1$ (when well-defined) transports the $d$-dimensional Gaussian distribution to the measure $p$.
Equation \eqref{eq:PDE} is well posed under minimal Cauchy-Lipschitz assumptions (Remark 7 in \cite{Ambrosio_Crippa_2014}).
\begin{proposition}\label{prop:condiwelldifned}
    Let $t\in (0,1]$ such that 
$$V\in L^1\Big([0,t],W^{1,\infty}_{loc}(\mathbb{R}^d,\mathbb{R}^d)\Big)\text{ and } \frac{V}{1+\|x\|}\in L^1\Big([0,t],L^{\infty}(\mathbb{R}^d,\mathbb{R}^d)\Big),$$
    then the equation \eqref{eq:PDE} has a unique solution $(X_s)_{s\in [0,t]}$, $x$ almost everywhere. Furthermore, for all $s\in [0,t]$, it verifies $(X_{s})_{\# }\gamma_d=\mu_{1-s}$.
\end{proposition}
We aim to transport the $d$-dimensional Gaussian distribution to probability measures $p$ that do not necessarily meet the conditions of Proposition \ref{prop:condiwelldifned} in time $t=1$. Nevertheless, transport maps can still be obtained, as demonstrated by the following lemma.

\begin{lemma}\label{lemma:convergenceofmap}
    Let $\mu$,$\nu$ two probability measures on $\mathbb{R}^d$ and $\mu_k$,$\nu_k$ two sequences of probability measures converging in distribution to $\mu$ and $\nu$ respectively. Suppose that the supports of $\mu$ and $\mu_k$  are open.  If there exists a sequence of map $T_k$ such that $(T_k)_{\# }\mu_k=\nu_k$ and $T_k\in \mathcal{H}^\alpha_C$, then there exists a map $T\in \mathcal{H}^\alpha_C$ satisfying $T_{\#} \mu=\nu$. 
\end{lemma}

\begin{proof}
 The existence of the Lipschitz transport map $T$ is given by Lemma 2.1  in  \cite{neeman2022lipschitz} and the higher order regularity can be deduced from the Arzelà–Ascoli theorem applied to each derivatives. 
\end{proof}

The proof of Theorems \ref{theo:thetheo} and \ref{theo:thetheo2} consists in showing that the transport map $X_1$, the solution to \eqref{eq:PDE} when $t\rightarrow 1$, is $(\beta+1)$-Hölder.

\subsubsection{Equality between the Föllmer and the Langevin transport map}\label{sec:follandlang}
In this work, we focus on establishing the regularity of the Langevin transport map. This map is typically constructed as the inverse of the limiting solution of an ordinary differential equation (ODE), as detailed below. In order to use a simple formulation, we show that the Föllmer transport map from the previous section is actually equal to the Langevin map.

Let us first recall the construction of the Langevin transport map. The Ornstein–Uhlenbeck process $(\overline{L}_t)_{t\in [0,\infty)}$ is defined by the following Itô SDE:
\begin{equation*}
    d\overline{L}_t=-\overline{L}_tdt+\sqrt{2}dB_t,\ \ \overline{L}_0\sim p,
\end{equation*}
with $B_t$ a Brownian motion in $\mathbb{R}^d$. We can show (Lemma 2 in \cite{mikulincer2023lipschitz}) that
\begin{equation*}
\overline{L}_t\overset{d}{=}\mu_{1-e^{-t}}.
\end{equation*}
Defining the flow of maps 
$\left(x\mapsto L_{s,t}(x)\right)_{s\in [0,\infty),t\in [0,s]}$ with $L_{s,t}$ solution to
\begin{equation*}
 \partial_t L_{s,t}(x)=\nabla (\log Q_{e^{-(s-t)}} r)(L_{s,t}(x)), \ \ L_{s,0}(x)=x,   
\end{equation*}
we obtain that
$$(L_{t,t})_{\#}\mu_{1-e^{-t}}\overset{d}{=}p,$$
in particular we have that when both the Föllmer and Langevin flows are well defined,
$$
(X_{1})_{\#}\gamma_d\overset{d}{=}p \quad \text{ and } \quad (L_{\infty,\infty}) _{\# }\gamma_d=\lim_{t\rightarrow \infty}(L_{t,t})_{\#\mu_{1-e^{-t}}}\overset{d}{=}p,
$$
Therefore, the Föllmer transport map $X_1$ and the Langevin transport map $L_{\infty,\infty}$ both transport the Gaussian measure to $p$. We now prove that these maps are, in fact, identical under classical Cauchy-Lipschitz conditions for the existence of the final maps $X_1$ and $L_{\infty,\infty}$.
\begin{proposition}\label{prop:equalitytmaps}
    Let $p$  a probability density on $\mathbb{R}^d$ being absolutely continuous with respect to the $d$-dimensional Gaussian measure and having finite 2nd moment. If the velocity $V$ defined in \eqref{eq:V} satisfies 
    $$V\in L^1\Big([0,1],W^{1,\infty}(\mathbb{R}^d,\mathbb{R}^d)\Big),$$
then $X_1=L_{\infty,\infty}$ almost everywhere.
\end{proposition}
\begin{proof}
Both flows are well defined as $V$ satisfies the classical Cauchy-Lipschitz assumption (see Section 2 in \cite{Ambrosio_Crippa_2014}).
Define the flow of maps 
$\left(x\mapsto Y_{s,t}(x)\right)_{s\in [0,\infty),t\in [e^{-s},1]}$ by $$Y_{s,t} = L_{s,s+\log(t)}.$$ In particular $Y$ is solution to
\begin{align*}
 \partial_t Y_{s,t}(x)=&\frac{1}{t}\nabla (\log Q_{e^{-(s-(s+\log(t)))}} r)(Y_{s,t}(x)), \ \ Y_{s,e^{-s}}(x)=x.\\
 =&\frac{1}{t}\nabla (\log Q_{t} r)(Y_{s,t}(x)).
\end{align*}
Now, for the process $(a_t^s)_{t\in [e^{-s},1]}$ defined by
$$a_t^s:=\| Y_{s,t}(x)-X_t(x)\|^2,$$
we have (in a weak sense) that
\begin{align*}
    \partial_t a_t^s =& 2 \left\langle Y_{s,t}(x)-X_t(x),\frac{1}{t}\nabla (\log Q_t r)(Y_{s,t}(x))-\frac{1}{t}\nabla (\log Q_t r)(X_t(x)) \right\rangle\\
    \leq & 2\|\nabla V(t,\cdot)\|_{L^\infty} a_t^s.
\end{align*}
Then, by Gronwall's lemma we obtain that for all $t\in [e^{-s},1]$,
$$a_t^s\leq a^s_{e^{-s}}\exp\left(\int_{e^{-s}}^t 2\|\nabla V(u,\cdot)\|_{L^\infty}du\right) \leq Ca^s_{e^{-s}}.$$
Furthermore, we have
$$\lim_{s\rightarrow \infty} a_{e^{-s}}^s= \lim_{s\rightarrow \infty}\|x-\left(x+\int_0^{e^{-s}} \frac{1}{t}\nabla (\log Q_t r)(X_t(x))dt \right)\|=0.$$
Therefore,
$$L_{\infty,\infty} = \lim_{s\rightarrow \infty} Y_{s,1}=X_1.$$
\end{proof}

Having the equivalence of the two maps, we use the Föllmer formulation to prove Theorems \ref{theo:thetheo} and \ref{theo:thetheo2}.

\section{Estimate on  $\|\nabla X_1\|_{\mathcal{H}^{\lfloor \beta \rfloor}}$ in the setting of Theorem \ref{theo:thetheo}}\label{sec:nablak}
In this section, we show that under the assumptions of Theorem \ref{theo:thetheo} the transport map $X_1$ (the solution to \eqref{eq:PDE} at time $t=1$) satisfies that $\nabla X_1 \in \mathcal{H}^{\lfloor \beta \rfloor}_C(\mathbb{R}^d,\mathbb{R}^{d})$. The proof that $\nabla^{\lfloor \beta \rfloor} X_1 \in \mathcal{H}^{\beta-\lfloor \beta \rfloor}_C(\mathbb{R}^d)$ follows the same ideas but with more involved derivations. Therefore, it is postponed to the Appendix Section \ref{sec:holdregulastderiv}.
\subsection{General reasoning for Theorems \ref{theo:thetheo} and \ref{theo:thetheo2}}\label{sec:generalreas} 
This section is devoted to presenting a general framework for the proof of Theorem \ref{theo:thetheo}, which will also be used in the proof of Theorem \ref{theo:thetheo2}.  We will first assume that the flow of transport maps $(X_t)$ solution to \eqref{eq:PDE} is well defined for all $t\in [0,1)$ and derive estimates on $\lim_{t\rightarrow 1}\|\nabla^k X_t\|_\infty$. In subsequent sections, we  show that under the assumptions of Theorem \ref{theo:thetheo}, the flow $(X_t)$ is actually well defined for all $t\in [0,1]$ which will allow to conclude the proof using the derived estimates. For Theorem \ref{theo:thetheo2}, we will show Section \ref{sec:prooftheo2} that the flow $(X_t)$ is well defined for all $t\in [0,1)$ and will conclude the proof using Lemma \ref{lemma:convergenceofmap}.
\subsubsection{A general Gronwall bound}\label{sec:gronwall}
Let $(X_t)_{t\in [0,1)}$ be the solution of \eqref{eq:PDE} for $p$ satisfying either the assumptions of Theorems \ref{theo:thetheo} or \ref{theo:thetheo2}. For $k\in \{1,...,\lfloor \beta \rfloor +1\}$, $w\in \mathbb{S}^{d-1}$ and $x\in \mathbb{R}^d$, let us provide a bound for all $t\in [0,1)$ on the process 
$$
\alpha_t^k:=\|\nabla^k X_t(x)w^k\|^2.
$$ 
The map $x\mapsto \nabla^k X_t(x)$ is well defined for all $t\in [0,1)$ by classical regularity theory of ordinary differential equations (see for example Theorem 14.7 in \cite{tu2011manifolds}).
As 
$\partial_t X_t(x)=V(t,X_t(x))$,
we are going to compute the differentials of the map $x\mapsto \partial_t X_t(x)$ using the Faa di Bruno formula. For two unit-variate functions $f,g:\mathbb{R}\rightarrow \mathbb{R}$, the Faa di Bruno formula states that the $k$-derivative of the composition of $f$ and $g$ is
\begin{equation*}
    (f\circ g)^{(k)}(x) =\sum_{\pi \in \Pi_k} f^{(|\pi|)} (g(x)) \prod_{B\in \pi} g^{(|B|)} (x),
\end{equation*}
where $\Pi_k$ denotes the set of all partitions of the set $\{1,...,k\}$. For simplicity, in the multivariate setting we will keep this notation by writing for $f,g:\mathbb{R}^d\rightarrow \mathbb{R}^d$,
$$\nabla^k (f\circ g)(x)=\sum_{\pi \in \Pi_k} \nabla^{|\pi|}f(g(x)) \prod_{B\in \pi} \nabla^{|B|} g(x),$$
where the latter is defined by
\begin{align*}
&\left(\sum_{\pi \in \Pi_k} \nabla^{|\pi|}f(g(x)) \prod_{B\in \pi} \nabla^{|B|} g(x)\right)w^k\\
& := \sum_{\pi=(B_1,B_2,...,B_i) \in \Pi_k} \nabla^{|\pi|}f(g(x)) \left( \nabla^{|B_1|} g(x)w^{|B_1|},\nabla^{|B_2|}g(x)w^{|B_2|},...,\nabla^{|B_i|} g(x)w^{|B_i|}\right).
\end{align*}

Using this notation, we obtain by Schwarz theorem that

$$\partial_t\nabla^k  X_t(x)=\sum_{\pi \in \Pi_k} \nabla^{|\pi|}V(t,X_t(x)) \prod_{B\in \pi} \nabla^{|B|} X_t(x).$$
 We then deduce that
\begin{align*}
    \partial_t \alpha_t^k &= 2\left\langle \nabla^k X_t(x)w^k,\partial_t \nabla^k X_t(x)\right\rangle \\
    &= 2\left\langle \nabla^k X_t(x)w^k, \sum_{\pi \in \Pi_k,|\pi|>1} \nabla^{|\pi|}V(t,X_t(x)) \prod_{B\in \pi} \nabla^{|B|} X_t(x)w^{|B|} + \nabla V(t,X_t(x)) \nabla^k X_t(x)w^k\right\rangle\\
    & \leq 2 \lambda_{1}^k(t)(\alpha_t^k)^{1/2} + 2 \lambda_{2}(t)\alpha_t^k,
\end{align*}
with $$\lambda_{1}^k(t):=\|\sum_{\pi \in \Pi_k,|\pi|>1} \nabla^{|\pi|}V(t,X_t(x)) \prod_{B\in \pi} \nabla^{|B|} X_t(x)w^{|B|} \| $$
and 
$$\lambda_{2}(t):=\lambda_{\max}\left( \nabla V(t,X_t(x))\right).$$

For $k=1$, we have from Gronwall's inequality that
\begin{equation}\label{eq:thegronwallline}
(\alpha_t^1)^{1/2}\leq (\alpha_0^1)^{1/2}\exp\left(\int_0^t \lambda_2(z)dz\right)=\exp\left(\int_0^t \lambda_2(z)dz\right).
\end{equation}
For $k\geq 2$, as $\alpha_0^k=0$ we deduce from the non linear Gronwall's inequality (Theorem 21 in \cite{dragomir2003some}) that 
\begin{equation}\label{eq:thegronwall}
(\alpha_t^k)^{1/2}\leq \int_0^t \lambda_{1}^k(s)\exp\left(\int_s^t \lambda_2(z)dz\right)ds.
\end{equation}

We will first bound $\lambda_2$ which will give us by \eqref{eq:thegronwallline} that $\alpha_t^1$ is bounded. Then, giving a bound on the quantity
$$\sum_{j=2}^k\int_0^1 \|\nabla^j V(s,X_s(x))\|ds,$$
we will finally obtain the bound on $\alpha_t^k$ using the following lemma. 

\begin{lemma}\label{lemma:addlemme} Let $k\in \{1,...,\lfloor \beta \rfloor +1\}$, $A>0$ and $x\in \mathbb{R}^d$ such that
$$
\int_0^1 \lambda_{\max}\left(\nabla V(s,X_s(x)\right)ds+ \mathds{1}_{\{k\geq 2\}}\sum_{j=2}^k\int_0^1 \|\nabla^j V(s,X_s(x))\|ds\leq A.
$$
Then, there exists $C>0$ depending only on $k$ and $A$ such that for all $t\in [0,1)$ and $w\in \mathbb{S}^{d-1}$, 
$$\|\nabla^k X_t(x)w^k\|\leq C.$$
\end{lemma}

\begin{proof}
From \eqref{eq:thegronwallline}, we have that if
$\int_0^1\lambda_{\max}\left( \nabla V(s,X_s(x))\right)ds\leq A,$
then
$\sup_{s\in [0,1]}(\alpha_s^1)^{1/2}\leq \exp(A).$
Let us now treat the case $k\geq 2$ by induction. Suppose that for all $j\in \{1,...,k-1\},$ we have that 
$\sup_{s\in [0,1]}(\alpha_s^j)^{1/2}\leq C$.
Then, from \eqref{eq:thegronwall} we have 
\begin{align*}
    (\alpha_t^k)^{1/2}\leq & \int_0^t \lambda_{1}^k(s)\exp\left(\int_s^t \lambda_2(z)dz\right)ds\nonumber\\
    \leq &  C\int_0^t \lambda_{1}^k(s)ds\nonumber
    \\
    \leq & C \int_0^t \sum_{\pi \in \Pi_k,|\pi|>1} \|\nabla^{|\pi|}V(s,X_s(x))\| \prod_{B\in \pi} (\alpha_s^{|B|})^{1/2} ds\nonumber\\
    \leq & C \sum_{j=2}^k \int_0^1 \|\nabla^j V(s,X_s(x))\|\prod_{i=1}^{k-1} \Big((\alpha_s^{i})^{1/2} \vee 1 \Big)^{k-1}ds\nonumber\\
   \leq & C\sum_{j=2}^k \int_0^1 \|\nabla^j V(s,X_s(x))\|ds\nonumber\\
   \leq & C.
\end{align*}
\end{proof}

Lemma \ref{lemma:addlemme} shows that obtaining a bound on the derivative of the transport map reduces to controlling the quantity $
\int_0^1 \|\nabla^j V(s,X_s(x))\|ds,
$ for all $j\in \{1,...,k\}$ and $x\in \mathbb{R}^d$. The subsequent sections are devoted to obtaining this estimate.

\subsubsection{Derivation of $\nabla^k V(t,x)$ for $k\in \{2,...,\lfloor \beta \rfloor +1\}$}
For $l\in \{0,...,\lfloor \beta \rfloor\}$ and $y\in \mathbb{R}^d$, let us note
$$ f^{l}(y):=\frac{1}{r(y)}\nabla^l r(y),$$
with $r(y)=\frac{p(y)}{\gamma_d(y)}$; and for $t\in [0,1)$, $x\in \mathbb{R}^d$, we note
$$p^{t,x}(y):=\frac{ 1}{Q_tr(x)}\ \varphi^{t,x}(y)r(y),$$
with $Q_tr(x)$ defined in \eqref{eq:Qtr} and $\varphi^{t,x}$ the density of the d-dimensional Gaussian measure with mean $tx$ and covariance $(1-t^2)\text{Id}$. Let us give a first bound on the norm of the differentials of $V$.
\begin{lemma}\label{lemma:lemma1}
    For all $k\in \{2,...,\lfloor \beta \rfloor +1\}$, $t\in [0,1)$ and $x\in \mathbb{R}^d$, we have
    \begin{align*}
    \|\nabla^{k} V(t,x)\|\leq &  C \sum_{l=1}^{k-1} \Biggl( \|\int f^{l}(y)\nabla^2_x p^{t,x}(y)d\lambda^d(y)\| \prod_{j=1}^{k-2}\left(\|\int f^{j}(y) dp^{t,x}(y)\|\vee 1\right)^{k-2}\\
    &+ \|\int f^{l}(y) \nabla_x p^{t,x}(y)d\lambda^d(y)\|\sum_{j=1}^{k-2}\|\int f^{j}(y) \nabla_x p^{t,x}(y)d\lambda^d(y)\|\prod_{i=1}^{k-3}\left(\|\int f^{i}(y)  dp^{t,x}(y)\|\vee 1\right)^{k-3}\Biggl).
\end{align*}
\end{lemma}
\begin{proof}
Recalling that
\begin{align*}
     Q_{t} r(x)&=\int \varphi^{t,x}(y)r(y)d\lambda^d(y)= \int r(tx+\sqrt{1-t^2}z)d\gamma_d(z),
\end{align*}
 we have for all $l\in \{0,...,k-2\}$ that
\begin{align}\label{align:nablalQ}
     \nabla^{l} Q_{t} r(x)=t^l\int \nabla^{l}r(tx+\sqrt{1-t^2}z)d\gamma_d(z).
\end{align}
As $\nabla^{k} V(t,x)= \frac{1}{t}\nabla^{k+1} \log Q_{t} r(x)$ and $r$ has only $\lfloor \beta \rfloor$ derivatives, we first compute $ \nabla^{k-2} V(t,x)$ using \eqref{align:nablalQ} and then we compute the two last derivatives using a change of variable allowing to not differentiate the function $r$ anymore.

Using the Faa di Bruno formula we have

\begin{align*}
    \nabla^{k-2} V(t,x)&=\frac{1}{t}\nabla^{k-1} \log Q_{t} r(x)\\
    & = \frac{1}{t}\sum_{\pi \in \Pi_{k-1}} \frac{(-1)^{|\pi|+1}}{(Q_{t} r(x))^{|\pi|}} \prod_{B\in \pi} \nabla^{|B|} Q_{t} r(x)\\
    & =\frac{1}{t} \sum_{\pi \in \Pi_{k-1}} (-1)^{|\pi|+1} \prod_{B\in \pi} \frac{t^{|B|}}{Q_{t} r(x)} \int \nabla^{|B|}r(\sqrt{1-t^2}z+tx)d\gamma_d(z)\\
    & = t^{k-2}\sum_{\pi \in \Pi_{k-1}} (-1)^{|\pi|+1} \prod_{B\in \pi} \int f^{|B|}(\sqrt{1-t^2}z+tx)\frac{r(\sqrt{1-t^2}z+tx)}{Q_{t} r(x)}d\gamma_d(z)\\
    & = t^{k-2}\sum_{\pi \in \Pi_{k-1}} (-1)^{|\pi|+1} \prod_{B\in \pi} \int f^{|B|}(y)p^{t,x}(y)d\lambda^d(y),
\end{align*}
with
$$ f^{l}(y):=\nabla^l r(y)/r(y)\ \ \ 
\text{ and }\ \ \
p^{t,x}(y):=\frac{ \varphi^{t,x}(y)r(y)}{\int \varphi^{t,x}(z)r(z)d\lambda^d(z)}.$$
Now for any set of functions $(f_i)_{i=1,...,m}\in \mathcal{H}^2(\mathbb{R}^d,\mathbb{R})^m$, we have 
$$\nabla^2 \left(\prod_{i=1}^m f_i\right)= \sum_{i=1}^m \nabla^2 f_i\prod_{j\neq i} f_j+\nabla f_i\sum_{j\neq i} \nabla f_j \prod_{l\neq i,j} f_l,$$
we then deduce that

\begin{align}\label{align:nablakplusdeux}
    \nabla^{k} V(t,x)=&t^{k-2}\sum_{\pi \in \Pi_{k-1}} (-1)^{|\pi|+1} \sum_{B\in \pi}\Biggl( \int f^{|B|}(y)\nabla^2_x dp^{t,x}(y) \prod_{C\in \pi\setminus \{B\}}\int f^{|C|}(y) dp^{t,x}(y)\nonumber\\
    & + \int f^{|B|}(y) \nabla_x dp^{t,x}(y)\sum_{C\in \pi\setminus \{B\}}\int f^{|C|}(y) \nabla_x dp^{t,x}(y)\prod_{D\in \pi\setminus \{B,C\}}\int f^{|D|}(y)  dp^{t,x}(y)\Biggl),
\end{align}
with the notation
$$\int f(y)\nabla^i_xdp^{t,x}(y)=\int f(y)\nabla^i_x p^{t,x}(y)d\lambda^d(y),$$
for $i=0,1,2$. Now, as for $\pi \in \Pi_{k-1}$, $B\in \pi$, $C\in \pi\setminus \{B\}$, $D\in \pi\setminus \{B,C\}$ we have that $|B|\in \{1,...,k-1\}$, $|C|\in \{1,...,k-2\}$ and $|D|\in \{1,...,k-3\}$, from \eqref{align:nablakplusdeux} we deduce that
    \begin{align*}
    \|\nabla^{k} V(t,x)\|\leq &  C \sum_{l=1}^{k-1} \Biggl( \|\int f^{l}(y)\nabla^2_x p^{t,x}(y)d\lambda^d(y)\| \prod_{j=1}^{k-2}\left(\|\int f^{j}(y) dp^{t,x}(y)\|\vee 1\right)^{k-2}\\
    &+ \|\int f^{l}(y) \nabla_x p^{t,x}(y)d\lambda^d(y)\|\sum_{j=1}^{k-2}\|\int f^{j}(y) \nabla_x p^{t,x}(y)d\lambda^d(y)\|\prod_{i=1}^{k-3}\left(\|\int f^{i}(y)  dp^{t,x}(y)\|\vee 1\right)^{k-3}\Biggl).
    \end{align*}

\end{proof}

Let us now derive the different quantities that intervenes in the bound of Lemma \ref{lemma:lemma1}.
Writing 
\begin{equation}\label{eq:noth}
H(y,x):=y-\int z dp^{t,x}(z),
\end{equation}
we have the following result on the integration against $\nabla_x^i p^{t,x}$, for $i\in \{1,2\}$.
\begin{lemma}\label{lemma:valueofgradp}
    For all $f\in L^2(\mathbb{R}^d)$, $t\in [0,1)$ and $x\in \mathbb{R}^d$, we have
    \begin{align*}
    \int f(y)\nabla_x p^{t,x}(y)d\lambda^d(y) =\frac{t}{1-t^2}\int f(y)H(y,x) dp^{t,x}(y)
\end{align*}
and 
\begin{align*}
     \int f(y)\nabla_x^2 p^{t,x}(y)d\lambda^d(y)
    = & \frac{t^2}{(1-t^2)^2}\int f(y)\left(H(y,x)^{\otimes 2} - \int H(w,x)^{\otimes 2} dp^{t,x}(w)\right)dp^{t,x}(y).
\end{align*}
\end{lemma}

\begin{proof} We have
\begin{align*}
    \int f(y) \nabla_x p^{t,x}(y)d\lambda^d(y)& =\int f(y) \nabla_x \left(\frac{ \varphi^{t,x}(y)}{\int \varphi^{t,x}(z)r(z)d\lambda^d(z)}\right)r(y)d\lambda^d(y)\\
    & =\int f(y) \frac{t}{1-t^2}\left(y-tx -\int (w-tx)\frac{ \varphi^{t,x}(w)r(w)d\lambda^d(w)}{\int \varphi^{t,x}(z)r(z)d\lambda^d(z)}\right) \frac{ \varphi^{t,x}(y)r(y)d\lambda^d(y)}{\int \varphi^{t,x}(z)r(z)d\lambda^d(z)}\\
    & =\frac{t}{1-t^2}\int f(y)\left(y-\int wdp^{t,x}(w)\right) dp^{t,x}(y).
\end{align*}
Using this result, we can now compute the second derivative as follow:
\begin{align*}
     \int f(y)&\nabla^2_xp^{t,x}(y)d\lambda^d(y)= \frac{t}{1-t^2}\int f(y)\Biggl(\left(-\int w\nabla_xp^{t,x}(w)d\lambda^d(w)\right)  p^{t,x}(y)\\
     & +\left(y-\int wdp^{t,x}(w)\right) \nabla_x p^{t,x}(y)\Biggl)d\lambda^d(y)\\
     = & \frac{t^2}{(1-t^2)^2}\int f(y)\Biggl(-\int w\otimes \left(w -\int zdp^{t,x}(z))\right)dp^{t,x}(w) p^{t,x}(y)\\
    & +\left(y-\int wdp^{t,x}(w)\right)^{\otimes 2} p^{t,x}(y)\Biggl)d\lambda^d(y)\\
    = & \frac{t^2}{(1-t^2)^2}\int f(y)\Biggl(\left(y -\int wdp^{t,x}(w)\right)^{\otimes 2} - \int \left(w -\int zdp^{t,x}(z)\right)^{\otimes 2} dp^{t,x}(w)\Biggl)dp^{t,x}(y).
\end{align*}
\end{proof}

Having derived the different terms appearing in the bound on $\|\nabla^{k} V(t,x)\|$ of Lemma \ref{lemma:lemma1}, we control them differently for the two settings. The distinction arises from the fact that in Theorem \ref{theo:thetheo} the function $r$ belongs to $\mathcal{H}_K^\beta$, whereas in Theorem \ref{theo:thetheo2}, its derivatives diverge near the boundary of the ball. As a result, the estimates require prior control of the localization of $X_t(x)$.

\subsection{Proof of the bound on  $\|\nabla X_1\|_{\mathcal{H}^{\lfloor \beta \rfloor}}$ for Theorem \ref{theo:thetheo}}
In this section we focus on  a probability density $p$ satisfying the assumptions of Theorem \ref{theo:thetheo}. The case corresponding to Theorem \ref{theo:thetheo2} is addressed in Appendix Section \ref{sec:prooftheo2}. Using the results of Section \ref{sec:generalreas}, we give a bound on $\int_0^1 \|\nabla^j V(t,\cdot)\|_\infty dt$ for all $j\in \{1,...,k\}$. The case $j=1$ gives in particular the well-definedness of the flow $(X_t)_{t\in [0,1]}$. Then, using Lemma \ref{lemma:addlemme} we obtain the estimate on $\|\nabla^k X_t(x)\|$ for all  $k\in \{1,...,\lfloor \beta \rfloor +1\}$ and $t\in[0,1]$.

\subsubsection{Concentration of the mass under $p^{t,x}$}
To address the term of order $(1-t^2)^{-1}$ appearing in the estimates of Lemma \ref{lemma:valueofgradp}, we first assess the mass concentration of the probability $p^{t,x}$. Using that for all $t\in [0,1)$, $x\in \mathbb{R}^d$, the probability $p^{t,x}$ is bounded by $\exp(2K)\varphi^{t,x}$, we obtain in the next result that it concentrates the mass within a radius $(1-t^2)^{1/2}$ around the point $tx$.
\begin{proposition}\label{prop:concentHtheo1} For all $x\in \mathbb{R}^d$, $t\in [0,1)$ and $i,j\in \{1,...,d\}$
    we have
    $$\int H(y,x)_i^2 dp^{t,x}(y)\leq C(1-t^2)$$
    and 
    $$\int \left(H(y,x)^{\otimes2}_{ij}- \int H(z,x)^{\otimes2}_{ij}dp^{t,x}(z)\right)^2dp^{t,x}(y)\leq C(1-t^2)^{2}.$$
\end{proposition}
\begin{proof}
Let $Y,\overline{Y}$ be two independent random variables of law $p^{t,x}$ and $Z,\overline{Z}$ two independent random variables of law $\varphi^{t,x}$. As for all $y\in \mathbb{R}^d$ we have $$p^{t,x}(y)=\frac{r(y)}{Q_tr(x)}\varphi^{t,x}(y)\leq C\varphi^{t,x}(y),$$ then
\begin{align*}
    \int H(y,x)_i^2 dp^{t,x}(y) = &\int \left(y_i-\int y_idp^{t,x}(z)\right)^2  dp^{t,x}(y)\\
    =  &\text{Var}_{p^{t,x}}[Y_i]\\
    = &\frac{1}{2} \mathbb{E}_{p^{t,x}}[(Y_i-\overline{Y}_i)^2]\\
    \leq & C \mathbb{E}_{\varphi^{t,x}}[(Z_i-\overline{Z}_i)^2]
    \leq  C\text{Var}_{\varphi^{t,x}}(Z_i)\\
    \leq & C(1-t^2).
\end{align*}
Likewise, using Cauchy-Schwarz inequality we get
\begin{align*}
    \int \left(H(y,x)^{\otimes2}_{ij}- \int H(z,x)^{\otimes2}_{ij}dp^{t,x}(z)\right)^2dp^{t,x}(y)
  &  =  \text{Var}_{p^{t,x}}[(Y_i-\mathbb{E}_{p^{t,x}}(Y_i))(Y_j-\mathbb{E}_{p^{t,x}}(Y_j))]\\
    &\leq \mathbb{E}_{p^{t,x}}[(Y_i-\mathbb{E}_{p^{t,x}}(Y_i))^2(Y_j-\mathbb{E}_{p^{t,x}}(Y_j))^2]
    \\
    &\leq \mathbb{E}_{p^{t,x}}[(Y_i-\mathbb{E}_{p^{t,x}}(Y_i))^4]^{1/2}\mathbb{E}_{p^{t,x}}[(Y_j-\mathbb{E}_{p^{t,x}}(Y_j))^4]^{1/2}
\end{align*}
and 
\begin{align*}
    \mathbb{E}_{p^{t,x}}[(Y_i-\mathbb{E}_{p^{t,x}}(Y_i))^4]\leq & \mathbb{E}_{p^{t,x}}[(Y_i-\overline{Y}_i)^4]\\
    \leq &  C\mathbb{E}_{\varphi^{t,x}}[(Z_i-\overline{Z}_i)^4]\\
    \leq & C (1-t^2)^2,
\end{align*}
which gives the result.
\end{proof}

Using Proposition \ref{prop:concentHtheo1} we obtain a bound on the different terms appearing in the estimate given by Lemma \ref{lemma:lemma1}.
\begin{proposition}\label{prop:firstboundongradsth1}
    For all $f\in L^2(\mathbb{R}^d)$, $i\in \{1,2\}$, $t\in[ 0,1)$ and $x,\xi\in \mathbb{R}^d$, we have
    \begin{align*}
    \|\int f(y) \nabla_x^i p^{t,x}(y)d\lambda^d(y)\| 
    &\leq \frac{Ct^i}{(1-t^2)^{i/2}}\left(\int \|f(y)-\xi\|^2dp^{t,x}(y)\right)^{1/2}.
\end{align*}
\end{proposition}

\begin{proof}
We first take care of the term $\|\int f(y) \nabla_x^2 p^{t,x}(y)d\lambda^d(y)\|$. Using Lemma \ref{lemma:valueofgradp} and the Cauchy-Schwarz inequality we obtain 
\begin{align*}
    \|& \int f(y)\nabla^2_x p^{t,x}(y)d\lambda^d(y)\| \\
    =&\|  \frac{t^2}{(1-t^2)^2}\int \left(f(y)-\xi\right)\left(H(y,x)^{\otimes 2} - \int H(w,x)^{\otimes 2} dp^{t,x}(w)\right)dp^{t,x}(y)\|\\
    \leq &\frac{t^2}{(1-t^2)^2} \left(\int \|f(y)-\xi\|^2dp^{t,x}(y)\right)^{1/2}\sum_{ij} \Biggl(\int\left(H(y,x)^{\otimes 2}_{ij} - \int H(w,x)^{\otimes 2}_{ij} dp^{t,x}(w)\right)^2dp^{t,x}(y)\Biggl)^{1/2}.
\end{align*}
    Then, from Proposition \ref{prop:concentHtheo1} we deduce that 
\begin{align*}
    \|\int f(y) \nabla_x^2 p^{t,x}(y)d\lambda^d(y)\| 
    &\leq \frac{Ct^2}{(1-t^2)}\left(\int \|f(y)-\xi\|^2dp^{t,x}(y)\right)^{1/2}.
\end{align*}
Similarly, we have
\begin{align*}
    \|\int & f(y) \nabla_x dp^{t,x}(y)\| = \frac{t}{1-t^2}\|\int (f(y)-\xi)\left(y-\int wdp^{t,x}(w)\right) dp^{t,x}(y)\|\\
    &\leq \frac{t}{1-t^2}\left(\int \|f(y)-\xi\|^2dp^{t,x}(y)\right)^{1/2}\sum_i \left( \int\left(y_i-\int w_idp^{t,x}(w)\right)^2 dp^{t,x}(y)\right)^{1/2}
    \\
    &\leq \frac{C}{(1-t^2)^{1/2}}\left(\int \|f(y)-\xi\|^2dp^{t,x}(y)\right)^{1/2}.
\end{align*}
\end{proof}
Proposition \ref{prop:firstboundongradsth1} highlights the fact that the integration of a function against the $i$-th derivatives of $p^{t,x}$ grows at most as $(1-t^2)^{-i/2}$. It additionally shows that it is proportional to the variance of the function with respect to $p^{t,x}$. With the estimates from Proposition \ref{prop:firstboundongradsth1} in hand, we are now ready to bound the derivatives of the transport map.

\subsubsection{Well-definedness of the flow and bound on  $\|\nabla X_1\|_\infty$}
The estimate on the Jacobian of $X_1$ differs from that of the higher-order derivatives as it does not involve the derivatives of the density $p$. To apply Lemma \ref{lemma:addlemme}, we first demonstrate using the fact that  $a\in \mathcal{H}^{\beta,\theta}_K$, that the quantity$ \int_0^t \|\nabla V(s,\cdot)\|_{\infty}ds$ remains bounded as $t$ approaches~1.
\begin{proposition}\label{prop:boundnabalvgauss}
For all $t\in [0,1)$ and $x\in \mathbb{R}^d$,  we have 
\begin{align*}
    \|\nabla V(t,x)\|\leq C t(1-t^2)^{(\beta-\lfloor \beta \rfloor)/2-1}\log\Big(\big((1-t^2)\wedge 1/2\big)^{-1}\Big)^{-\theta},
\end{align*}
which gives in particular that
    $$\int_0^1 \|\nabla V(t,\cdot)\|_{\infty}dt<\infty.$$
\end{proposition}
\begin{proof}
Let $x\in \mathbb{R}^d$, we have
\begin{align}\label{align:decompnabla1}
    \|\nabla V(t,x)\| = &\frac{1}{t}\|\nabla^2 \log(Q_t r)(x)\|\nonumber\\
     =&\|\frac{t}{(1-t^2)^2}\int (y-tx)H(y,x) dp^{t,x}(y)-\frac{t}{1-t^2}\text{Id}\|\nonumber\\
         \leq & \frac{t}{(1-t^2)^2}\|\int(y-tx)^{\otimes2}d p^{t,x}(y)-(1-t^2)\text{Id}\|\nonumber\\
         &+\frac{t}{(1-t^2)^2}\|\int(y-tx)\left(tx-\int zdp^{t,x}(z)\right)d p^{t,x}(y)\|\nonumber\\
         \leq & \frac{t}{(1-t^2)^2}\|\int\left(y-tx\right)^{\otimes2}\left( \frac{r(y)}{\int r(z)\varphi^{t,x}(z)d\lambda^d(z)}-1\right)\varphi^{t,x}(y)d\lambda^d(y)\|\nonumber\\
         &+\frac{t}{(1-t^2)^2}\|\int(y-tx)d p^{t,x}(y)\|^2.
\end{align}
As $r=e^a\in \mathcal{H}^{\beta,\theta}_C \subset \mathcal{H}^{\beta-\lfloor \beta \rfloor,\theta}_C$, we have 
\begin{align*}
\|\int(y-tx)d p^{t,x}(y)\|^2 & \leq C\|\int(y-tx)(r(y)-r(tx))\varphi^{t,x}(y)d\lambda^d(y)\|^2\\
     \leq &C \left(\int\|y-tx\|^{1+\beta-\lfloor \beta \rfloor}\log((\|y-tx\|\wedge 1/2)^{-1})^{-\theta}\varphi^{t,x}(y)d\lambda^d(y)\right)^2\\
        \leq &C \left(\int\|y-tx\|^{1+\beta-\lfloor \beta \rfloor}\log((\|y-tx\|\wedge 1/2)^{-1})^{-\theta}\varphi^{t,x}(y)\mathds{1}_{\{\|y-tx\|\leq (1-t^2)^{1/4}\}}d\lambda^d(y)\right)^2\\
                &+C \left(\int\|y-tx\|^{1+\beta-\lfloor \beta \rfloor}\log((\|y-tx\|\wedge 1/2)^{-1})^{-\theta}\varphi^{t,x}(y)\mathds{1}_{\{\|y-tx\|> (1-t^2)^{1/4}\}}d\lambda^d(y)\right)^2\\
    \leq &C \log\Big(\big((1-t^2)^{1/4}\wedge 1/2\big)^{-1}\Big)^{-2\theta}\left(\int\|y-tx\|^{1+\beta-\lfloor \beta \rfloor}\varphi^{t,x}(y)d\lambda^d(y)\right)^2\\
    &+C\exp\left(-\frac{(1-t^2)^{-1/4}}{2}\right)\\
     \leq & C(1-t^2)^{1+\beta-\lfloor \beta \rfloor}\log\Big(\big((1-t^2)\wedge 1/2\big)^{-1}\Big)^{-2\theta}.
\end{align*}
On the other hand, we have 
\begin{align*}
\|&\int\left(y-tx\right)^{\otimes2}\left( \frac{r(y)}{\int r(z)\varphi^{t,x}(z)d\lambda^d(z)}-1\right)\varphi^{t,x}(y)d\lambda^d(y)\|\\
= & \frac{1}{Q_tr(x)}\|\int\left(y-tx\right)^{\otimes2}\left( r(y)-r(tx)+r(tx)-\int r(z)\varphi^{t,x}(z)d\lambda^d(z)\right)\varphi^{t,x}(y)d\lambda^d(y)\|\\
      \leq & C \|\int\left(y-tx\right)^{\otimes2}\left( r(y)-r(tx)\right)\varphi^{t,x}(y)d\lambda^d(y)\|\nonumber\\
      & + C \| \int (r(tx)-r(z))\varphi^{t,x}(z)d\lambda^d(z)\| \|\int\left(y-tx\right)^{\otimes2}\varphi^{t,x}(y)d\lambda^d(y)\|
      \nonumber\\
    \leq & C \|\int\|y-tx\|^{2+\beta-\lfloor \beta \rfloor}\log((\|y-tx\|\wedge 1/2)^{-1})^{-\theta}\varphi^{t,x}(y)d\lambda^d(y)\|\\
    & +C\| \int \|x-y\|^{\beta-\lfloor \beta \rfloor}\log((\|y-tx\|\wedge 1/2)^{-1})^{-\theta}\varphi^{t,x}(z)d\lambda^d(z)\| \|\int\|y-tx\|^2\varphi^{t,x}(y)d\lambda^d(y)\|\\
    \leq & C(1-t^2)^{1+(\beta-\lfloor \beta \rfloor)/2}\log\Big(\big((1-t^2)\wedge 1/2\big)^{-1}\Big)^{-\theta}.
\end{align*}
We then conclude that
\begin{align*}
    \|\nabla V(t,x)\|\leq C \frac{t}{(1-t^2)^{1-(\beta-\lfloor \beta \rfloor)/2}}\log\Big(\big((1-t^2)\wedge 1/2\big)^{-1}\Big)^{-\theta},
\end{align*}
which gives the result as $\theta>1$.
\end{proof}
 Using Proposition \ref{prop:boundnabalvgauss} we obtain that the velocity $V$ satisfies the assumptions of Proposition \ref{prop:condiwelldifned}, which gives well-definedness of the flow of transport maps. We conclude in the next proposition that the Langevin transport map $T:=X_1$ satisfies that $\nabla T \in \mathcal{H}^0_C(\mathbb{R}^d)$.
\begin{proposition}
    Under the assumptions of Theorem \ref{theo:thetheo}, there exists a unique solution for all $t\in [0,1]$ to the equation
    $$\partial_t X_t(x) =V(t,X_t(x)), \quad X_0(x)=x.$$
    Furthermore, the solution at time $t=1$ verifies that
    $$\nabla X_1 \in \mathcal{H}^0_C.$$
\end{proposition}
\begin{proof}
    To apply Proposition \ref{prop:condiwelldifned} we need to bound the infinite norm of $V_t$. For $t\in[0,1]$ and $x\in \mathbb{R}^d$, we have
    \begin{align*}
        \| V(t,x)\| = &\frac{1}{t}\|\nabla \log(Q_t r)(x)\|\\
     =&\|\frac{1}{1-t^2}\int (y-tx)dp^{t,x}(y)\|
     \\
     \leq &\frac{C}{1-t^2}\int \|y-tx\|\varphi^{t,x}(y)d\lambda^d(y)\\
    \leq & C(1-t^2)^{-1/2}.
    \end{align*}
    Putting together this estimate with the one given by Proposition \ref{prop:boundnabalvgauss}, we get the well-definedness of the flow. Finally, combining Lemma \ref{lemma:addlemme} and Proposition \ref{prop:boundnabalvgauss} we obtain that for all $x\in \mathbb{R}^d$, $\|\nabla X_1(x)\|\leq C$. 
\end{proof}

\subsubsection{Bounding $\|\nabla^k X_1\|_\infty$}
In this section we show that for all $x\in \mathbb{R}^d$ and $k\in\{2,...,\lfloor \beta \rfloor +1\}$, we have
$$\int_0^1\|\nabla^{k} V(t,X_t(x))\|dt\leq  C,$$
which gives by Lemma \ref{lemma:addlemme} and Proposition \ref{prop:boundnabalvgauss} that $\|\nabla^k X_1\|_\infty$ is bounded. In order to use the estimate of Proposition \ref{prop:firstboundongradsth1}, let us first give a result on  the concentration of the mass under $p^{t,x}$.
\begin{proposition}\label{prop:concentration1}
Let $t\in [0,1)$ and $x\in \mathbb{R}^d$, then for all $f\in \mathcal{H}_1^{\beta-\lfloor \beta \rfloor,\theta}(\mathbb{R}^d)$  we have
    \begin{align*}
    \int \|f(y)-f(tx)\|^2 dp^{t,x}(y) \leq C (1-t^2)^{\beta-\lfloor \beta \rfloor}\log\Big(\big((1-t^2)\wedge 1/2\big)^{-1}\Big)^{-2\theta}.
\end{align*}
\end{proposition}
\begin{proof}
 The result follows from the fact that for all $y\in \mathbb{R}^d$ we have
$p^{t,x}(y)\leq C \varphi^{t,x}(y).$
\end{proof}

Next, we control the terms in the bound provided by Lemma \ref{lemma:lemma1}. Fix $l\in \{1,...,\lfloor \beta \rfloor\}$, $t\in [0,1)$ and $x\in \mathbb{R}^d$. Taking $\xi=f^l(tX_t(x))=\frac{1}{r(tX_t(x))}\nabla^l r(tX_t(x))$ in  Proposition 
\ref{prop:firstboundongradsth1}, we obtain 
$$ 
\|\int f^{l}(y)\nabla^2_x dp^{t,X_t(x)}(y) \|
    \leq \frac{C}{1-t^2} \left(\int \|f^l(y)-f^l(tX_t(x))\|^2dp^{t,X_t(x)}(y)\right)^{1/2}.
    $$
    
On the other hand,  we also have from Proposition \ref{prop:firstboundongradsth1} that for $m\in \{1,...,\lfloor \beta \rfloor-1\}$,
\begin{align*}
\|&\int f^{l}(y)\nabla_x p^{t,X_t(x)}(y)d\lambda^d(y)\| \|\int f^{m}(y)\nabla_x p^{t,X_t(x)}(y)d\lambda^d(y) \|\\
    &\leq \frac{C}{1-t^2} \left(\int \|f^l(y)-f^l(tX_t(x))\|^2dp^{t,X_t(x)}(y)\right)^{1/2}\left(\int \|f^m(y)-f^m(tX_t(x))\|^2dp^{t,X_t(x)}(y)\right)^{1/2}.
    \end{align*}

Finally, from Proposition \ref{prop:concentration1}  we have
\begin{align}\label{eq:enfaitealignmaistauv}
    \int \|f^{l}(y)-f^{l}(tX_t(x))\|^2 dp^{t,X_t(x)}(y) & \leq C (1-t^2)^{\beta-\lfloor \beta \rfloor}\log\Big(\big((1-t^2)\wedge 1/2\big)^{-1}\Big)^{-2\theta}.
\end{align}

Putting everything together in Lemma \ref{lemma:lemma1}, we obtain 

\begin{align}\label{align:conclubanblakbound1}
     \|\nabla^{k} V(t,X_t(x))\|& \leq  C(1-t^2)^{(\beta-\lfloor \beta \rfloor)/2-1}\log\Big(\big((1-t^2)\wedge 1/2\big)^{-1}\Big)^{-\theta}.
\end{align}
As $\theta>1$, we deduce that 
$$\int_0^1\|\nabla^{k} V(s,X_t(x))\|dt\leq  C,$$
so from Proposition \ref{prop:boundnabalvgauss} and Lemma \ref{lemma:addlemme}  we obtain the bound on $\|\nabla^k X_1\|_\infty$ for $k\in\{2,...,\lfloor \beta \rfloor +1\}$.   We can then conclude that the derivatives of the transport map $T:= X_1$ are bounded.

\begin{proposition}
    Under the assumptions of Theorem \ref{theo:thetheo}, the Langevin transport map $T$ satisfies 
    $\nabla^k T \in \mathcal{H}^0_C$ for all $k\in\{1,...,\lfloor \beta \rfloor +1\}$.
\end{proposition}

\section{Applications}\label{sec:app}
In this section we present different applications of Theorem \ref{theo:thetheo}. The first application involves extending functional inequalities, a classical use of (Lipschitz) regularity theory for transport maps. The second one is the proof of the optimality of the widely used GAN estimator in the context of density estimation.

\subsection{Generalized log-Sobolev inequality}
The theory of Lipschitz transport maps has been highly motivated by their ability to prove functional inequalities, by transferring them from
the source measure to the target measure. Having obtained in Theorem \ref{theo:thetheo} a transport map with higher order regularity, we can  naturally extend Gaussian inequalities involving high order derivatives. One of such inequality is the Generalized log-Sobolev inequality.
\begin{theorem}[Generalized log-Sobolev inequality \cite{rosen1976sobolev}]\label{theo:wsobo} For $\phi\in C^2(\mathbb{R}^d)$ satisfying $\phi(x)=a\|x\|^{1+s}$ for $\|x\|\geq K,$ with $ K,a,s>0$, define the measure $\mu$ having a density $d\mu(x) = \exp(-\phi(x))dx$ with respect to the $d$-dimensional Lebesgue measure. Then, for all $k\in \mathbb{N}_{>0}$ and $f:\mathbb{R}^d\rightarrow \mathbb{R}$ with $\|f\|_{L_{\mu}^2}=1$, we have
\begin{align*}
    \int |f(x)|^2|\log(|f(x)|)|^{2sk/(s+1)}d\mu(x)\leq C \sum_{|\alpha|=0}^k\|D^\alpha f\|^2_{L_{\mu}^2}.
\end{align*}
    
\end{theorem}
This inequality extends the classical log-Sobolev inequalities, which are essential tools in analysis and PDE theory, by incorporating logarithmic weights that reflect the growth conditions imposed by the function $\phi$. Using the Föllmer transport map, we can extend this inequality to measure $\mu$ satisfying the assumptions of Theorem \ref{theo:thetheo}.

\begin{corollary}\label{coro:logsogen}
    For $a\in\mathcal{H}^{\beta,\theta}_K(\mathbb{R}^d,\mathbb{R})$ with $\beta\geq 0$, $\theta>1$ and $K>0$, define the probability measure $\mu$ having a density $d\mu(x) = \exp(-\|x\|^2/2+a(x)+c_a)dx$ with respect to the $d$-dimensional Lebesgue measure (with $c_a\in \mathbb{R}$ a normalizing constant). Then, for all $k\in \{1,...,\lfloor \beta \rfloor +1\}$, and $f:\mathbb{R}^d\rightarrow \mathbb{R}$ with $\|f\|_{L_{\mu}^2}=1$ we have
\begin{align*}
    \int |f(x)|^2|\log(|f(x)|)|^{k}d\mu\leq C \sum_{|\alpha|=0}^k\|D^\alpha f\|^2_{L_{\mu}^2}.
\end{align*}
\begin{proof}
    From Theorem \ref{theo:wsobo} with $\phi(x)=\frac{\|x\|^2}{2}$, we have for all $g:\mathbb{R}^d\rightarrow \mathbb{R}$ with $\|g\|_{L_2(\gamma_d)}=1$
    \begin{align*}
    \int |g(x)|^2|\log(|g(x)|)|^{k}d\gamma_d(x)\leq C \sum_{|\alpha|=0}^k\|D^\alpha g\|^2_{L^2_{\gamma_d}}.
\end{align*}
Then, from Theorem \ref{theo:thetheo}, there exists a map $T\in C^{\lfloor \beta \rfloor+1}(\mathbb{R}^d,\mathbb{R}^d)$ with $\nabla T\in  \mathcal{H}^{\beta}_C \black (\mathbb{R}^d,\mathbb{R}^d)$ such that
\begin{align*}
    \int |f(x)|^2|\log(|f(x)|)|^{k}d\mu(x)= &  \int |f\circ T(x)|^2|\log(|f\circ T(x)|)|^{k}d\gamma_d(x)\\
    \leq &C \sum_{|\alpha|=0}^k\|D^\alpha (f\circ T)\|^2_{L^2_{\gamma_d}}
    \\
    = &C \sum_{|\alpha|=0}^k\|D^\alpha f\|^2_{L_{\mu}^2}.
\end{align*}
\end{proof}
\end{corollary}

The Generalized log-Sobolev inequality established by \cite{rosen1976sobolev} finds many applications in the theory of semigroups and  Hypercontractivity \citep{simon1982schrodinger,davies1984ultracontractivity,fabes1993logarithmic}.

\subsection{Optimal density estimator on non-compact domain via GANs}
Suppose that we observe an i.i.d.~sample $X_1,...X_n\in \mathbb{R}^d$ from  a probability measure having a density $p^\star:\mathbb{R}^d\rightarrow \mathbb{R}$. The problem of
estimating the density from the data was extensively studied through various methods such as kernel \citep{mcdonald2017minimax},
$k$-nearest neighbors \citep{dasgupta2014optimal} and wavelets
\citep{donoho1996density}.

With increasing data dimensionality, recent years have seen the rise of Generative Adversarial Networks (GANs) \citep{goodfellow2014generative}, which enable both estimation and sampling from the target density. GANs have achieved remarkable results in fields such as image generation \citep{karras2021}, video synthesis \citep{vondrick2016generating}, and text generation \citep{SeqGANs}. A GAN comprises two main components: a class of generator functions, \( \mathcal{G} \), and a class of discriminators, \( \mathcal{D} \). Given an easily sampled distribution \( \nu \) defined on a latent space \( \mathcal{Z} \), the generator \( g \in \mathcal{G} \), where \( g: \mathcal{Z} \rightarrow \mathbb{R}^p \), aims to approximate the target density \( p^\star \) by minimizing an Integral Probability Metric (IPM) \citep{IPMsMuller} over \( \mathcal{G} \):
\begin{equation}\label{eq:IPM}
  d_{\mathcal{D}}(p^\star, g_{\# }\nu) := \sup_{D \in \mathcal{D}} \left( \mathbb{E}_{p^\star}[D(X)] - \mathbb{E}_\nu[D(g(Z))] \right).  
\end{equation}

The goal of each discriminator \( D \in \mathcal{D} \), where \( D: \mathbb{R}^p \rightarrow \mathbb{R} \), is to distinguish between the true distribution \( p^\star \) and the generated distribution \( g_{\# }\nu \). The efficiency of the estimator is often assessed using Hölder IPMs, defined as
\[
d_{\mathcal{H}_1^\gamma}(p^\star, g_{\# }\nu) := \sup_{D \in \mathcal{H}_1^\gamma} \left( \mathbb{E}_{p^\star}[D(X)] - \mathbb{E}_\nu[D(g(Z))] \right),
\]
which, when \( \gamma = 1 \) and the measures are supported within a prescribed ball, is equivalent to the Wasserstein distance.

The GAN methods has been proven to attain theoretical optimal rates of convergence in various settings \citep{belomestny2023rates, liang2021generative,stephanovitch2023wasserstein}. However none of these settings treat the case where the latent space $(\mathcal{Z},\nu)$ is Gaussian, whereas it is intensively used in practice \citep{goodfellow2014generative,biau2020some,wang2017generative}. The existence of smooth transport map given by Theorem \ref{theo:thetheo} finally allows to prove that the GAN attains optimal rates in the Gaussian setting.

\subsubsection{Assumptions on the target density}
We suppose that we observe an i.i.d.~sample $X_1,...X_n\in \mathbb{R}^d$ from a probability measure having a density $p^\star$ satisfying the following assumption.

\begin{assumption}\label{assump:gandens}
    The target density $p^\star$ belongs to
    $$\mathcal{F}_\beta:=\left\{p:\mathbb{R}^d\rightarrow \mathbb{R}\ \Big|\ p(x)=\gamma_d(\Sigma^{-1/2}x)\exp\left(a(x)\right),\Sigma \in \mathbb{M}, \int p=1, a \in \mathcal{H}^{\beta,2}_K(\mathbb{R}^d,\mathbb{R})\right\},$$
    with $\mathbb{M}$ the set of  symmetric  matrices $M$ satisfying $K^{-1}\text{Id}\preceq M\preceq K\text{Id}$.
\end{assumption}
This setting is a natural extension to unbounded domains of the classical compact setting. In the classical density estimation framework (see Section 2 of \cite{Tsybakov2008IntroductionTN}), data are assumed to be sampled from a density $p\in \mathcal{H}^{\beta}_K(\mathbb{K},\mathbb{R})$ with $\mathbb{K}\subset \mathbb{R}^d$ compact and $p$ bounded above and below by $K$ and $K^{-1}$, respectively. In this context, the density can be represented as$$f(x)=\lambda^{d}_{|\mathbb{K}}(x)\exp(a(x)),$$ with $\lambda^{d}_{|\mathbb{K}}$ the $d$-dimensional Lebesgue measure on $\mathbb{K}$ and $a\in \mathcal{H}^{\beta}_K(\mathbb{K},\mathbb{R})$. A natural extension to the non-compact case is to assume that  $p$ takes the form \begin{equation*}
p(x):=\exp\left(-\frac{\|\Sigma^{-1/2}x\|^2}{2}+a(x)\right),
\end{equation*}
with $a \in \mathcal{H}^{\beta}_K(\mathbb{R}^d,\mathbb{R})$ and $\Sigma \in \mathbb{M}$. In this unbounded setting, the Gaussian distributions replaces the Lebesgue measure $\lambda^{d}_{|\mathbb{K}}$ used in the compact case. Note that the additional weak regularity $\mathcal{H}^{0,2}_K(\mathbb{R}^d,\mathbb{R})$ supposed in Assumption \ref{assump:gandens} only impacts the rate of estimation up to a $\log(n)$ factor.  

\subsubsection{The GAN estimator}

Let us properly define the GAN estimator. For a compact class of generators $\mathcal{G}\subset L^2(\mathbb{R}^d,\mathbb{R}^d)$ and a class of discriminators $\mathcal{D}\subset L^2(\mathbb{R}^d,\mathbb{R})$, the GAN estimator is defined as
\begin{equation}\label{eq:gan}
    \hat{g}\in \argmin_{g\in \mathcal{G}}\ d_{ \mathcal{D}} (p^\star_n,g_{\# }\nu),
\end{equation}
with $d_{ \mathcal{D}}$ defined in \eqref{eq:IPM}, $p^\star_n:=\frac{1}{n}\sum_{i=1}^n \delta_{X_i}$ the empirical measure from the data and $\nu:=\gamma_d$ the $d$-dimensional Gaussian measure.

Let $L>0$ be the constant given by Corollary \ref{theo:diffeo} such that the transport map $T$ from the $d$-dimensional Gaussian measure to any probability measure within $\mathcal{F}_\beta$ satisfies  $\nabla T^{-1}\in  \mathcal{H}^{\beta}_L   (\mathbb{R}^d,\mathbb{R}^d)$. Fix $n\in \mathbb{N}_{>0}$ the number of data and take as the generator class
\begin{equation}\label{eq:G}
\mathcal{G}:= \left\{g \in \mathcal{H}^{\beta+1}_{C}(\mathbb{R}^d,\mathbb{R}^d)\Big|\ \inf_{x\in \mathbb{R}^{d}}\min_{r\in \mathbb{S}^{d-1}} \|\nabla g(x)r \|\geq L^{-1}\right\}
\end{equation}
 and as the discriminator class
\begin{equation}\label{eq:D}
\mathcal{D}:=\mathcal{H}^{d/2}_{1}(\mathbb{R}^d,\mathbb{R}).
\end{equation}
Using these classes of functions, we obtain in the following theorem that the GAN estimator attains optimal rates of convergence (writing $g(n)=\tilde{O}( f(n))$ for $g(n)\leq C \log(n)^{C_2} f(n)$).

\begin{theorem}\label{theo:minimaxgan}
 Let $n\in \mathbb{N}_{>0}$, $\beta> 0$, $p^\star \in \mathcal{F}_\beta$ and $(X_1,...,X_n)$ an i.i.d. sample of law $p^\star$. The GAN estimator $\hat{g}$ \eqref{eq:gan} using the classes of functions $\mathcal{G}$ \eqref{eq:G} and $\mathcal{D}$ \eqref{eq:D}, satisfies for all $\gamma>0$
\begin{align*}
    \sup_{p^\star \in \mathcal{F}_\beta}\ \mathbb{E}_{X_i}[d_{\mathcal{H}^{\gamma}_1}(\hat{g}_{\# }\gamma_d,p^\star)] =\tilde{O}\left( \inf_{\hat{\theta}\in \Theta}\ \sup_{p^\star\in \mathcal{F}_\beta}\ \mathbb{E}_{X_i}[d_{\mathcal{H}^{\gamma}_1}(\hat{\theta},p^\star)]\right),
\end{align*}
where $\Theta$ denotes the set of all possible estimators  of $p^\star$ based on n sample.
\end{theorem}

\begin{proof}
  Let us take $p^\star\in \mathcal{F}_\beta$ and $T:\mathbb{R}^d\rightarrow \mathbb{R}^d$ the transport map given by Corollary \ref{theo:diffeo} such that $T_{\# }\gamma_d=p^\star$. Let us first bound the expected error $\mathbb{E}_{X_i\sim p^\star}[d_{\mathcal{H}^{d/2}_1}(\hat{g}_{\# }\gamma_d,p^\star)]$ of our estimator. To this end, write $\hat{\mu}:=\hat{g}_{\# }\gamma_d$ and  define for all $\mu \in \mathcal{F}_\beta$
$$h_\mu \in \argmax_{h\in \mathcal{H}^{d/2}_1} \int h(x)d\mu(x)-\int h(x)dp^\star(x),$$
an optimal potential between $\mu$ and $p^\star$ and define
$$h_\mu^n \in \argmax_{h\in \mathcal{H}^{d/2}_1} \int h(x)d\mu(x)-\frac{1}{n}\sum_{i=1}^n h(X_i),$$
an optimal potential between $\mu$ and $p^\star_n:= \frac{1}{n}\sum_{i=1}^n \delta_{X_i}$. We have
\begin{align*}
    \mathbb{E}[d_{\mathcal{H}^{d/2}_1}(\hat{\mu},p^\star)] =& \mathbb{E}[\int h_{\hat{\mu}}(x)d\hat{\mu}(x)-\int h_{\hat{\mu}}(x)dp^\star(x)]\\
     = & \mathbb{E}[\int h_{\hat{\mu}}(x)d\hat{\mu}(x)-\frac{1}{n}\sum_{i=1}^n h_{\hat{\mu}}(X_i)+\frac{1}{n}\sum_{i=1}^n h_{\hat{\mu}}(X_i)-\int h_{\hat{\mu}}(x)dp^\star(x)]\\
     \leq & \mathbb{E}[\int h_{\hat{\mu}}^n(x)d\hat{\mu}(x)-\frac{1}{n}\sum_{i=1}^n h_{\hat{\mu}}^n(X_i)]+\mathbb{E}[d_{\mathcal{H}^{d/2}_1}(p^\star_n,p^\star)].
\end{align*}
Furthermore, by definition of $\hat{\mu}$ we have 
\begin{align*}
    \mathbb{E}[\int h_{\hat{\mu}}^n(x)d\hat{\mu}(x)-\frac{1}{n}\sum_{i=1}^n h_{\hat{\mu}}^n(X_i)]& \leq \mathbb{E}[\int h_{T_{\# }\gamma_d}^n(x)dT_{\# }\gamma_d(x)-\frac{1}{n}\sum_{i=1}^n h_{T_{\# }\gamma_d}^n(X_i)] \\
        & \leq \mathbb{E}[\int h_{T_{\# }\gamma_d}^n(x)d p^\star(x)-\frac{1}{n}\sum_{i=1}^n h_{T_{\# }\gamma_d}^n(X_i)] \\
    & \leq \mathbb{E}[\int h_{p ^\star}^n(x)dp^\star(x)-\frac{1}{n}\sum_{i=1}^n h_{p^\star}^n(X_i)]\\
    & = \mathbb{E}[d_{\mathcal{H}^{d/2}_1}(p^\star_n,p^\star)].
\end{align*}
Then, using Lemma 13 from \cite{tang2023minimax}
 we have that 
 $$\mathbb{E}[d_{\mathcal{H}^{d/2}_1}(p^\star_n,p^\star)]\leq C\log(n)^{C_2}n^{-1/2},$$
 which gives us using the previous derivations
 $$\mathbb{E}[d_{\mathcal{H}^{d/2}_1}(\hat{\mu},p^\star)] \leq C\log(n)^{C_2}n^{-1/2}.$$

 The density of $\hat{\mu}$ is equal to $p_{\hat{\mu}}(x)=\gamma_d (\hat{g}^{-1}(x))|\text{det}(\nabla \hat{g}^{-1}(x))|$ so as $\nabla \hat{g}^{-1}\in \mathcal{H}^{\beta}_{C}$, we deduce that $p_{\hat{\mu}} \in \mathcal{H}^{\beta}_{C} $. Then, as  $p^\star$ belongs to $\mathcal{H}^{\beta}_C(\mathbb{R}^d)$, using Theorem 3 from \cite{stephanovitch2024ipm} we have that for all $\gamma \in (0,d/2]$,
 \begin{align*}
\mathbb{E}[d_{\mathcal{H}^{\gamma}_1}(\hat{\mu},p^\star)]& \leq \mathbb{E}[\sup_{h \in \mathcal{H}^{\gamma}_1}\int_{B^d(0,\log(n))} h(x) d\hat{\mu}(x)- \int_{B^d(0,\log(n))} h(x) dp^\star(x)]+C/n\\
& \leq \mathbb{E}[\sup_{h \in \mathcal{H}^{d/2}_1}\left(\int_{B^d(0,\log(n))} h(x) d\hat{\mu}(x)- \int_{B^d(0,\log(n))} h(x) dp^\star(x)\right)^{\frac{\beta+\gamma}{\beta+d/2}}]+C/n\\
& \leq C\log(n)^{C_2}\left(\mathbb{E}[d_{\mathcal{H}^{d/2}_1}(\hat{\mu},p^\star)^{\frac{\beta+\gamma}{\beta+d/2}}]+1/n\right).
\end{align*}
 Using Jensen's inequality, we finally get for all $\gamma \in (0,\infty),$
  $$\mathbb{E}[d_{\mathcal{H}^{\gamma}_1}(\hat{\mu},p^\star)] \leq C\log(n)^{C_2}(n^{-\frac{\beta+\gamma}{2\beta+d}}\vee n^{-1/2})$$
  and this rate has been proven to be optimal (up to the logarithmic factor) in \cite{belomestny2019sobolev}.
\end{proof}

Note that the estimator \eqref{eq:gan} actually attains optimal rates of estimation for any probability measure $p^\star$ that can be written as $T_{\# }\gamma_d$ with $T\in C^{\lfloor \beta \rfloor+1}(\mathbb{R}^d,\mathbb{R}^d)$ a diffeomorphism  satisfying $\nabla T,\nabla T^{-1}\in  \mathcal{H}^{\beta}_L   (\mathbb{R}^d,\mathbb{R}^d)$. In the GAN setting, the existence of a $(\beta+1)$-Hölder transport map is crucial for achieving optimality, as it enables restricting the generator class $\mathcal{G}$ to a smooth class of function, thereby avoiding overfitting effects and enhancing generalization.

\subsection{Score-based Diffusion Models}
Recently, a new class of generative models known as score-based diffusion models has demonstrated state-of-the-art performance across a wide range of domains \citep{song2020score,ho2020denoising,de2022convergence}. The underlying method operates as follows.

Consider the
Ornstein–Uhlenbeck process initialized at a certain density $\mu_0$:
\begin{equation*}
    dY_t=-Y_tdt+\sqrt{2}dB_t,\ \ Y_0\sim \mu_0.
\end{equation*}
Then, taking $\tau>0$, this process can be reverted with the following SDE:
\begin{equation}\label{eq:backward}
    dZ_t=-Z_tdt+2\nabla \log Q_{e^{-(\tau-t)}}r(Z_t)dt+\sqrt{2}dB_t,\ \ Z_0\sim \text{law}(Y_\tau),
\end{equation}
for
\begin{align*}
    r(x):=\frac{\mu_0(x)}{\gamma_d(x)} \quad \text{ and } \quad Q_s r(x)= \int r(sx+\sqrt{1-s^2}z)d\gamma_d(z).
\end{align*}
Under minimal conditions, the final time $Z_\tau$ can be shown to satisfy $Z_\tau \overset{d}{=} \mu_0$ \citep{ANDERSON1982313}. Therefore, one can simulate the process $(Z_t)_t$ until
time $\tau$ to sample from $\mu_0$. Achieving this requires an accurate approximation of the so-called score function $$s(t,x):= \nabla \log \mu_t(x) =-x+\nabla \log Q_{e^{-t}}r(x),$$
for $\mu_t$ the law of $Y_t$.
Equation \eqref{eq:backward} can then be written as
\begin{equation*}
dZ_t=Z_tdt+2s(\tau-t,Z_t)dt+\sqrt{2}dB_t,\ \ Z_0\sim \text{law}(X_\tau).
\end{equation*}

The approximation of the score is typically carried out through a regression objective:
$$\inf_{f\in \mathcal{F}} \int_0^\tau \|f(t,\cdot)-s(t,\cdot)\|^2_{L^2_{\mu_t}}dt,$$
where  $\mathcal{F}$ is generally chosen to be a class of functions parameterized by a neural network. Theoretical studies often rely on controlling the Lipschitz constant of $s(t,\cdot)$ \citep{chen2022sampling,chen2023improved,chen2024probability,kwon2022score}, as these assumptions are crucial for managing discretization errors in numerical schemes used to simulate the process. Theorems \ref{theo:thetheo} and \ref{theo:thetheo2} extend the applicability of these assumptions by encompassing a broader range of distributions $\mu_0$ for which these assumptions can be justified. In fact, by synthesizing Propositions \ref{prop:boundnabalvgauss} and \ref{prop:boundimpossiblebuty}, we obtain the following result. 
\begin{corollary}
    Let $p:\mathbb{R}^d\rightarrow \mathbb{R}$ be a probability density function of the form $$p(x)=\exp\left(-\psi(x)+a(x)\right),$$
such that $a \in \mathcal{H}^{\beta}_K(\mathbb{R}^d,\mathbb{R})$ with $\beta\in (0,1]$ and $K>0$. If either $\psi(x)= -\|x\|^2/2$ or it verifies Assumption \ref{assum:u}, then for all $t>0$ we have
$$\sup_{x\in \mathbb{R}^d} \lambda_{\max}\big(\nabla s(t,x)+\text{Id}\big)\leq Ce^{-t}(1+t^{\beta/2-1}).$$
\end{corollary}
This result provides  a $L^1$ control over the largest eigenvalue of the
Hessian $\nabla^2 \log \mu_t(x)$, under low regularity assumptions on the target distribution. Furthermore, Theorem \ref{theo:thetheo} establishes that the higher regularity of the target distribution transfers to the score function. 
\begin{corollary}\label{coro:diffusregu}
Let $p : \mathbb{R}^d \to \mathbb{R}$ be a probability density function of the form
\[
p(x) = \exp\left( -\frac{\|x\|^2}{2} + a(x) \right),
\]
such that \( a \in H^\beta_K(\mathbb{R}^d, \mathbb{R}) \) with \( \beta \geq 1 \) and \( K > 0 \). Then, for all \( t \in [0, 1) \) we have
\[
\sum_{k = 0}^{\lfloor \beta \rfloor} \sup_{x \in \mathbb{R}^d} \|\nabla^k s(t,x)\| \leq Ce^{-t}(1+t^{-1/2})
\quad \text{and} \quad
\sup_{x \in \mathbb{R}^d} \|\nabla^{\lfloor \beta \rfloor +1} \log s(t,x)\| \leq Ce^{-t}(1 +t^{(\beta - \lfloor \beta \rfloor)/2 - 1}).
\]
\end{corollary}

The estimation of the score function being the core step of the methods, its additional regularity ensures enhanced stability and accuracy in numerical approximations, ultimately improving the quality of generated samples in score-based diffusion models.

\section*{Future directions}
In this work, we have established the high order regularity of the Langevin map transporting the Gaussian measure to a deformation. Not only this result advances our understanding of the properties of diffusion processes but also paves the way for further investigations into the high order regularity for other diffusion flows and other classes of transported measures.

Moreover, an intriguing direction for future research would be to estimate the value of the constant $C>0$ given by Theorem \ref{theo:thetheo}  and its dependency with respect to $K,\beta$ and $d$. This work being a first approach to study high order regularity, it does not focus on the size of this constant. However, we believe that it could be proven to be dimension free. Furthermore, having a direct  dependency with respect to $K$ and $\beta$ would allow to extend the second-order Poincaré inequality \cite{chatterjee2009fluctuations} which finds applications in statistic of eigenvalues \citep{bai2010spectral}. Such explorations could reveal deeper connections between regularity, transport phenomena, and functional inequalities, thereby enriching the field of measure transportation and its applications.

\bibliography{bib}

\newpage

\appendix

\section{Estimate on  $\|\nabla X_1\|_{\mathcal{H}^{\lfloor \beta \rfloor}}$ in the setting of Theorem \ref{theo:thetheo2}}\label{sec:prooftheo2}
In this section we focus on  a probability density $p$ satisfying the assumptions of Theorem \ref{theo:thetheo2}. Using the results of Section \ref{sec:generalreas}, we give a bound on $\int_0^1 \sup_{z\in \mathbb{R}^d}\lambda_{\max}(\nabla V(t,z))dt$ and $\int_0^1 \|\nabla^j V(t,X_t(x))\|dt$ for $j\in \{2,...,k\}$, $x\in B^d(0,\log(\epsilon^{-1}))$ with $\epsilon\in (0,1/2)$.

\subsection{Well-definedness of the flow and Lispchitz regularity}
\subsubsection{Well-definedness of the flow}
Let us start by showing that the flow of transport maps $(X_t)$ solution to \eqref{eq:PDE} is well defined for all $t\in [0,1)$ using Proposition \ref{prop:condiwelldifned}. In fact, the assumptions of Proposition \ref{prop:condiwelldifned} are automatically verified for compactly supported measure, as stated in the following Lemma.
\begin{lemma}\label{lemma:le5}
    If a probability density $p$ is compactly supported, then for all $t\in [0,1)$ its associated velocity field $V$ from \eqref{eq:V} satisfies
$$V\in L^1\Big([0,t],W^{1,\infty}_{loc}(\mathbb{R}^d,\mathbb{R}^d)\Big)\text{ and } \frac{V}{1+\|x\|}\in L^1\Big([0,t],L^{\infty}(\mathbb{R}^d,\mathbb{R}^d)\Big).$$
In particular there exists a unique solution for all $t\in [0,1)$ to the equation
    $$\partial_t X_t(x) =V(t,X_t(x)), \quad X_0(x)=x.$$
\end{lemma}
\begin{proof}
    Let $s\in [0,t]$ and $x\in \mathbb{R}^d$, we have
\begin{align*}
    \|V(s,x)\| = &\frac{1}{s}\|\nabla \log(Q_s r)(x)\|\\
     =&\|\frac{1}{(1-s^2)}\int (y-tx) dp^{s,x}(y)\|\\
     \leq &\frac{C+\|x\|}{1-t^2},
\end{align*}
as the probability $p^{s,x}$ has the same support than $p$. Furthermore,
\begin{align*}
           \|\nabla V(s,x)\| = &\frac{1}{s}\|\nabla^2 \log(Q_s r)(x)\|\\
     =&\|\frac{s}{(1-s^2)^2}\int \left(y-\int z dp^{s,x}(y)\right)^{\otimes 2}dp^{s,x}(y)-\frac{s}{1-s^2}\text{Id}\|\\
     \leq & \frac{s}{(1-t^2)^2}\left( \|\int \left(y-\int z dp^{s,x}(y)\right)^{\otimes 2} dp^{s,x}(y)\|+1\right)\\
     \leq &\frac{C}{(1-t^2)^2}.
\end{align*}
Therefore, for all compact Borel set $B\subset \mathbb{R}^d$, we have 
\begin{align*}
    \int_0^t \Big(\|\nabla V(s,\cdot)\|_{L^\infty(B)}+\| V(s,\cdot)\|_{L^\infty(B)}  \Big)ds <\infty,
\end{align*}
and
\begin{align*}
    \int_0^t \| \frac{V(s,\cdot)}{1+\|\cdot\|}\|_{L^\infty}ds <\infty.
\end{align*}
\end{proof}

\subsubsection{Bounding $\|\nabla X_1\|_\infty$}

The bound on the first derivatives differs from the one of higher order as it only needs an estimate on $\lambda_{\max}\left(\nabla V(t,x)\right)$, not on $\|\nabla V(t,x)\|$. In the following, we use the notion of semi-log concave probability measure. 
\begin{definition}\label{defi:logconcave}
    A probability measure having a density $p=\exp(v)$ with respect to the $d$-dimensional Lebesgue measure is said to be $\kappa$-semi-log concave for $\kappa\in \mathbb{R}$ if its support $\Omega\subset \mathbb{R}^
    d$ is convex and $v\in C^2(\Omega)$ satisfies
    $$-\nabla^2 v \succeq \kappa \text{Id}.$$
\end{definition}

For strictly log concave probability measures, we will use the well known Brascamp-Lieb inequality.
\begin{proposition}\label{prop:BL}[\cite{BRASCAMP1976366}]
    Let a probability density function $q(x)=\exp(-\phi(x))$ with $\phi$ convex. Then, for any derivable function $S:\mathbb{R}^d\rightarrow \mathbb{R}$ we have $$\int \left(S(y)-\int S(z)dq(z)\right)^2dq(y)\leq \int \big(\nabla^2 \phi(y)\big)^{-1}(\nabla S(y) ,\nabla S(y)) dq(y).$$
\end{proposition}

Having this tool in hand, let us provide a bound on the growth of $\lambda_{\max}(\nabla V(t,x))$ with respect to $t$.

\begin{proposition}\label{prop:boundimpossiblebuty}
For all $t\in [0,1)$ and $x\in \mathbb{R}^d$,  we have
$$\lambda_{\max}(\nabla V(t,x))\leq C(1-t^2)^{(\beta-\lfloor \beta\rfloor)/2-1},$$
so in particular 
    $$\int_0^1 \sup_{x\in \mathbb{R}^d}\lambda_{\max}(\nabla V(t,x))dt<\infty.$$
\end{proposition}

\begin{proof}
 We have 
\begin{align}\label{eq:vubifvhd}
        \lambda_{\max}(\nabla V(t,x)) = &\frac{1}{t}\lambda_{\max}((\nabla^2 \log(Q_t r)(x))\nonumber\\
=&\lambda_{\max}\left(\frac{t}{(1-t^2)^2}\int H(y,x)^{\otimes 2} dp^{t,x}(y)-\frac{t}{1-t^2}\text{Id}\right)\nonumber\\
=&\lambda_{\max}\left(\frac{t}{(1-t^2)^2}\int H(y,x)^{\otimes 2} dp^{t,x}(y)\right)-\frac{t}{1-t^2}.
\end{align}
In the case $t\leq 1/2$, we immediately get that
$$\lambda_{\max}(\nabla V(t,x))\leq C,$$
as $p^{t,x}$ is supported on the ball. 

Let us now treat the case $t\geq 1/2.$ Since \( p \) is of the form \( p(x) = \exp(-u(\|x\|) + a(x)) \), with \( a \) only Hölder continuous, the Brascamp--Lieb inequality cannot be applied directly.
 To circumvent this, we introduce an auxiliary measure to which the inequality can be applied.
Let
$$q(y):=\exp(-u(\|y\|^2)+\|y\|^2/2)=r(y)e^{-a(y)}$$
and
$$\nu^{t,x}(y) :=\frac{q(y)\varphi^{t,x}(y)}{Q_tq(x)}=\frac{\exp(-u(\|y\|^2)+\|y\|^2/2)\varphi^{t,x}(y)}{\int \exp(-u(\|z\|^2)+\|z\|^2/2)\varphi^{t,x}(z)d\lambda^d(z)}.$$
Then, we have 
\begin{align}\label{align:vfdihsksgs}
\lambda_{\max}\left(\int H(y,x)^{\otimes 2} dp^{t,x}(y)\right)\leq &\lambda_{\max}\left(\int H(y,x)^{\otimes 2} dp^{t,x}(y)-\int\left(y-\int zd\nu^{t,x}(z)\right)^{\otimes 2}d\nu^{t,x}(y)\right)\nonumber\\
&+ \lambda_{\max}\left(\int\left(y-\int zd\nu^{t,x}(z)\right)^{\otimes 2}d\nu^{t,x}(y)\right).
\end{align}

As the probability density $\nu^{0,x}=\frac{q\gamma_d}{\int qd\gamma_d}$  is $0$-semi-log concave, we obtain that 
 $\nu^{t,x}$ is $\left(\frac{t^2}{1-t^2}\right)$-log-concave.
Indeed, for all $y\in \mathbb{R}^d$ we have
\begin{equation}\label{eq:ptxislogconc}
-\nabla^2\log(\nu^{t,x}(y))=-\nabla^2\log(q(y)\gamma_d(y))-\nabla^2\log(\varphi^{t,x}(y)/\gamma_d(y))\succeq \frac{t^2}{1-t^2} \text{Id}.
\end{equation}
Then, using the Brascamp-Lieb inequality applied to functions of the form $x\mapsto \langle x,w\rangle $ with $w\in \mathbb{S}^{d-1}$, we obtain
\begin{align*}
    \lambda_{\max}\left(\int \left(y-\int z d\nu^{t,x}(z)\right)^{\otimes2} d\nu^{t,x}(y)\right)\leq \frac{1-t^2}{t^2},
\end{align*}
so
\begin{align*}  \frac{t}{(1-t^2)^2}\lambda_{\max}\left(\int \left(y-\int z d\nu^{t,x}(z)\right)^{\otimes2} d\nu^{t,x}(y)\right)-\frac{t}{1-t^2}&\leq \frac{t}{1-t^2} \frac{1-t^2}{t^2}\\
& = \frac{1}{t}.
\end{align*}

Let us give an estimate on the other term appearing in \eqref{align:vfdihsksgs}. We have
\begin{align}\label{align:jsdujzbsio}
    &\int \left(y-\int zdp^{t,x}(z)\right)^{\otimes 2}dp^{t,x}(y)-\int\left(y-\int zd\nu^{t,x}(z)\right)^{\otimes 2}d\nu^{t,x}(y)\nonumber\\
    &=\int \left(\left(y-\int zdp^{t,x}(z)\right)^{\otimes 2}\frac{e^{a(y)}}{Q_tr(x)}-\left(y-\int zd\nu^{t,x}(z)\right)^{\otimes 2}\frac{1}{Q_tq(x)}\right)\nonumber\\
    & \quad \quad  \times \exp(-u(\|y\|^2)+\|y\|^2/2)d\varphi^{t,x}(y)\nonumber\\
    &=\int \left(\left(y-\int zdp^{t,x}(z)\right)^{\otimes 2}-\left(y-\int zd\nu^{t,x}(z)\right)^{\otimes 2}\right)\frac{e^{a(y)}}{Q_tr(x)}\exp(-u(\|y\|^2)+\|y\|^2/2)d\varphi^{t,x}(y)\nonumber\\
    &+\int \left(y-\int zd\nu^{t,x}(z)\right)^{\otimes 2}\left(\frac{e^{a(y)}}{Q_tr(x)}-\frac{1}{Q_tq(x)}\right)\exp(-u(\|y\|^2)+\|y\|^2/2)d\varphi^{t,x}(y)\nonumber\\
    &=\int \left(\left(y-\int zdp^{t,x}(z)\right)^{\otimes 2}-\left(y-\int zd\nu^{t,x}(z)\right)^{\otimes 2}\right)\frac{e^{a(y)}Q_tq(x)}{Q_tr(x)}d\nu^{t,x}(y)\nonumber\\
    &+\int \left(y-\int zd\nu^{t,x}(z)\right)^{\otimes 2}\left(\frac{e^{a(y)}Q_tq(x)}{Q_tr(x)}-1\right)d\nu^{t,x}(y).
\end{align}
Now,
\begin{align*}
    &\|\left(y-\int zdp^{t,x}(z)\right)^{\otimes 2}-\left(y-\int zd\nu^{t,x}(z)\right)^{\otimes 2}\|\\
    = &\|\left(\int zdp^{t,x}(z)-y\right)\otimes \int zdp^{t,x}(z)+\left(\int zd\nu^{t,x}(z)-\int zdp^{t,x}(z)\right)\otimes y+ \left(y-\int zd\nu^{t,x}(z)\right)\otimes \int zd\nu^{t,x}(z)\|\\
    \leq &\|\left(\int zdp^{t,x}(z)-y\right)\otimes \left(\int zdp^{t,x}(z)-\int zd\nu^{t,x}(z)\right)\|\\
    & +\|\left(\int zdp^{t,x}(z)-\int zd\nu^{t,x}(z)\right)\otimes \left( \int zd\nu^{t,x}(z)-y\right)\|\\
    & \leq \|\int zdp^{t,x}(z)-\int zd\nu^{t,x}(z)\| \left(\| \int zd\nu^{t,x}(z)-y\|+\| \int zdp^{t,x}(z)-y\|\right)\\
    & \leq \|\int zdp^{t,x}(z)-\int zd\nu^{t,x}(z)\| \left(2\| \int zd\nu^{t,x}(z)-y\|+\| \int zdp^{t,x}(z)-\int zd\nu^{t,x}(z)\|\right)
\end{align*}
and
\begin{align*}
    \|\int yd\nu^{t,x}(y)-\int ydp^{t,x}(y)\|& = \|\int \left(y-\int zd\nu^{t,x}(z)\right)d\nu^{t,x}(y)-\int \left(y-\int zd\nu^{t,x}(z)\right)dp^{t,x}(y)\|\\
    & = \|\int \left(y-\int zd\nu^{t,x}(z)\right)\left(1-\frac{e^{a(y)}Q_tq(x)}{Q_tr(x)}\right)d\nu^{t,x}(y)\|.
\end{align*}

On the other hand,
$$\frac{Q_tq(x)}{Q_tr(x)} = \frac{\int r(y)e^{-a(y)}\varphi^{t,x}(y)d\lambda^d(y)}{\int r(y)\varphi^{t,x}(y)d\lambda^d(y)} \in [C^{-1},C],$$
so from \eqref{align:jsdujzbsio} we have
\begin{align}\label{align:brzdijnzide}
    \|&\int \left(y-\int zdp^{t,x}(z)\right)^{\otimes 2}dp^{t,x}(y)-\int\left(y-\int zd\nu^{t,x}(z)\right)^{\otimes 2}d\nu^{t,x}(y)\|\nonumber\\
     \leq &C\int \|\left(y-\int zdp^{t,x}(z)\right)^{\otimes 2}-\left(y-\int zd\nu^{t,x}(z)\right)^{\otimes 2}\|d\nu^{t,x}(y)\nonumber\\
    &+\int \|y-\int zd\nu^{t,x}(z)\|^{ 2}|\frac{e^{a(y)}Q_tq(x)}{Q_tr(x)}-1|d\nu^{t,x}(y)\nonumber\\
     \leq &C\left(\int \|y-\int zd\nu^{t,x}(z)\||\frac{e^{a(y)}Q_tq(x)}{Q_tr(x)}-1|d\nu^{t,x}(y)\right)^2\nonumber\\
    & + C\int \|y-\int zd\nu^{t,x}(z)\|d\nu^{t,x}(y)\int \|y-\int zd\nu^{t,x}(z)\||\frac{e^{a(y)}Q_tq(x)}{Q_tr(x)}-1|d\nu^{t,x}(y)\nonumber\\
    &+\int \|y-\int zd\nu^{t,x}(z)\|^{ 2}|\frac{e^{a(y)}Q_tq(x)}{Q_tr(x)}-1|d\nu^{t,x}(y).
\end{align}

Let us now show that additional concentration can be extracted from the discrepancy term \( \left| \frac{e^{a(y)} Q_t q(x)}{Q_t r(x)} - 1 \right| \). We have,
\begin{align*}
&|\frac{e^{a(y)}Q_tq(x)}{Q_tr(x)}-1|\\
&=\frac{1}{Q_tr(x)}|e^{a(y)}Q_tq(x)-e^{a\left(\int zd\nu^{t,x}(z)\right)}Q_tq(x)+e^{a\left(\int zd\nu^{t,x}(z)\right)}Q_tq(x)-Q_tr(x)|\\
&\leq C |e^{a(y)}-e^{a\left(\int zd\nu^{t,x}(z)\right)}|+\frac{1}{Q_tr(x)}|\int\left( e^{a\left(\int zd\nu^{t,x}(z)\right)}-e^{a(w)}\right)q(w)d\varphi^{t,x}(w)|\\
&\leq C \|y-\int zd\nu^{t,x}(z)\|^{\beta-\lfloor \beta\rfloor}+C\int\|\int zd\nu^{t,x}(z)-w\|^{\beta-\lfloor \beta\rfloor}d\nu^{t,x}(w).
\end{align*}

Furthermore, for $\alpha \in (0,1)$ using the Brascamp-Lieb and Jensen's inequalities we obtain
\begin{align*}
    &\int \|y-\int z d\nu^{t,x}(z)\|^{2+\alpha} d\nu^{t,x}(y) \leq \left(\int \|y-\int z d\nu^{t,x}(z)\|^{4} d\nu^{t,x}(y)\right)^{(2+\alpha)/4}\\
   & \leq \left(\int \|y-\int z d\nu^{t,x}(z)\|^{4} d\nu^{t,x}(y)-\left(\int \|y-\int z d\nu^{t,x}(z)\|^{2} d\nu^{t,x}(y)\right)^{2}\right)^{(2+\alpha)/4}\\
   & +\left(\int \|y-\int z d\nu^{t,x}(z)\|^{2} d\nu^{t,x}(y)\right)^{(2+\alpha)/2}\\
   & \leq \left(\int \left(\|y-\int z d\nu^{t,x}(z)\|^{2} -\int \|w-\int z d\nu^{t,x}(z)\|^{2} d\nu^{t,x}(w)\right)^2d\nu^{t,x}(y)\right)^{(2+\alpha)/4}+C\left(\frac{1-t^2}{t^2}\right)^{(2+\alpha)/2}\\
   & \leq \left(C\frac{1-t^2}{t^2} \int \|y-\int z d\nu^{t,x}(z)\|^{2}d\nu^{t,x}(y)\right)^{(2+\alpha)/4}+C\left(\frac{1-t^2}{t^2}\right)^{(2+\alpha)/2}
   \\
   & \leq \left(C\left(\frac{1-t^2}{t^2}\right)^2\right)^{(2+\alpha)/4}+C\left(\frac{1-t^2}{t^2}\right)^{(2+\alpha)/2}\\
   & \leq C\left(\frac{1-t^2}{t^2}\right)^{(2+\alpha)/2}.
\end{align*}
Then, using Jensen's inequality and putting everything together in \eqref{align:brzdijnzide}, we obtain

\begin{align*}
    \|\int \left(y-\int zdp^{t,x}(z)\right)^{\otimes 2}dp^{t,x}(y)-\int\left(y-\int zd\nu^{t,x}(z)\right)^{\otimes 2}d\nu^{t,x}(y)\|\leq C\left(\frac{1-t^2}{t^2}\right)^{(2+\beta-\lfloor \beta\rfloor)/2}.
\end{align*}
Therefore, we finally have for $t\geq 1/2$ that
$$\lambda_{\max}(\nabla V(t,x))\leq \frac{1}{t(1-t^2)^{1-(\beta-\lfloor \beta\rfloor)/2}}.$$
As for $t\leq 1/2$, we have
$$\lambda_{\max}(\nabla V(t,x))\leq C,$$
we deduce that 
$$\int_0^1 \sup_{x\in \mathbb{R}^d}\lambda_{\max}(\nabla V(t,x))dt\leq C.$$
\end{proof}
Using Proposition \ref{prop:boundimpossiblebuty}, we finally obtain that there exists a Lipschitz transport map from the $d$-dimensional Gaussian distribution to the probability $p$.

\begin{corollary}\label{coro:onasue}
    Under the assumptions of Theorem \ref{theo:thetheo2}, there exists a unique solution for all $t\in [0,1)$ to the equation
    $$\partial_t X_t(x) =V(t,X_t(x)), \quad X_0(x)=x,$$
    and the solution satisfies $X_t\in \mathcal{H}^1_C$.
    Furthermore, there exists a limit map $T$ such that $T_{\#}\gamma_d=p$ and
    $T \in \mathcal{H}^1_C.$
\end{corollary}
\begin{proof}
    From Lemma \ref{lemma:le5} we have the existence of the solution and from  Proposition \ref{prop:boundimpossiblebuty} and Lemma \ref{lemma:addlemme}, we obtain its Lipschitz regularity. We can then conclude to the existence of the Lipschitz transport map using Lemma \ref{lemma:convergenceofmap}.
\end{proof}

\subsection{Higher order estimates}

\subsubsection{Localisation of $X_t(x)$}
Let us fix $\epsilon\in (0,1/2)$, in this section we control the norm of  $X_t(x)$ for $x\in \mathbb{R}^d$ and $t\in [0,1)$ that satisfies
\begin{equation}\label{eq:hyp1}
    \|x\|^2\leq \log(\epsilon^{-1})\ \ \text{ and }\ \ 1-t^2\leq C_\star^{-1}\log(\epsilon^{-1})^{-C_\star},
\end{equation}
for $C_\star$ large enough to be determined later. This estimate is needed to control the norm of the derivatives of $p$ at the point $X_t(x)$. 
\begin{proposition}\label{prop:loca3}
There exists $C^\star>0$ such that for $t\in [0,1)$  and $x\in \mathbb{R}^d$ satisfying \begin{equation*}
1-t^2\leq C_\star^{-1}\log(\epsilon^{-1})^{-C_\star} \text{ and } \|x\|^2\leq \log(\epsilon^{-1}),
\end{equation*} 
we have
\begin{align*}
\|tX_t(x)\|^2 \leq 1-C_2^{-1}\log(\epsilon^{-1}+1)^{-K}.
\end{align*}
\end{proposition}
\begin{proof}
Let us first recall from \eqref{eq:laxXt} that we have 
\begin{equation}\label{eq:lawbis}
(X_t)_{\#}\gamma_d\overset{d}{=}tX+\sqrt{1-t^2}Y,
\end{equation}
with $X\sim p$ and $Y\sim \gamma_d$. Writing $f_{X_t}$ the density of $(X_t)_{\#}\gamma_d$, we have
\begin{align}\label{eq:densityxt}
f_{X_t}(z)&=\frac{1}{t^d(1-t^2)^{d/2}}\int p(\frac{y}{t})\gamma_d(\frac{z-y}{(1-t^2)^{1/2}})d\lambda^d(y)\nonumber \\
&=\frac{1}{(1-t^2)^{d/2}}\int \frac{p(y)}{\gamma_d(y)}\exp\left(-\frac{1}{2(1-t^2)}\big(\|z-ty\|^2+(1-t^2)\|y\|^2\big)\right)d\lambda^d(y)\nonumber \\
&=\frac{1}{(1-t^2)^{d/2}}\int r(y)\exp\left(-\frac{1}{2(1-t^2)}\big(\|z\|^2-2t\langle z,y\rangle+\|y\|^2\big)\right)d\lambda^d(y)\nonumber \\
& = \exp(-\frac{\|z\|^2}{2}) \int r(y)\varphi^{t,z}(y)d\lambda^d(y).
\end{align}
As for all $t\in [0,1)$ we have that $x\mapsto X_t(x)$ is a diffeomorphism, we deduce that we also have 
\begin{align*}
f_{X_t}(z)&=\gamma_d(X_t^{-1}(z))\text{det}(\nabla X_t(X_t^{-1}(z))),
\end{align*}
so 
\begin{align*}
f_{X_t}(X_t(x))&=\gamma_d(x)\text{det}(\nabla X_t(x))^{-1}.
\end{align*}
Therefore, from \eqref{eq:densityxt} we deduce that
\begin{align}\label{align:lowerboundQ}
\int r(y)\varphi^{t,X_t(x)}(y)d\lambda^d(y)&= \exp(\frac{\|X_t(x)\|^2}{2})\gamma_d(x)|\text{det}(\nabla X_t(x))|^{-1}\nonumber\\
& \geq C^{-1}\exp(\frac{1}{2}(\|X_t(x)\|^2-\|x\|^2)),
\end{align}
as $X_t$ is Lipschitz for all $t\in [0,1)$ from Corollary \ref{coro:onasue}. Now, we have

\begin{align}\label{align:upperboundQ}
    \int& r(y)\varphi^{t,X_t(x)}(y)d\lambda^d(y)\nonumber\\
    = &\int r(y)\varphi^{t,X_t(x)}(y)\mathds{1}_{\{\|y-tX_t(x)\|\leq (1-t^2)^{\frac{1}{4}}\}}d\lambda^d(y)+\int r(y)\varphi^{t,X_t(x)}(y)\mathds{1}_{\{\|y-tX_t(x)\|> (1-t^2)^{\frac{1}{4}}\}}d\lambda^d(y)\nonumber \\
    & \leq C \|\exp\left(-u(\|\cdot\|^2)\right)\|_{\mathcal{H}^0(B^d(tX_t(x),(1-t^2)^{\frac{1}{4}}))}\mathds{1}_{\|tX_t(x)\|<1 + (1-t^2)^{\frac{1}{4}}}+C\exp\left(-\frac{1}{2(1-t^2)^{\frac{1}{2}}}\right).
\end{align}
Suppose that we have $$\|\exp\left(-u(\|\cdot\|^2)\right)\|_{\mathcal{H}^0(B^d(tX_t(x),(1-t^2)^{\frac{1}{4}}))}\mathds{1}_{\|tX_t(x)\|<1 + (1-t^2)^{\frac{1}{4}}}\leq\exp\left(-\frac{1}{2(1-t^2)^{\frac{1}{2}}}\right),$$
then from \eqref{align:lowerboundQ} we have

\begin{align*}
C\exp\left(-\frac{1}{2(1-t^2)^{\frac{1}{2}}}\right) & \geq C^{-1}\exp(\frac{1}{2}(\|X_t(x)\|^2-\|x\|^2)),
\end{align*}
so 
\begin{align*}
\log(C)-\frac{1}{(1-t^2)^{\frac{1}{2}}}+ \|x\|^2& \geq \|X_t(x)\|^2.
\end{align*}
This is not possible for $x$ and $t$ satisfying \eqref{eq:hyp1} with $C^\star>0$ large enough as 
\begin{align*}
    \log(C)-\frac{1}{(1-t^2)^{\frac{1}{2}}}+ \|x\|^2\leq \log(C)-(C^\star)^{1/2}\log(\epsilon^{-1})^{C^\star/2}+ \log(\epsilon^{-1})<0.
\end{align*}

This gives us that  $$\|\exp\left(-u(\|\cdot\|^2)\right)\|_{\mathcal{H}^0(B^d(tX_t(x),(1-t^2)^{\frac{1}{4}}))}\mathds{1}_{\|tX_t(x)\|<1 + (1-t^2)^{\frac{1}{4}}}>\exp\left(-\frac{1}{2(1-t^2)^{\frac{1}{2}}}\right),$$
so from \eqref{align:upperboundQ} and \eqref{align:lowerboundQ} we obtain
$$C\|\exp\left(-u(\|\cdot\|^2)\right)\|_{\mathcal{H}^0(B^d(tX_t(x),(1-t^2)^{\frac{1}{4}}))}\mathds{1}_{\|tX_t(x)\|<1 + (1-t^2)^{\frac{1}{4}}}\geq C^{-1}\exp(\frac{1}{2}(\|X_t(x)\|^2-\|x\|^2)),$$
which gives 
\begin{align*}
\log(C)-2\min_{y\in B^d(tX_t(x),(1-t^2)^{\frac{1}{4}})}u(\|y\|^2)+ \|x\|^2& \geq \|X_t(x)\|^2\geq 0.
\end{align*}
Then, for $x$ and $t$ satisfying \eqref{eq:hyp1}, we get
$$ 2\min_{y\in B^d(tX_t(x),(1-t^2)^{\frac{1}{4}})}u(\|y\|^2)\leq \log(C)+\log(\epsilon^{-1}) ,$$
so as $u(s)\geq (1-s)^{-\frac{1}{K}}$, we have
\begin{align}
(1-(\|tX_t(x)\|-(1-t^2)^{\frac{1}{4}}))^2)^{-1/K}&\leq C\log(\epsilon^{-1}+1)\nonumber.
\end{align}
Finally, as $\|tX_t(x)\|< 1 + (1-t^2)^{\frac{1}{4}}$, we get for $C^\star>0$ large enough
\begin{align}\label{eq:loca3}
\|tX_t(x)\|^2\leq     1-C^{-1}\log(\epsilon^{-1}+1)^{-K}.
\end{align}
\end{proof}

The lower bound \eqref{eq:loca3} is a key result stating that $X_t(x)$ is not too close to the boundary of the ball which gives us in particular using Assumption \ref{assum:u} that
$$\|f^{l}(X_t(x))\|=\|\nabla^l r(X_t(x))/r(X_t(x))\|\leq C\log(\epsilon^{-1}) ^{C_2}.$$

\subsubsection{Concentration of the mass under $p^{t,x}$} 
In this section we derive a bound on the quantity $\int \|f(y)-f(tx)\|^2dp^{t,x}(y)$ with respect to $t\in [1-C_\star^{-1}\log(\epsilon^{-1})^{-C_\star},1]$ and the growth of the function $f$ around the point $tx$. Let us first prove a concentration result on $p^{t,x}$, analogous to Proposition \ref{prop:concentHtheo1} in the case of Theorem \ref{theo:thetheo2}.

\begin{proposition}\label{prop:concentHtheo2} There exists  $0<C_\star<C_2$ such that if $t\in [0,1)$  and $x\in \mathbb{R}^d$ satisfy \begin{equation}\label{eq:hyp12}
1-t^2\leq C_\star^{-1}\log(\epsilon^{-1})^{-C_\star} \text{ and } \|tx\|^2\leq 1-C_2^{-1}\log(\epsilon^{-1}+1)^{-K},
\end{equation}
then for all $i,j\in \{1,...,d\}$ we have
    $$\int H(y,x)_i^2 dp^{t,x}(y)\leq C(1-t^2)$$
    and 
    $$\int \left(H(y,x)^{\otimes2}_{ij}- \int H(z,x)^{\otimes2}_{ij}dp^{t,x}(z)\right)^2dp^{t,x}(y)\leq C(1-t^2)^{2}.$$
\end{proposition}
\begin{proof}
Let $Y,\overline{Y}$ two independent random variables of law $p^{t,x}$ and $Z,\overline{Z}$ two independent random variables of law $\varphi^{t,x}$. We have
\begin{align*}
    \int H(y,x)_i^2 dp^{t,x}(y) = &\int \left(y_i-\int y_idp^{t,x}(z)\right)^2  dp^{t,x}(y)\\
    =  &\text{Var}_{p^{t,x}}[Y_i]
    = \frac{1}{2} \mathbb{E}_{p^{t,x}}[(Y_i-\overline{Y}_i)^2]\\
    = & \frac{1}{2} \mathbb{E}_{p^{t,x}}[(Y_i-\overline{Y}_i)^2\mathds{1}_{\{\|Y-tx\|\leq \delta\}\cap \{\|\overline{Y}-tx\|\leq \delta\}}]+ \frac{1}{2}\mathbb{E}_{p^{t,x}}[(Y_i-\overline{Y}_i)^2\mathds{1}_{\{\|Y-tx\|\geq \delta\}\cup \{\|\overline{Y}-tx\|\geq \delta\}}],
\end{align*}
with $\delta=(1-t^2)^{1/4}$. In order to control the first term,  we are going to bound $Q_tr(x)^{-1}$, for $x$ and $t$ satisfying
\begin{equation}\label{eq:lequationencore}
1-t^2\leq C_\star^{-1}\log(\epsilon^{-1})^{-C_\star} \text{ and } \|tx\|^2 \leq 1-C_2^{-1}\log(\epsilon^{-1}+1)^{-K},
\end{equation} 
with $C^\star,C_2>0$ given by Proposition \ref{prop:loca3}.
We obtain
\begin{align}\label{align:boundbelowq}
    Q_{t}r(x) & = \int \varphi^{t,x}(y)r(y)d\lambda^d(y)\nonumber\\
    & \geq  \min_{w\in B^d(tx,(1-t^2)^{1/2}))}\exp\left(-u(\|w\|^2)+a(w)\right)\int \varphi^{t,x}(y)\mathds{1}_{\|y-tx\|\leq(1-t^2)^{1/2}}d\lambda^d(y)\nonumber\\
    & \geq C^{-1}\min_{w\in B^d(tx,(1-t^2)^{1/2}))}\exp\left(-u(\|w\|^2)\right)\int \mathds{1}_{\|y\|\leq\frac{(1-t^2)^{1/2}}{(1-t^2)^{1/2}}}\gamma_d(y)d\lambda^d(y),\nonumber\\
    & \geq C^{-1}\min_{w\in B^d(tx,(1-t^2)^{1/2}))}\exp\left(-u(\|w\|^2)\right).
\end{align}
On the other hand, for all $y\in \mathbb{R}^d$ we have 
$$r(y)=\exp(-u(\|y\|^2+a(y))\leq C\exp(-u(\|y\|^2)),$$
and for $t$ and $x$ satisfying \eqref{eq:lequationencore}, we have
$$\max_{z\in B^d(tx,\delta)}u^{(1)}(\|z\|^2)\leq K(1-(\|tx\|+\delta)^2)^{-(K+1)}\leq C\log(\epsilon^{-1})^{K(K+1)}\leq (1-t^2)^{-1/4},$$
taking $C^\star>0$  large enough.
We deduce for the first term that
\begin{align*}
    \mathbb{E}_{p^{t,x}}&[(Y_i-\overline{Y}_i)^2\mathds{1}_{\{\|Y-tx\|\leq \delta\}\cap \{\|\overline{Y}-tx\|\leq \delta\}}]\\
    =&\int\int(y_i-\overline{y}_i)^2\frac{r(y)r(\overline{y})}{Q_tr(x)^2}\varphi^{t,x}(y)\varphi^{t,x}(\overline{y})\mathds{1}_{\{\|y-tx\|\leq \delta\}\cap \{\|\overline{y}-tx\|\leq \delta\}}d\lambda^d(y)d\lambda^d(\overline{y})\\
    \leq & C \int\int(y_i-\overline{y}_i)^2\exp\left(2\max_{z\in B^d(tx,(1-t^2)^{1/2})}u(\|z\|^2)-2\min_{w\in B^d(tx,\delta)}u(\|w\|^2)\right)\varphi^{t,x}(y)\varphi^{t,x}(\overline{y})d\lambda^d(y)d\lambda^d(\overline{y})\\
    \leq & C \int\int(y_i-\overline{y}_i)^2\exp\left(2\max_{z\in B^d(tx,\delta)}u^{(1)}(\|z\|^2)\delta\right)\varphi^{t,x}(y)\varphi^{t,x}(\overline{y})d\lambda^d(y)d\lambda^d(\overline{y})
    \\
    \leq & C \int\int(y_i-\overline{y}_i)^2\varphi^{t,x}(y)\varphi^{t,x}(\overline{y})d\lambda^d(y)d\lambda^d(\overline{y})\\
    \leq & C \mathbb{E}_{\varphi^{t,x}}[(Z_i-\overline{Z}_i)^2]\\
    \leq  & C\text{Var}_{\varphi^{t,x}}(Z_i)\\
    \leq & C(1-t^2).
\end{align*}
Now, for the second term, as $p^{t,x}$ is supported on the ball we have
\begin{align*}
    \mathbb{E}_{p^{t,x}}[(Y_i-\overline{Y}_i)^2\mathds{1}_{\{\|Y-tx\|\geq \delta\}\cup \{\|\overline{y}-tx\|\geq \delta\}}] &\leq C\int\frac{r(y)}{Q_tr(x)}\varphi^{t,x}(y)\mathds{1}_{\{\|Y-tx\|\geq \delta\}}d\lambda^d(y)\\
    &\leq C\frac{\exp(-\frac{\delta^2}{2(1-t^2)})}{Q_tr(x)}\\
        & \leq C\exp\left(-\frac{(1-t^2)^{-1/2}}{2}\right)\exp\left(\max_{z\in B^d(0,\|tx\|+(1-t^2)^{1/2})}u^{(1)}(\|z\|^2)\right)\\
    & \leq C\exp\left(-\frac{(1-t^2)^{-1/2}}{2}+(1-(\|tx\|+(1-t^2)^{1/2})^2)^{-(K+1)}\right)\\ 
     & \leq C\exp\left(-\frac{(1-t^2)^{-1/2}}{4}\right).
\end{align*}

Applying the same reasoning, we obtain
    $$\int \left(H(y,x)^{\otimes2}_{ij}- \int H(z,x)^{\otimes2}_{ij}dp^{t,x}(z)\right)^2dp^{t,x}(y)\leq C(1-t^2)^{2}.$$

\end{proof}

Using Proposition \ref{prop:concentHtheo2}, we directly obtain the analog of Proposition \ref{prop:firstboundongradsth1} in the case of Theorem \ref{theo:thetheo2}.
\begin{proposition}\label{prop:firstboundongradsth2} There exists  $0<C_\star<C_2$ such that if $t\in [0,1)$  and $x\in \mathbb{R}^d$ satisfy \begin{equation}\label{eq:hyp13}
1-t^2\leq C_\star^{-1}\log(\epsilon^{-1})^{-C_\star} \text{ and } \|tx\|^2\leq 1-C_2^{-1}\log(\epsilon^{-1}+1)^{-K},
\end{equation}
    then for all $f\in L^2(\mathbb{R}^d)$, $i\in \{1,2\}$ and $\xi\in \mathbb{R}^d$, we have
    \begin{align*}
    \|\int f(y) \nabla_x^i p^{t,x}(y)d\lambda^d(y)\| 
    &\leq \frac{Ct^i}{(1-t^2)^{i/2}}\left(\int \|f(y)-\xi\|^2dp^{t,x}(y)\right)^{1/2}.
\end{align*}
\end{proposition}
Let us now give a result on the concentration of the mass under $p^{t,x}$.

\begin{proposition}\label{prop:concentration3}
There exists  $0<C_\star<C_2$ such that if $t\in [0,1)$  and $x\in \mathbb{R}^d$ satisfy \begin{equation}\label{eq:hyp14}
1-t^2\leq C_\star^{-1}\log(\epsilon^{-1})^{-C_\star} \text{ and } \|tx\|^2\leq 1-C_2^{-1}\log(\epsilon^{-1}+1)^{-K},
\end{equation}
then for all $l\in \{1,...,\lfloor \beta \rfloor\}$ we have
    $$\int \|f^l(y)-f^l(tx)\|^2dp^{t,x}(y)\leq C\log(\epsilon^{-1})^{C_3}(1-t^2)^{\beta-\lfloor \beta \rfloor }\log((1-t^2)^{-1})^{-2\theta}.$$
\end{proposition}

\begin{proof}Let $C^\star>0$ be the constant given by Proposition \ref{prop:loca3};
taking $\delta=C_\star^{-1}\log(\epsilon^{-1})^{-C_\star}$, we provide separate bounds on 
    \begin{align*}
     \int \|f^l(y)-f^l(tx)\|^2\mathds{1}_{\|y\|\leq 1-\delta}dp^{t,x}(y)        
\end{align*}   
and
\begin{align*}  
     \int \|f^l(y)-f^l(tx)\|^2\mathds{1}_{\|y\|\geq 1-\delta}dp^{t,x}(y).
\end{align*}
For the first term,  as for all $l\in \{0,...,\lfloor \beta\rfloor\}$ and $y\in B^d(0,1-\delta)$, we have $\|u^{(l+1)}(\|y\|^2)\|\leq C \delta^{-(K+\lfloor \beta \rfloor+1)}$ and $\|\nabla^la(y)-\nabla^la(tx)\|\leq K\|y-tx\|^{\beta-\lfloor \beta \rfloor }\log((\|y-tx\|\wedge 1/2)^{-1})^{-\theta}$, we deduce that
\begin{align*}
    \int \|f^l(y)-f^l(tx)\|^2\mathds{1}_{\|y\|\leq 1-\delta}dp^{t,x}(y) & \leq \int \frac{\delta^{-2(K+\lfloor \beta \rfloor+1)}\|y-tx\|^{2(\beta-\lfloor \beta \rfloor )}}{\log((\|y-tx\|\wedge 1/2)^{-1})^{2\theta}}\mathds{1}_{\|y\|\leq 1-\delta}dp^{t,x}(y).
\end{align*}
Let $\phi:\mathbb{R}^d\rightarrow \mathbb{R}^d$ be the $C^1$ diffeomorphism of the polar change of coordinates, we have 
\begin{align*}
    \int&\frac{\|y-tx\|^{2(\beta-\lfloor \beta \rfloor )}}{\log((\|y-tx\|\wedge 1/2)^{-1})^{2\theta}}dp^{t,x}(y)\\
     = &\int_0^\infty\int_0^\pi... \int_0^{2\pi} \frac{\rho^{2(\beta-\lfloor \beta \rfloor) }}{\log((\rho\wedge 1/2)^{-1})^{2\theta}}\frac{r(tx +\phi(\rho,\theta_1,...,\theta_{d-1}))}{Qtr(x)} \frac{e^{-\frac{\rho^2}{2(1-t^2)}}}{(2\pi(1-t^2))^{d/2}}|\text{det}(\nabla \phi(\rho,\theta_1,...,\theta_{d-1}))|d\rho d\theta_1...d\theta_{d-1}\\
    \leq & \frac{C}{Qtr(x)}\int_0^{\infty}\int_0^\pi... \int_0^{2\pi} \frac{\rho^{2(\beta-\lfloor \beta \rfloor) +d-1 }}{\log((\rho\wedge 1/2)^{-1})^{2\theta}}r(tx +\phi(\rho,\theta_1,...,\theta_{d-1}))\frac{e^{-\frac{\rho^2}{2(1-t^2)}}}{(1-t^2)^{d/2}}d\rho d\theta_1...d\theta_{d-1}\\
    \leq & \frac{(1-t^2)^{\beta-\lfloor \beta \rfloor}}{Qtr(x)}\int_0^{\infty}\int_0^\pi... \int_0^{2\pi} \frac{\rho^{2(\beta-\lfloor \beta \rfloor) +d-1 }}{\log((\rho(1-t^2)^{1/2}\wedge 1/2)^{-1})^{2\theta}}r(tx +\phi(\rho(1-t^2)^{1/2},\theta_1,...,\theta_{d-1}))e^{-\frac{\rho^2}{2}}d\rho d\theta_1...d\theta_{d-1}.
\end{align*}
Let $\gamma=(1-t^2)^{-1/4}$, we have
\begin{align*}
\int_0^{\infty}&\int_0^\pi... \int_0^{2\pi} \frac{\rho^{2(\beta-\lfloor \beta \rfloor) +d-1 }}{\log((\rho(1-t^2)^{1/2}\wedge 1/2)^{-1})^{2\theta}}r(tx +\phi(\rho(1-t^2)^{1/2},\theta_1,...,\theta_{d-1}))e^{-\frac{\rho^2}{2}}\mathds{1}_{\{\rho\leq \gamma\}}d\rho d\theta_1...d\theta_{d-1}\nonumber\\
= & \int_0^{\infty}\int_0^\pi... \int_0^{2\pi} \frac{\rho^{2(\beta-\lfloor \beta \rfloor) +d-1 }}{(\log((\rho)^{-1})+\log((1-t^2)^{-1/2}))^{2\theta}}r(tx +\phi(\rho(1-t^2)^{1/2},\theta_1,...,\theta_{d-1}))e^{-\frac{\rho^2}{2}}\mathds{1}_{\{\rho\leq \gamma\}}d\rho d\theta_1...d\theta_{d-1}\\
\leq & \frac{C}{\log((1-t^2)^{-1/2})^{2\theta}} \int_0^{\infty}\int_0^\pi... \int_0^{2\pi} \rho^{2(\beta-\lfloor \beta \rfloor) +d-1 }r(tx +\phi(\rho(1-t^2)^{1/2},\theta_1,...,\theta_{d-1}))e^{-\frac{\rho^2}{2}}\mathds{1}_{\{\rho\leq \gamma\}}d\rho d\theta_1...d\theta_{d-1}.
\end{align*}

On the other hand, 
using \eqref{align:boundbelowq} for $x$ and $t$ satisfying \eqref{eq:hyp14} we have 
\begin{align*}
  &\frac{1}{Q_tr(x)} \int_0^{\infty}\int_0^\pi... \int_0^{2\pi} \rho^{2(\beta-\lfloor \beta \rfloor) +d-1 }r(tx +\phi(\rho(1-t^2)^{1/2},\theta_1,...,\theta_{d-1}))e^{-\frac{\rho^2}{2}}\mathds{1}_{\{\rho\leq \gamma\}}d\rho d\theta_1...d\theta_{d-1}  \\
  &\leq C\int_0^{\infty}\int_0^\pi... \int_0^{2\pi} \rho^{2(\beta-\lfloor \beta \rfloor) +d-1 }\exp\left(\max_{z\in B^d(tx,(1-t^2)^{1/2})}u(\|z\|^2)-\min_{w\in B^d(tx,\gamma(1-t^2)^{1/2})}u(\|w\|^2)\right)\\
  &\quad \quad \times e^{-\frac{\rho^2}{2}}\mathds{1}_{\{\rho\leq \gamma\}}d\rho d\theta_1...d\theta_{d-1}
  \\
  &\leq C\int_0^{\infty}\int_0^\pi... \int_0^{2\pi} \rho^{2(\beta-\lfloor \beta \rfloor) +d-1 }\exp\left(\max_{z\in B^d(tx,(1-t^2)^{1/4})}u^{(1)}(\|z\|^2)C(1-t^2)^{1/4}\right)e^{-\frac{\rho^2}{2}}\mathds{1}_{\{\rho\leq \gamma\}}d\rho d\theta_1...d\theta_{d-1}
  \\
  &\leq C\int_0^{\infty}\int_0^\pi... \int_0^{2\pi} \rho^{2(\beta-\lfloor \beta \rfloor) +d-1 }e^{-\frac{\rho^2}{2}}\mathds{1}_{\{\rho\leq \gamma\}}d\rho d\theta_1...d\theta_{d-1}\\
  &\leq C,
\end{align*}
where we used that $\|tx\|^2\leq 1-C_2^{-1}\log(\epsilon^{-1}+1)^{-K}$ to obtain
\begin{align*}
\max_{z\in B^d(tx,(1-t^2)^{1/4})}u^{(1)}(\|z\|^2)&\leq C\Big(1-(\|tx\|+(1-t^2)^{1/4})^2\Big)^{-(K+1)}\\
& \leq C\log(\epsilon^{-1})^C\\
& \leq \left(C^\star \log(\epsilon^{-1})^{C^\star}\right)^{1/4}\\
&\leq (1-t^2)^{-1/4},
\end{align*}
for $C^\star>0$ large enough.
On the other hand,
\begin{align*}
    &\frac{1}{Q_tr(x)} \int_0^{\infty}\int_0^\pi... \int_0^{2\pi} \rho^{2(\beta-\lfloor \beta \rfloor) +d-1 }r(tx +\phi(\rho(1-t^2)^{1/2},\theta_1,...,\theta_{d-1}))e^{-\frac{\rho^2}{2}}\mathds{1}_{\{\rho> \gamma\}}d\rho d\theta_1...d\theta_{d-1}\\
    &\leq C\frac{e^{-\frac{\gamma^2}{2}}}{Q_tr(x)}\\
    & \leq C\exp\left(-\frac{(1-t^2)^{-1/2}}{2}\right)\exp\left(\max_{z\in B^d(0,\|tx\|+(1-t^2)^{1/2})}u^{(1)}(\|z\|^2)\right)\\
    & \leq C\exp\left(-\frac{(1-t^2)^{-1/2}}{2}+(1-(\|tx\|+(1-t^2)^{1/2})^2)^{-K+1}\right)\\ 
     & \leq C\exp\left(-\frac{(1-t^2)^{-1/2}}{4}\right).
\end{align*}
We can therefore conclude that
\begin{align*}
     \int \|f^l(y)-f^l(tx)\|^2\mathds{1}_{\|y\|\leq 1-\delta}dp^{t,x}(y) \leq C\delta^{-2(K+\lfloor\beta\rfloor +1)}(1-t^2)^{\beta-\lfloor \beta \rfloor}\log((1-t^2)^{-1})^{-2\theta}.     
\end{align*}   
Now, for $\xi=(1-t^2)^{1/4}$ we have
\begin{align*}
    \int \|f^l(y)-f^l(tx)\|^2\mathds{1}_{\|y\|\geq 1-\delta}dp^{t,x}(y)&= \int \|f^l(y)-f^l(tx)\|^2\mathds{1}_{\|y\|\geq 1-\delta}\mathds{1}_{\|y-tx\|\geq \xi}dp^{t,x}(y)\\
   & \leq  C\int \|f^l(y)-f^l(tx)\|^2\frac{\exp(-u(\|y\|^2))}{Qtr(x)(1-t^2)^{d/2}}\exp(-\frac{\xi^2}{1-t^2})d\lambda^d(y)\\
    &\leq C\frac{\exp(-\frac{\xi^2}{1-t^2})}{Q_tr(x)(1-t^2)^{d/2}}\int \|f^l(y)-f^l(tx)\|^2\exp(-u(\|y\|^2)/2)d\lambda^d(y)\\
    &\leq C\frac{\exp(-\frac{\xi^2}{1-t^2})}{Q_tr(x)(1-t^2)^{d/2}}\\
    & \leq C\exp(-\frac{1}{4(1-t^2)^{1/2}}).
\end{align*}

\end{proof}

Using Proposition \ref{prop:concentration3} we can give an estimate on the term $\int \|f^l(y)-\xi\|^2dp^{t,x}(y)$ in Proposition \ref{prop:firstboundongradsth2}. Taking $\xi=f^l(tx)$ with $x$ and $t$ satisfying
\begin{equation*}
1-t^2\leq C^{-1}_\star\log(\epsilon^{-1})^{-C_\star} \text{ and } \|tx\|^2\leq 1-C^{-1}_2\log(\epsilon^{-1}+1)^{-K},
\end{equation*}  
with given by $C^\star,C_2>0$ given Proposition \ref{prop:concentration3} we  obtain that 
\begin{equation}\label{eq:concentratioint}
\int \|f^l(y)-f^l(tx)\|^2dp^{t,x}(y)\leq C\log(\epsilon^{-1})^{C_3}\log((1-t^2)^{-1})^{-2\theta}.    
\end{equation}
To obtain estimates when $1-t^2\geq C^{-1}_\star\log(\epsilon^{-1})^{-C_\star}$, we first establish a bound on$\|X_t(x)\|$ in this setting.
\begin{proposition}\label{prop:localfortsmall}
    For $x\in B^d(0,\log(\epsilon^{-1}))$ and $t\in [0,1-C_\star^{-1}\log(\epsilon^{-1})^{-C_\star}]$ with $C^\star>0$ given by Proposition \ref{prop:loca3}, we have $$\|X_t(x)\|\leq C\log(\epsilon^{-1})^{C_2}.$$
\end{proposition}

\begin{proof}
We have 
\begin{align*}
  \partial_t \|X_t(x)\|^2&=2\left\langle X_t(x), V(t,X_t(x))\right\rangle\\
  &=2\left\langle X_t(x), \frac{1}{1-t^2} \int (y-tX_t(x))dp^{t,X_t(x)}(y)\right\rangle\\
  & \leq C\frac{1}{1-t^2} \|X_t(x)\|^{1/2}, 
\end{align*}
so using the non linear Gronwall's inequality \citep{dragomir2003some} we have 
$$ \|X_t(x)\|\leq \|x\| +\int_0^t \frac{1}{1-s^2}ds\leq  C\log(\epsilon^{-1})^{C_2}.$$
\end{proof}

\subsubsection{Bounding $\|\nabla^k X_1\|_\infty$ for $k\in\{2,...,\lfloor \beta \rfloor +1\}$}\label{sec:evertogether3}
Let us take $C^\star>0$ large enough so that  Propositions \ref{prop:loca3}, \ref{prop:concentHtheo2} and \ref{prop:firstboundongradsth2} hold. We first treat the case $1-t^2\geq C^{-1}_\star\log(\epsilon^{-1})^{-C_\star}$, i.e. $t$ is not too close to $1$. From the Faa Di Bruno formula we have 
\begin{align}\label{align:gradvtpetit}
    \nabla^{k} V(t,x)&=\frac{1}{t}\sum_{\pi \in \Pi_{k+1}} \frac{(-1)^{|\pi|+1}}{(Q_{t} r(x))^{|\pi|}} \prod_{B\in \pi} \nabla^{|B|} Q_{t} r(x)\nonumber \\
    & = \frac{1}{t}\sum_{\pi \in \Pi_{k+1}} (-1)^{|\pi|+1} \prod_{B\in \pi} \frac{1}{Q_{t} r(x)} \int \nabla^{|B|}_x\varphi^{t,x}(y)r(y)d\lambda^d(y)\nonumber\\
    & = \frac{t^{k}}{(1-t^2)^{k+1}}\sum_{\pi \in \Pi_{k+1}} (-1)^{|\pi|+1} \prod_{B\in \pi} \int F_{|B|}(y-tx)dp^{t,x}(y),
\end{align}
with $$ F_{|B|}(z):=(1-t^2)^{|B|}\exp(\frac{\|z\|^2}{1-t^2})\nabla^{|B|} \exp(-\frac{\|z\|^2}{1-t^2}).$$

Now, as $p^{t,x}$ has support in $B^d(0,1)$ for all $t\in [0,1)$ and $x\in \mathbb{R}^d$, we deduce from \eqref{align:gradvtpetit} and Proposition \ref{prop:localfortsmall} that 

\begin{align*}\|\nabla^{k} V(t,X_t(x))\|&=
    \frac{t^{k}}{(1-t^2)^{k+1}}\|\sum_{\pi \in \Pi_{k+1}} (-1)^{|\pi|+1} \prod_{B\in \pi} \int F_{|B|}(y-tX_t(x))dp^{t,X_t(x)}(y)\|\\
    & \leq \frac{C}{(1-t^2)^{k+1}}\sum_{\pi \in \Pi_{k+1}}  \prod_{B\in \pi} \int \sum_{i=1}^{|B|}(\|y\|^i+\|X_t(x)\|^i)dp^{t,X_t(x)}(y)\\
    & \leq  C\log(\epsilon^{-1})^{C_2}.
\end{align*}

Let us now consider the case $1-t^2\leq C^{-1}_\star\log(\epsilon^{-1})^{-C_\star}$. We are going to give estimates on different terms of the bound given by Lemma \ref{lemma:lemma1}. Let us take $l,m\in \{1,...,k-1\}$, from Proposition \ref{prop:firstboundongradsth2} we have  
$$ \|\int f^{l}(y)\nabla^2_x dp^{t,x}(y) \|
    \leq \frac{C}{1-t^2} \left(\int \|f^l(y)-f^l(tx)\|^2dp^{t,x}(y)\right)^{1/2}$$
    and 
$$ \|\int f^{l}(y)\nabla_x p^{t,x}(y)d\lambda^d(y)\| \|\int f^{m}(y)\nabla_x p^{t,x}(y)d\lambda^d(y) \|
    \leq \frac{C}{1-t^2} \int \|f^l(y)-f^l(tx)\|^2dp^{t,x}(y).$$

Therefore, from \eqref{eq:concentratioint} and lemma \ref{lemma:lemma1} we conclude that     

\begin{equation}
     \|\nabla^{k} V(t,X_t(x))\|\leq  C\log(\epsilon^{-1})^{C_2}\frac{1}{(1-t^2)\log((1-t^2)^{-1})^\theta}.
\end{equation}
We deduce that for all $t\in [0,1)$,
we have 
$$\int_0^t\|\nabla^{k} V(s,X_s(x))\|ds\leq  C\log(\epsilon^{-1})^{C_2},$$
so from Lemma \ref{lemma:addlemme} it concludes the bound on $\|\nabla^k X_t\|_\infty$ for $k\in\{2,...,\lfloor \beta \rfloor +1\}$. Using Lemma \ref{lemma:convergenceofmap}, we can then establish the bound on the derivatives of the transport map for Theorem \ref{theo:thetheo2} with the next result.

\begin{proposition}
    Under the assumptions of Theorem \ref{theo:thetheo}, the limit Langevin transport map $T$ satisfies that for all $\epsilon\in (0,1)$, the restriction of $T$ to $B^d(0,\log(\epsilon^{-1}))$ belongs to $ \mathcal{H}_{C\log(\epsilon^{-1})^{C_2}}^{\beta+1}\Big(B^d(0,\log(\epsilon^{-1})),B^d(0,1)\Big)$.
\end{proposition}

\section{Estimates on  $\|\nabla^{\lfloor \beta \rfloor+1} X_1\|_{\mathcal{H}^{\beta-\lfloor \beta \rfloor}}$}\label{sec:holdregulastderiv}
In this section we give the second part of the proof of Theorems \ref{theo:thetheo} and \ref{theo:thetheo2} which deals with the Hölder regularity of the $(\lfloor \beta \rfloor+1)$-derivative of the transport map. The proof's ideas are similar to the ones of Section \ref{sec:nablak}.
\subsection{General reasoning for the two cases} 
This section is dedicated to the results applying both to the proof of Theorems \ref{theo:thetheo} and \ref{theo:thetheo2}.

\subsubsection{Gronwall}
Let $(X_t)_{t\in [0,1)}$ be the solution of \eqref{eq:PDE} for $p$ satisfying either the assumptions of Theorems \ref{theo:thetheo} or \ref{theo:thetheo2}.
Take $k:=\lfloor \beta \rfloor +1$, $x_1,x_2\in \mathbb{R}^d$, $w\in \mathbb{S}^{d-1}$ and define $\alpha_t:=\|\left(\nabla^k X_t(x_1)-\nabla^k X_t(x_2)\right)w^k\|^2$, we have 
\begin{align*}
    \partial_t \left(\nabla^k X_t(x_1)-\nabla^k X_t(x_2)\right)  = &\sum_{\pi \in \Pi_{k}} \nabla^{|\pi|}V(t,X_t(x_1)) \prod_{B\in \pi} \nabla^{|B|} X_t(x_1)-\nabla^{|\pi|}V(t,X_t(x_2)) \prod_{B\in \pi} \nabla^{|B|} X_t(x_2)\\
     = & \sum_{\pi \in \Pi_{k}} \nabla^{|\pi|}V(t,X_t(x_1))\left( \prod_{B\in \pi} \nabla^{|B|} X_t(x_1)-\prod_{B\in \pi} \nabla^{|B|} X_t(x_2)\right)\\
    & + \sum_{\pi \in \Pi_{k}} \left(\nabla^{|\pi|}V(t,X_t(x_1)) -\nabla^{|\pi|}V(t,X_t(x_2))\right) \prod_{B\in \pi} \nabla^{|B|} X_t(x_2).
\end{align*}
As
$$\partial_t \alpha_t = 2\Big\langle\left(\nabla^k X_t(x_1)-\nabla^k X_t(x_2) \right)w^k,\partial_t \left(\nabla^k X_t(x_1)-\nabla^k X_t(x_2)\right)w^k\Big\rangle,$$
we obtain
\begin{align*}
    \partial_t \alpha_t \leq & 2\big(\lambda_1(t)+\lambda_2(t)\big)\alpha_t^{1/2}+2\lambda_3(t) \alpha_t,
\end{align*}
with
$$\lambda_1(t)=\sum_{\pi \in \Pi_{k}} \|\nabla^{|\pi|}V(t,X_t(x_1)) -\nabla^{|\pi|}V(t,X_t(x_2))\| \prod_{B\in \pi} \nabla^{|B|}\| X_t(x_2)\|,$$
$$\lambda_2(t)=\sum_{\pi \in \Pi_{k},|\pi|>1} \|\nabla^{|\pi|}V(t,X_t(x_1))\|\| \prod_{B\in \pi} \nabla^{|B|} X_t(x_1)-\prod_{B\in \pi} \nabla^{|B|} X_t(x_2)\|$$
and
$$\lambda_3(t)=\lambda_{\max}(\nabla V(t,X_t(x_1))).$$

Then, we deduce from the non linear Gronwall's inequality (Theorem 21 in \cite{dragomir2003some}) that 
\begin{equation}\label{eq:thegronwallholder}
\alpha_t^{1/2}\leq \int_0^t (\lambda_1(s)+\lambda_2(s))\exp\left(\int_s^t \lambda_3(z)dz\right)ds
\end{equation}
Let $\epsilon \in (0,1/2)$, we shall distinguish the cases of Theorems \ref{theo:thetheo} or \ref{theo:thetheo2} by taking $C_\epsilon:=C$ if we treat the case of Theorems \ref{theo:thetheo}  and $C_\epsilon:=C\log(\epsilon^{-1})^{C_2}$ if we treat the case of Theorems \ref{theo:thetheo2}. Likewise, we will consider that $x_1,x_2\in B^d(0,\log(\epsilon^{-1}))$ if we treat the case of Theorems \ref{theo:thetheo2}.

From Sections \ref{sec:nablak} and \ref{sec:prooftheo2} we have that for all $l \in \{1,...,k\}$, $$\|\nabla^{l}V(t,X_t(x_1))\|\leq C_\epsilon \Big((1-t^2)\log((1-t^2)^{-1})^\theta\Big)^{-1}$$ and $$\|\nabla^{l-1} X_t(x_1)-\nabla^{l-1} X_t(x_2)\|\leq C_\epsilon \|x_1-x_2\|,$$
so we deduce that for all $s\in [0,1)$
$$\lambda_2(s)\leq C_\epsilon \Big((1-s^2)\log((1-s^2)^{-1})^\theta\Big)^{-1}\|x_1-x_2\|$$
and
$$\int_s^t \lambda_3(z)dz\leq C_\epsilon.$$
Furthermore, we have 
$$\|\nabla^{l-1} X_t(x_2)\|\leq C_\epsilon$$
so we deduce that for all $s\in [0,1)$

and $$\lambda_1(s)\leq C_\epsilon \|\nabla^{k}V(s,X_s(x_1)) -\nabla^{k}V(s,X_s(x_2))\|+C_\epsilon\|x_1-x_2\|.$$
Therefore, using \eqref{eq:thegronwallholder} we finally obtain
\begin{equation}\label{eq:thegronwallholdersimple}
\alpha_t^{1/2}\leq C_\epsilon \int_0^t \|\nabla^{k}V(s,X_s(x_1)) -\nabla^{k}V(s,X_s(x_2))\|ds+C_\epsilon\|x_1-x_2\|.
\end{equation}
Let us distinguish two different cases.

\paragraph*{The case $k=1$.} We have
\begin{align*}
\nabla V(s,x) &= \frac{1}{1-s^2} \nabla \int(y-sx)dp^{s,x}(y)\\
& = \frac{1}{1-s^2} \left(-\text{Id} + \frac{s}{1-s^2}\int(y-sx)\otimes \nabla_x p^{t,x}(y)d\lambda^d(y)\right),
\end{align*}
so taking 
\begin{equation}
    R_0(s,X_s(x_1),X_s(x_2)) = \|\int\Big((y-sX_s(x_1))\nabla_x p^{s,X_s(x_1)}(y)-(y-sX_s(x_2))\nabla_x p^{s,X_s(x_2)}(y)\Big)d\lambda^d(y)\|
\end{equation}
and using \eqref{eq:thegronwallholdersimple} we deduce the following result.

\begin{proposition}\label{prop:rodebut} Let $x_1,x_2\in \mathbb{R}^d$
    such that for all $s\in [0,1)$ we have
    $$\frac{s}{1-s^2}R_0(s,X_s(x_1),X_s(x_2))\leq C_\epsilon\frac{\|X_s(x_2)-X_s(x_1)\|^{\beta-\lfloor \beta \rfloor}}{\log((1-s^2)^{-1})^\theta}.$$
    Then, for all $t\in [0,1)$ and  $w\in \mathbb{S}^{d-1}$ we have
    $$\|\left(\nabla X_t(x_1)-\nabla X_t(x_2)\right)w\|\leq C_\epsilon \|x_2-x_1\|^{\beta-\lfloor \beta \rfloor}.$$
\end{proposition}

\paragraph*{The case $k\geq2$.} Recall from \eqref{align:nablakplusdeux} that we have  
\begin{align*}
    \nabla^{k} V(s,x)&=s^{k-2}\sum_{\pi \in \Pi_{k-1}} (-1)^{|\pi|+1} \sum_{B\in \pi}\Biggl( \int f^{|B|}(y)\nabla^2_x dp^{s,x}(y) \prod_{C\in \pi\setminus \{B\}}\int f^{|C|}(y) dp^{s,x}(y)\nonumber\\
    & + \int f^{|B|}(y) \nabla_x dp^{s,x}(y)\sum_{C\in \pi\setminus \{B\}}\int f^{|C|}(y) \nabla_x dp^{s,x}(y)\prod_{D\in \pi\setminus \{B,C\}}\int f^{|D|}(y)  dp^{s,x}(y)\Biggl),
\end{align*}
with 
$f^l(y):=\nabla^l r(y)/r(y)$. Then, we deduce that
\begin{align}\label{align:vkx1moinsvkx2}
    &\|\nabla^{k} V(s,X_s(x_1))-\nabla^{k} V(s,X_s(x_2))\|\leq C\sum_{i=1}^3 R_i(s,X_s(x_1),X_s(x_2))
\end{align}
with
\begin{align*}
    R_1&(s,X_s(x_1),X_s(x_2))\\
    := & \sum_{l=1}^{k-1} \|\int f^{l}(y)(\nabla^2_x dp^{s,X_s(x_1)}(y)-\nabla^2_x dp^{s,X_s(x_2)}(y))\| \prod_{j=1}^{k-2}\left(\|\int f^{j}(y) dp^{s,X_s(x_1)}(y)\|\vee 1\right)^{k-2},
\end{align*}
\begin{align*}
R_2(s,X_s(x_1),X_s(x_2)):= &  \sum_{l=1}^{k-1} \|\int f^{l}(y) (\nabla_x dp^{s,X_s(x_1)}(y)-\nabla_x dp^{s,X_s(x_2)}(y))\|\\
    & \times \sum_{j=1}^{k-2}\|\int f^{j}(y) \nabla_x dp^{s,X_s(x_1)}(y)\|\prod_{i=1}^{k-2}\left(\|\int f^{i}(y)  dp^{s,X_s(x_1)}(y)\|\nonumber\vee 1 \right)^{k-2},
\end{align*}
\begin{align*}
    R_3&(s,X_s(x_1),X_s(x_2))\\
    := & \sum_{j=1}^{k-2} \|\int f^{j}(y)  (dp^{s,X_s(x_1)}(y)-dp^{s,X_s(x_2)}(y))\|\prod_{i=1}^{k-3}\left(\|\int f^{i}(y)  dp^{s,X_s(x_1)}(y)\|\vee 1\right)^{k-3}\\
    & \times\Biggl(\sum_{l=1}^{k-1}  \|\int f^{l}(y)\nabla^2_x dp^{s,X_s(x_2)}(y)\|+\|\int f^{l}(y)\nabla_x dp^{s,X_s(x_2)}(y)\|\sum_{r=1}^{k-1}\|\int f^{r}(y)\nabla_x dp^{s,X_s(x_2)}(y)\|\Biggl).
\end{align*}
Using \eqref{eq:thegronwallholdersimple} we deduce the following result.

\begin{proposition}\label{prop:boundrjdebut} Let $k\geq 2$, $x_1,x_2\in \mathbb{R}^d$
    such that for all $s\in [0,1)$ and $i\in \{1,2,3\}$ we have
$$R_i(s,X_s(x_1),X_s(x_2))\leq \frac{C_\epsilon}{(1-s^2)\log((1-s^2)^{-1})^\theta} \|X_s(x_2)-X_s(x_1)\|^{\beta-\lfloor \beta \rfloor},$$
    Then, for all $t\in [0,1)$ and  $w\in \mathbb{S}^{d-1}$ we have
    $$\|\left(\nabla^k X_t(x_1)-\nabla^k X_t(x_2)\right)w^k\|\leq C_\epsilon \|x_2-x_1\|^{\beta-\lfloor \beta \rfloor}.$$
\end{proposition}

\subsubsection{Bound on $\|\nabla^{\lfloor \beta \rfloor+1} V(t,X_t(x_1))-\nabla^{\lfloor \beta \rfloor+1} V(t,X_t(x_2))\|$ for $(1-t^2)^{1/2}\leq \| X_t(x_1)- X_t(x_2)\|$ }\label{sec:boundnablavx1x2}
Let us give general estimates on $(R_j)_{j=0,1,2,3}$ that we will then bound differently for each of the cases in the following sections. We shall distinguish whether $\| X_t(x_1)- X_t(x_2)\|\geq (1-t^2)^{1/2}$ or not as when it is the case, we do not have to compare $p^{t,X_t(x_1)}$ and $p^{t,X_t(x_2)}$ as they each concentrate the mass at a scale $(1-t^2)^{1/2}$. 

Let us take $\theta_1:= X_t(x_1)$, $\theta_2:=X_t(x_2)$ and $t\in [0,1)$ such that $(1-t^2)^{1/2}\leq \|\theta_1-\theta_2\|$. For the term $R_0$, we have
\begin{align}
R_0&(t,\theta_1,\theta_2) = \|\int\Big((y-t\theta_2)\nabla_x p^{t,\theta_2}(y)-(y-t\theta_2)\nabla_x p^{t,\theta_2}(y)\Big)d\lambda^d(y)\|\nonumber\\
         \leq & \sum_{i=1}^2 \|\int(y-t\theta_i)\nabla_x p^{t,\theta_i}(y)-(1-t^2)\text{Id}\|\label{eq:roboundtpetit2}\\
         \leq & \frac{t}{1-t^2}\sum_{i=1}^2 \|\int(y-t\theta_i)\left(y-\int zdp^{t,\theta_i}(z)\right)d p^{t,\theta_i}(y)-(1-t^2)\text{Id}\|\nonumber\\
         \leq & \frac{t}{1-t^2}\sum_{i=1}^2 \|\int(y-t\theta_i)^{\otimes2}d p^{t,\theta_i}(y)-(1-t^2)\text{Id}\|+\|\int(y-t\theta_i)\left(t\theta_i-\int zdp^{t,\theta_i}(z)\right)d p^{t,\theta_i}(y)\|\nonumber\\
         \leq & \frac{t}{1-t^2}\sum_{i=1}^2\|\int\left(y-t\theta_i\right)^{\otimes2}\left( \frac{r(y)}{\int r(z)\varphi^{t,\theta_i}(z)d\lambda^d(z)}-1\right)\varphi^{t,\theta_i}(y)d\lambda^d(y)\|+\|\int(y-t\theta_i)d p^{t,\theta_i}(y)\|^2\label{eq:roboundtpetit}.
\end{align}

Now, for the terms $(R_j)_{j=1,2,3}$,
we have from Proposition \ref{prop:firstboundongradsth1} and \ref{prop:firstboundongradsth2} that  for all $f\in \mathbb{L}^2$, $t\in[ 0,1)$, $i\in \{1,2\}$ and $\xi_1,\xi_2\in \mathbb{R}^d$,
\begin{align}\label{align:x1x2grand}
    \|\int f(y)\big(\nabla^i_x p^{t,\theta_1}(y)-\nabla^i_x p^{t,\theta_2}(y)\big)d\lambda^d(y)\|& \leq \sum_{j=1}^2 \|\int f(y)\nabla^i_x p^{t,\theta_j}(y)\|d\lambda^d(y)\nonumber\\
    & \leq \frac{C}{(1-t^2)^{i/2}}\sum_{j=1}^2\left(\int \|f(y)-\xi_i\|^2dp^{t,\theta_j}(y)\right)^{1/2}.
\end{align}
In the case of Theorem \ref{theo:thetheo2}, we have from Proposition \ref{prop:loca3} that if $1-t^2\leq C_\star^{-1}\log(\epsilon^{-1})^{-C_\star}$, then 
$$\|t\theta_1\|^2+\|t\theta_2\|^2  \leq 1-C^{-1}\log(\epsilon^{-1}+1)^{-K},$$
which gives in particular that for all $m\in 1,...,k$,
$$\|\int f^{m}(y) dp^{s,\theta_1}(y)\|\leq C_\epsilon.$$
Therefore, whether we are in the case of Theorem \ref{theo:thetheo} or \ref{theo:thetheo2}, for $t$ large enough and $\xi_1,...,\xi_{k-1}\in \mathbb{R}^d$, we have
\begin{align*}
    R_1(s,\theta_1,\theta_2) & \leq \frac{C}{1-t^2}\sum_{j=1}^2\sum_{l=1}^{k-1}\left(\int \|f^l(y)-\xi_l\|^2dp^{t,\theta_j}(y)\right)^{1/2} \prod_{m=1}^{k-2}\left(\|\int f^{m}(y) dp^{s,\theta_1}(y)\|\vee 1\right)\\
    & \leq \frac{C_\epsilon}{1-t^2}\sum_{j=1}^2\sum_{l=1}^{k-1}\left(\int \|f^l(y)-\xi_l\|^2dp^{t,\theta_j},(y)\right)^{1/2}.
\end{align*}
Likewise,
\begin{align*}
R_2(s,\theta_1,\theta_2)\leq & \frac{C_\epsilon}{1-t^2}\sum_{j=1}^2  \sum_{l=1}^{k-1}\left(\int \|f^l(y)-\xi_l\|^2dp^{t,\theta_j}(y)\right)^{1/2} \sum_{j=1}^{k-2}\left(\int \|f^j(y)-\xi_l\|^2dp^{t,\theta_1}(y)\right)^{1/2}
\end{align*}
and
\begin{align*}
    R_3(s,\theta_1,\theta_2)\leq & \frac{C_\epsilon}{1-t^2} \sum_{j=1}^{k-2} \|\int f^{j}(y)  (dp^{s,\theta_1}(y)-dp^{s,\theta_2}(y))\|\sum_{l=1}^{k-1}\left(\int \|f^l(y)-\xi_l\|^2dp^{t,\theta_2}(y)\right)^{1/2} \\
    & \times \Biggl(1+\sum_{r=1}^{k-1}\left(\int \|f^r(y)-\xi_r\|^2dp^{t,\theta_2}(y)\right)^{1/2}\Biggl).
\end{align*}

\subsubsection{Bound on $\|\nabla^{\lfloor \beta \rfloor+1} V(t,X_t(x_1))-\nabla^{\lfloor \beta \rfloor+1} V(t,X_t(x_2))\|$ for $(1-t^2)^{1/2}> \| X_t(x_1)- X_t(x_2)\|$ }\label{sec:boundnablavx1x22}

Let us give general estimates on $(R_j)_{j=0,1,2,3}$ that we will then bound differently for each of the cases in the following sections. Let $f\in L^2(\mathbb{R}^d)$, $\xi\in \mathbb{R}^d$ and take $\theta_1:= X_t(x_1)$, $\theta_2:=X_t(x_2)$ such that $(1-t^2)^{1/2}> \| X_t(x_1)- X_t(x_2)\|$.

For the term $R_0$ we have 
\begin{align}\label{align:roboundtgrand}  R_0&(t,\theta_1,\theta_2)\nonumber\\ =& \|\int\Big((y-t\theta_1)\nabla_x p^{t,\theta_1}(y)-(y-t\theta_2)\nabla_x p^{t,\theta_2}(y)\Big)d\lambda^d(y)\|\nonumber\\
    \leq & \|\theta_1-\theta_2\|\|\int \nabla_x p^{t,\theta_1}(y)d\lambda^d(y)\|\nonumber\\
    & + \|\int(y-t\theta_2)\Big(\nabla_x p^{t,\theta_1}(y)-\nabla_x p^{t,\theta_2}(y)\Big)d\lambda^d(y)\|\nonumber\\
    \leq & \frac{t}{1-t^2}\|\theta_1-\theta_2\|\|\int H(y,\theta_1) dp^{t,\theta_1}(y)\|\nonumber\\
    & + \frac{t}{1-t^2}\|\int(y-t\theta_2)\Big((H(y,\theta_1)-H(y,\theta_2))p^{t,\theta_1}(y)+ H(y,\theta_2)(p^{t,\theta_1}(y)-p^{t,\theta_2}(y))\Big)d\lambda^d(y)\|\nonumber\\
    \leq & \frac{t}{1-t^2}\sum_{i=1}^2\|\int H(y,\theta_i) dp^{t,\theta_i}(y)\|\left(\|\theta_1-\theta_2\|+\|\int w (p^{t,\theta_1}(w)-p^{t,\theta_2}(w))d\lambda^d(w)\|\right)\nonumber\\
    & + \frac{t}{1-t^2}\|\int (y-t\theta_2)H(y,\theta_2)(p^{t,\theta_1}(y)-p^{t,\theta_2}(y))d\lambda^d(y)\|.
\end{align}

Now, in order to provide a bound on $R_2$, let us start by deriving the term 
$$\int f(y)\big(\nabla_x p^{t,\theta_1}(y)-\nabla_x p^{t,\theta_2}(y)\big)d\lambda^d(y),$$
recalling the notation from Section \ref{sec:nablak}
$$H(y,x):=y-\int wp^{t,x}(w)d\lambda^d(w).$$
We have
\begin{align*}
 &\frac{1-t^2}{t}\int f(y)\big(\nabla_x p^{t,\theta_1}(y)-\nabla_x p^{t,\theta_2}(y)\big)d\lambda^d(y)=\int f(y)\left(H(y,\theta_1)p^{t,\theta_1}(y)- H(y,\theta_2)p^{t,\theta_2}(y)\right)d\lambda^d(y)\\
 & =\int (f(y)-\xi)\Big((H(y,\theta_1)-H(y,\theta_2))p^{t,\theta_1}(y)+ H(y,\theta_2)(p^{t,\theta_1}(y)-p^{t,\theta_2}(y))\Big)d\lambda^d(y),
\end{align*}
so 
\begin{align}\label{align:boundR2}
    \frac{1-t^2}{t}&\|\int f(y)\big(\nabla_x p^{t,\theta_1}(y)-\nabla_x p^{t,\theta_2}(y)\big)d\lambda^d(y)\|\nonumber\\
    \leq & \left(\int \|f(y)-\xi\|^2dp^{t,\theta_1}(y)\right)^{1/2}\left(\int \|H(y,\theta_1)-H(y,\theta_2)\|^{2}dp^{t,\theta_1}(y)\right)^{1/2}\nonumber\\
    & + \|\int (f(y)-\xi)H(y,\theta_2)(p^{t,\theta_1}(y)-p^{t,\theta_2}(y))d\lambda^d(y)\|\nonumber\\
    \leq & \left(\int \|f(y)-\xi\|^2dp^{t,\theta_1}(y)\right)^{1/2}\|\int w (p^{t,\theta_1}(w)-p^{t,\theta_2}(w))d\lambda^d(w)\|\nonumber\\
    & + \|\int (f(y)-\xi)H(y,\theta_2)(p^{t,\theta_1}(y)-p^{t,\theta_2}(y))d\lambda^d(y)\|.
\end{align}

In order to provide a bound on
$R_1$, let us now derive the term 
$$\int f(y)\big(\nabla_x^2 p^{t,\theta_1}(y)-\nabla_x^2 p^{t,\theta_2}(y)\big)d\lambda^d(y).$$
We have
\begin{align*}
\frac{(1-t^2)^2}{t^2}&\int f(y)\big(\nabla^2_x p^{t,\theta_1}(y)-\nabla^2_x p^{t,\theta_2}(y)\big)d\lambda^d(y)\\
     = &\int f(y)\Biggl( \left(H(y,\theta_1)^{\otimes 2}- \int H(w,\theta_1)^{\otimes 2} dp^{t,\theta_1}(w)\right)p^{t,\theta_1}(y)\\
    & - \left(H(y,\theta_2)^{\otimes 2}-\int H(w,\theta_2)^{\otimes 2} dp^{t,\theta_2}(w)\right)p^{t,\theta_2}(y)\Biggl)d\lambda^d(y)\\
     = & A_1 +A_2,
\end{align*}
with 
\begin{align*}
    A_1 =&\int f(y)\Biggl( H(y,\theta_1)^{\otimes 2}- \int H(w,\theta_1)^{\otimes 2} dp^{t,\theta_1}(w)\\
    & - \left(H(y,\theta_2)^{\otimes 2}-\int H(w,\theta_2)^{\otimes 2} dp^{t,\theta_1}(w)\right)\Biggl)dp^{t,\theta_1}(y)
\end{align*}
and 
\begin{align*}
    A_2 =& \int f(y)\Biggl(\left(H(y,\theta_2)^{\otimes 2}-\int H(w,\theta_2)^{\otimes 2} dp^{t,\theta_1}(w)\right)p^{t,\theta_1}(y)\\
    &-\left(H(y,\theta_2)^{\otimes 2}-\int H(w,\theta_2)^{\otimes 2} dp^{t,\theta_2}(w)\right)p^{t,\theta_2}(y)\Biggl)d\lambda^d(y).
\end{align*}
Then applying Propositions \ref{prop:concentHtheo1} or \ref{prop:concentHtheo2} whether we are in the setting of Theorem \ref{theo:thetheo} or \ref{theo:thetheo2}, we get 
\begin{align}\label{eq:premiertermlipgrand}
    \|&A_1\| \leq  \left(\int \|f(y)-\xi\|^2dp^{t,\theta_1}(y)\right)^{1/2}\nonumber\\
    & \times\sum_{ij} \Biggl(\int |H(y,\theta_1)_{ij}^{\otimes2}-H(y,\theta_2)_{ij}^{\otimes2} -\int (H(y,\theta_1)_{ij}^{\otimes2}-H(y,\theta_2)_{ij}^{\otimes2})dp^{t,\theta_1}(w)|^2dp^{t,\theta_1}(y)\Biggl)^{1/2}\nonumber\\
    \leq & C(1-t^2)^{1/2}\left(\int \|f(y)-\xi\|^2dp^{t,\theta_1}(y)\right)^{1/2}\nonumber\\
    & \times \sum_{ij}\left(\int (|H(y,\theta_1)_i-H(y,\theta_2)_i|^2+|H(y,\theta_1)_j-H(y,\theta_2)_j|^2)dp^{t,\theta_1}(y)\right)^{1/2}\nonumber\\
    \leq & C(1-t^2)^{1/2}\left(\int \|f(y)-\xi\|^2dp^{t,\theta_1}(y)\right)^{1/2} \|\int z(p^{t,\theta_1}(z)-p^{t,\theta_2}(z))d\lambda^d(z)\|.
\end{align}

On the other hand
\begin{align*}
    A_2= & \int f(y)\Biggl(\left(H(y,\theta_2)^{\otimes 2}-\int H(w,\theta_2)^{\otimes 2} dp^{t,\theta_1}(w)\right)p^{t,\theta_1}(y)\\
    &-\left(H(y,\theta_2)^{\otimes 2}-\int H(w,\theta_2)^{\otimes 2} dp^{t,\theta_2}(w)\right)p^{t,\theta_2}(y)\Biggl)d\lambda^d(y)\\
    = & \int \left((f(y)-f(y+t(\theta_2-\theta_1)\right)\left(H(y,\theta_2)^{\otimes 2}-\int H(w,\theta_2)^{\otimes 2} dp^{t,\theta_1}(w)\right)dp^{t,\theta_1}(y)\\
    + & \int f(y+t(\theta_2-\theta_1)\left(H(y,\theta_2)^{\otimes 2}-\int H(w,\theta_2)^{\otimes 2} dp^{t,\theta_1}(w)\right)dp^{t,\theta_1}(y)\\
    &- \int f(y)\left(H(y,\theta_2)^{\otimes 2}-\int H(w,\theta_2)^{\otimes 2} dp^{t,\theta_2}(w)\right)dp^{t,\theta_2}(y)\\
    & = B_1+B_2,
\end{align*}
with 
\begin{align*}
    B_1 =\int \left((f(y)-f(y+t(\theta_2-\theta_1)\right)\left(H(y,\theta_2)^{\otimes 2}-\int H(w,\theta_2)^{\otimes 2} dp^{t,\theta_1}(w)\right)dp^{t,\theta_1}(y)
\end{align*}
and 
\begin{align*}
    B_2 =& \int (f(y)-\xi)\Biggl(\left( H(y-t(\theta_2-\theta_1),\theta_2)^{\otimes 2} -\int H(w,\theta_2)^{\otimes 2} dp^{t,\theta_1}(w)\right)p^{t,\theta_1}(y-t(\theta_2-\theta_1))\\
     & - \left(H(y,\theta_2)^{\otimes 2}-\int H(w,\theta_2)^{\otimes 2} dp^{t,\theta_2}(w)\right)p^{t,\theta_2}(y)\Biggl)d\lambda^d(y).
\end{align*}

Now from Proposition \ref{prop:firstboundongradsth1} and \ref{prop:firstboundongradsth2} we have
\begin{align}\label{align:B1}
    \|B_1\|  
     \leq& \left(\int \|f(y)-f(y+t(\theta_2-\theta_1))\|^2dp^{t,\theta_1}(y)\right)^{1/2}\nonumber\\
     & \times \sum_{ij}\left(\int|H(y,\theta_2)_{ij}^{\otimes2}-\int H(y,\theta_2)_{ij}^{\otimes2} dp^{t,\theta_1}(w)|^2dp^{t,\theta_1}(y)\right)^{1/2}\nonumber\\
     \leq& C (1-t^2)^{1/2}\left(\int \|f(y)-f(y+t(\theta_2-\theta_1))\|^2dp^{t,\theta_1}(y)\right)^{1/2}\nonumber\\
     & \times \sum_{ij}\left(\int(|H(y,\theta_2)_i|^2+|H(y,\theta_2)_j|^2)dp^{t,\theta_1}(y)\right)^{1/2}\nonumber\\
     \leq& C (1-t^2)^{1/2}\left(\int \|f(y)-f(y+t(\theta_2-\theta_1))\|^2dp^{t,\theta_1}(y)\right)^{1/2}\nonumber\\
     & \times \sum_{ij}\left(\int(|H(y,\theta_1)_i|^2+|H(y,\theta_1)_j|^2+\|\int w (p^{t,\theta_1}(w)-p^{t,\theta_2}(w))d\lambda^d(w)\|^2)dp^{t,\theta_1}(y)\right)^{1/2}\nonumber\\
     \leq& C (1-t^2)\left(\int \|f(y)-f(y+t(\theta_2-\theta_1))\|^2dp^{t,\theta_1}(y)\right)^{1/2}\nonumber\\
     &+ C (1-t^2)^{1/2}\left(\int \|f(y)-f(y+t(\theta_2-\theta_1))\|^2dp^{t,\theta_1}(y)\right)^{1/2}\nonumber\\
     &\times \|\int w (p^{t,\theta_1}(w)-p^{t,\theta_2}(w))d\lambda^d(w)\|.
\end{align}
On the other hand,
\begin{align*}
    B_2 =& \int (f(y)-\xi)\Biggl(\left( H(y-t(\theta_2-\theta_1),\theta_2)^{\otimes 2} -\int H(w,\theta_2)^{\otimes 2} dp^{t,\theta_1}(w)\right)p^{t,\theta_1}(y-t(\theta_2-\theta_1))\\
     & - \left(H(y,\theta_2)^{\otimes 2}-\int H(w,\theta_2)^{\otimes 2} dp^{t,\theta_2}(w)\right)p^{t,\theta_2}(y)\Biggl)d\lambda^d(y)\\
     =& \int (f(y)-\xi)\left( H(y-t(\theta_2-\theta_1),\theta_2)^{\otimes 2} -\int H(w,\theta_2)^{\otimes 2} dp^{t,\theta_1}(w)\right)p^{t,\theta_1}(y-t(\theta_2-\theta_1))d\lambda^d(y)\\
     &+ \int (f(y)-\xi)\left( H(y,\theta_2)^{\otimes 2} -H(y,\theta_2)^{\otimes 2}\right)p^{t,\theta_1}(y-t(\theta_2-\theta_1))d\lambda^d(y)\\
      &+ \int (f(y)-\xi)\left( \int H(w,\theta_2)^{\otimes 2} dp^{t,\theta_2}(w) -\int H(w,\theta_2)^{\otimes 2} dp^{t,\theta_2}(w)\right)p^{t,\theta_1}(y-t(\theta_2-\theta_1))d\lambda^d(y)\\
     & -  \int (f(y)-\xi)\left(H(y,\theta_2)^{\otimes 2}-\int H(w,\theta_2)^{\otimes 2} dp^{t,\theta_2}(w)\right)p^{t,\theta_2}(y)d\lambda^d(y)\\
      =& \int (f(y-\xi)\left(H(y-t(\theta_2-\theta_1)),\theta_2)^{\otimes 2}-H(y,\theta_2)^{\otimes 2}\right)dp^{t,\theta_1}(y-t(\theta_2-\theta_1)))\\
      & +\int (f(y)-\xi)\int H(w,\theta_2)^{\otimes 2}(dp^{t,\theta_2}(w)- dp^{t,\theta_1}(w))\Biggl)dp^{t,\theta_1}(y-t(\theta_2-\theta_1)))\\
     & + \int (f(y)-\xi)\left(H(y,\theta_2)^{\otimes 2}-\int H(w,\theta_2)^{\otimes 2} dp^{t,\theta_2}(w)\right)(p^{t,\theta_1}(y-t(\theta_2-\theta_1)))-p^{t,\theta_2}(y))d\lambda^d(y)\\
     =&D_1+D_2+D_3,
\end{align*}
with
\begin{align*}
    D_1=\int (f(y+t(\theta_2-\theta_1))-\xi)\left(H(y,\theta_2)^{\otimes 2}-H(y+t(\theta_2-\theta_1),\theta_2)^{\otimes 2}\right)dp^{t,\theta_1}(y),
\end{align*}
\begin{align*}
    D_2=\int (f(y+t(\theta_2-\theta_1))-\xi)\int H(w,\theta_2)^{\otimes 2}(dp^{t,\theta_1}(w)- dp^{t,\theta_1}(w))dp^{t,\theta_1}(y),
\end{align*}
\begin{align*}
    D_3=&\int (f(y)-\xi)\left(H(y,\theta_2)^{\otimes 2}-\int H(w,\theta_2)^{\otimes 2} dp^{t,\theta_2}(w)\right)(p^{t,\theta_1}(y-t(\theta_2-\theta_1)))-p^{t,\theta_1}(y))d\lambda^d(y)\\
     & + \int (f(y)-\xi)\left(H(y,\theta_2)^{\otimes 2}-\int H(w,\theta_2)^{\otimes 2} dp^{t,\theta_2}(w)\right)(p^{t,\theta_1}(y)-p^{t,\theta_2}(y))d\lambda^d(y).
\end{align*}

We have
\begin{align*}
    \Big(  H&(y+t(\theta_2-\theta_1),\theta_2)^{\otimes 2} - H(y,\theta_2)^{\otimes 2}\Big)_{i,j}\\
     =&\left(y_i+t(\theta_2-\theta_1)_i-\int w_i dp^{t,\theta_2}(w)\right)\left(y_j+t(\theta_2-\theta_1)_j-\int w_j dp^{t,\theta_2}(w)\right)\\
    & - \left(y_i-\int w_i dp^{t,\theta_2}(w)\right)\left(y_j-\int w_j dp^{t,\theta_2}(w)\right)\\
     = &t(\theta_2-\theta_1)_i\left(y_j-\int w_j dp^{t,\theta_2}(w)+y_i-\int w_i dp^{t,\theta_2}(w)+t(\theta_2-\theta_1)_i\right),
\end{align*}
so 
\begin{align}\label{align:H2lipplust}
    &\|H(y-t(\theta_2-\theta_1),\theta_2)^{\otimes 2} - H(y,\theta_2)^{\otimes 2}\|\leq C\|\theta_2-\theta_1\|\left(\|y-\int w dp^{t,\theta_2}(w)\|+\|\theta_2-\theta_1\|\right),
\end{align}
which gives 
\begin{align}\label{align:D1}
    \|D_1\|&\leq C\left(\int \|f(y+t(\theta_2-\theta_1))-\xi\|^2dp^{t,\theta_1}(y)\right)^{1/2}\nonumber\\
    &\times \|\theta_2-\theta_1\|\Biggl(\left(\int \|y-\int w dp^{t,\theta_2}(w)\|^2dp^{t,\theta_1}(y)\right)^{1/2}+ \|\theta_1-\theta_2\|\Biggl)\nonumber\\
    &\leq C\|\theta_2-\theta_1\|(1-t^2)^{1/2}\left(\int \|f(y+t(\theta_2-\theta_1))-\xi\|^2dp^{t,\theta_1}(y)\right)^{1/2},
\end{align}
recalling that $\|\theta_1-\theta_2\|<(1-t^2)^{1/2}$.
Furthermore,
\begin{align}\label{align:D2}
    \|D_2\|
    \leq & \left(\int \|f(y+t(\theta_2-\theta_1))-\xi\|^2dp^{t,\theta_1}(y)\right)^{1/2}\|\int H(w,\theta_2)^{\otimes 2}(dp^{t,\theta_1}(w)- dp^{t,\theta_1}(w))\|
\end{align}
and
\begin{align}\label{align:D3}
    \|D_3\|\leq &C \|\int (f(y)-\xi)\left(H(y,\theta_2)^{\otimes 2}-\int H(w,\theta_2)^{\otimes 2} dp^{t,\theta_2}(w)\right)(p^{t,\theta_1}(y)-p^{t,\theta_2}(y))d\lambda^d(y)\|.
\end{align}

From the previous derivations, we deduce that we have to bound the quantity
\begin{equation}\label{eq:leggeneral}
\|\int g(y) (p^{t,\theta_1}(y)-p^{t,\theta_2}(y))d\lambda^d(z)\|,
\end{equation}
for $g(y)$ being either $y,H(y,\theta_2)^{\otimes 2},(f(y)-\xi)\left(H(y,\theta_2)^{\otimes 2}-\int H(w,\theta_2)^{\otimes 2} dp^{t,\theta_2}(w)\right)$ or $(f(y)-\xi)H(y,\theta_2)$.
Let us give a general bound on \eqref{eq:leggeneral} writing $g(y)=(l(y)-\zeta)h(y)$.

\begin{proposition}\label{prop:generalboundlhp}
Let $l,h\in L^2_p(\mathbb{R}^d)$.
Then for $\theta_1,\theta_2\in \mathbb{R}^d$ with $\|\theta_1-\theta_2\|<1$, we have
\begin{align*}
   \|\int & (l(y)-\zeta)h(y)(p^{t,\theta_1}(y)-p^{t,\theta_2}(y))d\lambda^d(y)\|\\
   \leq & C \|\int(l(y)-l(y+t(\theta_2-\theta_1))h(y)dp^{t,\theta_1}(y)\|+ |\frac{Q_t(\theta_2)}{Q_t(\theta_1)}-1|\|\int (l(y)-\zeta)h(y-t(\theta_2-\theta_1))dp^{t,\theta_2}(y)\|\\
   &+\|\int (l(y)-\zeta)h(y-t(\theta_2-\theta_1))\frac{r(y-t(\theta_2-\theta_1))-r(y)}{r(y)}dp^{t,\theta_2}(y)\|\\
    & + \|\int (l(y)-\zeta)\left(h(y+t(\theta_2-\theta_1))-h(y)\right)dp^{t,\theta_2}(y)\|.
\end{align*}
\end{proposition}
\begin{proof}
    We have
\begin{align*}
   \int & (l(y)-\zeta)h(y)(p^{t,\theta_1}(y)-p^{t,\theta_2}(y))d\lambda^d(y)\\
   = &\int (l(y-t(\theta_2-\theta_1))-\zeta)h(y-t(\theta_2-\theta_1))p^{t,\theta_1}(y-t(\theta_2-\theta_1)d\lambda^d(y)\\
   & - \int (l(y)-\zeta)h(y)p^{t,\theta_2}(y))d\lambda^d(y)\\
    = & \int (l(y-t(\theta_2-\theta_1))-l(y))h(y-t(\theta_2-\theta_1))r(y-t(\theta_2-\theta_1))\frac{\varphi^{t,\theta_2}(y)}{Q_t(\theta_1)}d\lambda^d(y)\\
    & + \int (l(y)-\zeta)h(y-t(\theta_2-\theta_1))\left(\frac{r(y-t(\theta_2-\theta_1))}{Q_t(\theta_1)}-\frac{r(y)}{Q_t(\theta_2)} \right)\varphi^{t,\theta_2}(y)d\lambda^d(y)\\
    & + \int (l(y)-\zeta)\left(h(y-t(\theta_2-\theta_1))-h(y)\right)dp^{t,\theta_2}(y)\\
    = & \int (l(y-t(\theta_2-\theta_1))-l(y))h(y-t(\theta_2-\theta_1))r(y-t(\theta_2-\theta_1))\frac{\varphi^{t,\theta_2}(y)}{Q_t(\theta_1)}d\lambda^d(y)\\
    & + \int (l(y)-\zeta)h(y-t(\theta_2-\theta_1))\left(r(y-t(\theta_2-\theta_1))-r(y)\right)\frac{\varphi^{t,\theta_2}(y)}{Q_t(\theta_1)}d\lambda^d(y)\\
    & + \left(\frac{Q_t(\theta_2)}{Q_t(\theta_1)}-1 \right)\int (l(y)-\zeta)h(y-t(\theta_2-\theta_1))dp^{t,\theta_2}(y)\\
    & + \int (l(y)-\zeta)\left(h(y-t(\theta_2-\theta_1))-h(y)\right)dp^{t,\theta_2}(y)
\end{align*}
so
\begin{align*}
    \int & (l(y)-\zeta)h(y)(p^{t,\theta_1}(y)-p^{t,\theta_2}(y))d\lambda^d(y)\\
    = & \int (l(y)-l(y+t(\theta_2-\theta_1)))h(y)dp^{t,\theta_1}(y)+ \frac{Q_t(\theta_2)}{Q_t(\theta_1)}\int (l(y)-\zeta)h(y-t(\theta_2-\theta_1))\frac{r(y-t(\theta_2-\theta_1))-r(y)}{r(y)}dp^{t,\theta_2}(y)\\
    & + \left(\frac{Q_t(\theta_2)}{Q_t(\theta_1)}-1 \right)\int (l(y)-\zeta)h(y-t(\theta_2-\theta_1))dp^{t,\theta_2}(y)\\
    &+ \int (l(y)-\zeta)\left(h(y+t(\theta_2-\theta_1))-h(y)\right)dp^{t,\theta_2}(y).
\end{align*}
\end{proof}

\subsection{Proof of the bound on $\|\nabla^{\lfloor \beta \rfloor+1} X_1\|_{\mathcal{H}^{\beta-\lfloor \beta \rfloor}}$ for Theorem \ref{theo:thetheo2}}\label{sec:castheo1append}
\subsubsection{Integration against $p^{t,\theta_1}-p^{t,\theta_2}$ and mass concentration}
In this section we focus on the case $t\in [0,1)$ and $\theta_1,\theta_2\in \mathbb{R}^d$ satisfying 
\begin{equation}\label{eq:onremetlacontrainte}
1-t^2\leq C_\star^{-1}\log(\epsilon^{-1})^{-C_\star} \text{ and } \|t\theta_1\|^2+\|t\theta_2\|^2\leq 1-C_2^{-1}\log(\epsilon^{-1}+1)^{-K},
\end{equation}
with $C^\star,C_2>0$ given by Proposition \ref{prop:concentration3}. Let us first derive a general bound on the integration against $p^{t,\theta_1}-p^{t,\theta_2}$.

\begin{proposition}\label{prop:lipptx3}
Let $l,h\in L^2_p(\mathbb{R}^d)$ with
$h \in \mathcal{H}_{1}^{1}(B^d(0,1))$.
Then, for $t\in [0,1),z\in \mathbb{R}^d$, $\theta_1,\theta_2\in \mathbb{R}^d$ satisfying \eqref{eq:onremetlacontrainte}, in the case $\beta\geq 1$ we have 
    \begin{align*}
        &\|\int (l(y)-l(z))h(y)(p^{t,\theta_1}(y)-p^{t,\theta_2}(y))d\lambda^d(y)\|\\
         \leq &C_\epsilon \left(\int\|l(y)-l(y+t(\theta_2-\theta_1))\|^2dp^{t,\theta_1}(y)\right)^{1/2}\sum_{i=1}^2\left(\int\|h(y)\|^2p^{t,\theta_i}(y) \right)^{1/2}\\
        & +C_\epsilon\|\theta_1-\theta_2\|\left(\left(\int\|l(y)-l(z)\|^2dp^{t,\theta_2}(y) \right)^{1/2}+\exp(-\frac{(1-t^2)^{-1/4}}{8})\right)
    \end{align*}
    and in the case $\beta\in (0,1)$, $l(y)-l(z)=y-t\theta_1$, $h(y)=H(y,\theta_2)$ we have
    \begin{align*}
        &\|\int (y-t\theta_1)\left(H(y,\theta_2)\right)(p^{t,\theta_1}(y)-p^{t,\theta_2}(y))d\lambda^d(y)\|\\
         \leq &C\|\theta_1-\theta_2\|\left(\|\int H(y,\theta_2)dp^{t,\theta_2}(y) \|+\|\int \left(y-t\theta_1 \right)dp^{t,\theta_2}(y)\|\right)\\
        & +C_\epsilon\frac{\|\theta_1-\theta_2\|^{\beta}}{\log(\|\theta_1-\theta_2\|^{-1})^{\theta}}\left(\left(\int\|y-t\theta_2\|^2dp^{t,\theta_2}(y) \right)^{1/2}+\exp(-\frac{(1-t^2)^{-1/4}}{8})\right)\left(\int\|H(y,\theta_2)\|^2p^{t,\theta_1}(y) \right)^{1/2}.
    \end{align*}
\end{proposition}
\begin{proof} We are going to bound the different terms of Proposition \ref{prop:generalboundlhp}.
Without loss of generality, we can suppose that $\|\theta_1\|\geq \|\theta_2\|$. Let $\alpha=(1-t^2)^{-1/4}$, $\psi(z)=\exp(-u(\|z\|^2))$ and $\chi = \mathds{1}_{\{\beta<1\}}$, we have 
\begin{align*}
    |&Q_{t}r(\theta_2)-Q_{t}
    r(\theta_1)|\\
    &=|\int \left(r(t\theta_1+\sqrt{1-t^2}z)-r(t\theta_2+\sqrt{1-t^2}z)\right)\gamma_d(z)d\lambda^d(z)|\\
    &=|\int \exp(-u(\|t\theta_1+\sqrt{1-t^2}z\|^2))(\exp(a(t\theta_1+\sqrt{1-t^2}z))-\exp(a(t\theta_1+\sqrt{1-t^2}z))\gamma_d(z)d\lambda^d(z)|\\
    & +\int\exp(a(t\theta_1+\sqrt{1-t^2}z))(\exp(-u(\|t\theta_1+\sqrt{1-t^2}z\|^2))-\exp(-u(\|t\theta_2+\sqrt{1-t^2}z\|^2))\gamma_d(z)d\lambda^d(z)|\\
    &\leq C\|\theta_1-\theta_2\|^{\chi\beta+1-\chi}\log(\|\theta_1-\theta_2\|^{-1})^{-\chi\theta} \int \exp(-u(\|t\theta_1+\sqrt{1-t^2}z\|^2))\gamma_d(z)d\lambda^d(z)\\
    & +C \int \int_0^1\|\nabla \psi\left(t\theta_1+\sqrt{1-t^2}z+st(\theta_2-\theta_1)\right)(\theta_2-\theta_1)\|ds\left(\mathds{1}_{\{\|z\|\leq \alpha\}}+\mathds{1}_{\{\|z\|> \alpha\}}\right)\gamma_d(z)d\lambda^d(z)\\
     & \leq  C\|\theta_1-\theta_2\|^{\chi\beta+1-\chi}\log(\|\theta_1-\theta_2\|^{-1})^{-\chi\theta}Q_tr(\theta_1) + C\|\theta_2-\theta_1\|\|\nabla \psi\|_{\mathcal{H}^0(B^d(t\theta_1,\|\theta_1-\theta_2\|+\sqrt{1-t^2}\alpha)}\\
     & +C\|\theta_2-\theta_1\| e^{-\frac{\alpha^2}{4}}\int\int_0^1 \|\nabla \psi\left(t\theta_1+\sqrt{1-t^2}z+st(\theta_2-\theta_1)\right)\|\exp(-\|z\|^2/4)dsd\lambda^d(z).
\end{align*}
We then obtain
\begin{align*}
    |Q_{t}r(\theta_2)-Q_{t}
    r(\theta_1)|\leq &C\|\theta_1-\theta_2\|^{\chi\beta+1-\chi}\log(\|\theta_1-\theta_2\|^{-1})^{-\chi\theta}Q_tr(\theta_1)\\
    & + C \|\theta_2-\theta_1\|\left(\|\nabla \psi\|_{\mathcal{H}^0(B^d(t\theta_1,\|\theta_1-\theta_2\|+(1-t^2)^{1/4}))}+e^{-\frac{1}{4(1-t^2)^{1/2}}}\right),
\end{align*}
which gives 
\begin{align*}
    |\frac{Q_{t}r(\theta_2)}{Q_{t}
    r(\theta_1)}-1|\leq &C\|\theta_1-\theta_2\|^{\chi\beta+1-\chi}\log(\|\theta_1-\theta_2\|^{-1})^{-\chi\theta}\\
    & + C\frac{\|\theta_2-\theta_1\|}{Q_{t}
    r(\theta_1)}\left(\|\nabla \psi\|_{\mathcal{H}^0(B^d(t\theta_1,\|\theta_1-\theta_2\|+(1-t^2)^{1/4}))}+e^{-\frac{1}{4(1-t^2)^{1/2}}}\right),
\end{align*}
so for $\theta_1$ in the case \eqref{eq:onremetlacontrainte}, we deduce that
\begin{align*}
    |\frac{Q_{t}r(\theta_2)}{Q_{t}
    r(\theta_1)}-1|\leq C_\epsilon\|\theta_1-\theta_2\|^{\chi\beta+1-\chi}\log(\|\theta_1-\theta_2\|^{-1})^{-\chi\theta}.
\end{align*}

On the other hand 
\begin{align*}
    |\frac{r(y-t(\theta_2-\theta_1))}{r(y)}-1| & \leq C \frac{\|\nabla \psi\|_{\mathcal{H}^0(B^d(y,t\|\theta_1-\theta_2\|))}}{r(y)}\|\theta_2-\theta_1\|+C\|\theta_1-\theta_2\|^{\chi\beta+1-\chi}\log(\|\theta_1-\theta_2\|^{-1})^{-\chi\theta},
\end{align*}
so for $\theta_2$ in the case \eqref{eq:onremetlacontrainte}, we deduce that
\begin{align*}
    &\|\int (l(y)-l(z))h(y + t(\theta_2 -\theta_1))\frac{r(y-t(\theta_2-\theta_1))-r(y)}{r(y)}dp^{t,\theta_2}(y)\|\\
    &\leq C\|\theta_1-\theta_2\| \int \|l(y)-l(z)\|\left(C_\epsilon +\frac{\|\nabla \psi\|_{\mathcal{H}^0(B^d(y,t\|\theta_1-\theta_2\|))}}{r(y)}\mathds{1}_{\{\|y-t\theta_2\|\geq (1-t^2)^{1/4}\}}+\|\theta_1-\theta_2\|^{\chi\beta-\chi}\right)\\
    &\times\|h(y + t(\theta_2 -\theta_1))\|dp^{t,\theta_2}(y)\\
    &\leq C_\epsilon\|\theta_1-\theta_2\|^{\chi\beta+1-\chi}\log(\|\theta_1-\theta_2\|^{-1})^{-\chi\theta}\left(\left(\int\|l(y)-l(z)\|^2dp^{t,\theta_2}(y) \right)^{1/2}+\frac{\exp(-(1-t^2)^{-1/4}/4)}{Q_tr(\theta_2)}\right)\\
&\times\left(\int\|h(y)\|^2p^{t,\theta_1}(y) \right)^{1/2}
    \\
    &\leq C_\epsilon\|\theta_1-\theta_2\|^{\chi\beta+1-\chi}\log(\|\theta_1-\theta_2\|^{-1})^{-\chi\theta}\left(\left(\int\|l(y)-l(z)\|^2dp^{t,\theta_2}(y) \right)^{1/2}+\exp(-(1-t^2)^{-1/4}/8)\right)\\
&\times\left(\int\|h(y)\|^2p^{t,\theta_1}(y) \right)^{1/2}.
\end{align*}
Using Proposition \ref{prop:generalboundlhp} and Cauchy-Schwarz we get the result.
\end{proof}

Let us now bound the terms
$\int\|f(y)-f(y+t(\theta_2-\theta_1))\|^2dp^{t,\theta_1}(y)$ and $\int\|f(y)-f(z)\|^2dp^{t,\theta_2}(y)$ for $f=f^l$ with $l\in \{1,...,\lfloor \beta \rfloor\}$ and $z=t\theta_i$, $i\in \{1,2\}$. From Proposition \ref{prop:concentration3}, for $t\in [0,1)$ and $\theta_1,\theta_2\in \mathbb{R}^d$ satisfying 
\eqref{eq:onremetlacontrainte} we have 
\begin{equation}\label{align:concegenephase2} 
    \int\|f^l(y)-f^l(t\theta_i)\|^2dp^{t,\theta_i}(y)\leq C_\epsilon(1-t^2)^{\beta-\lfloor \beta \rfloor}\log((1-t^2)^{-1})^{-2\theta}.
\end{equation}

Furthermore,
\begin{align}\label{align:unautreboundparmit}
    &\left(\int \|f^l(y)-f^l(y+t(\theta_2-\theta_1))\|^2dp^{t,\theta_1}(y)\right)^{1/2}\nonumber\\
    & = \left(\int \left(\frac{\|f^l(y)-f^l(y+t(\theta_2-\theta_1))\|}{\|\theta_2-\theta_1\|^{\beta-\lfloor \beta \rfloor}\log(\|\theta_2-\theta_1\|^{-1})^{-\theta}}\right)^2\|\theta_2-\theta_1\|^{2(\beta-\lfloor \beta \rfloor)}\log(\|\theta_2-\theta_1\|^{-1})^{-2\theta}dp^{t,\theta_1}(y)\right)^{1/2}\nonumber\\
    & \leq C \|\theta_2-\theta_1\|^{\beta-\lfloor \beta \rfloor}\log(\|\theta_2-\theta_1\|^{-1})^{-\theta}\nonumber\\
    & \times \left(\int C_\epsilon\left(1+\mathds{1}_{y\notin B^d(0,1-(C^\star)^{-1}\log(\epsilon^{-1})^{-C^\star})}\int_0^1 \nabla^{l+1}u((\|y\|+st(\theta_1-\theta_2))^2) ds\right)^2dp^{t,\theta_1}(y)\right)^{1/2}\nonumber\\
    & \leq C_\epsilon\|\theta_2-\theta_1\|^{\beta-\lfloor \beta \rfloor}\log(\|\theta_2-\theta_1\|^{-1})^{-\theta}.
\end{align}

\subsubsection{Bounds on the different terms}\label{sec:bounondiff3}
Let $f:=f^{l}$ with $l\in \{0,...,\lfloor \beta \rfloor\}$ and take $t\in [0,1)$ satisfying $1-t^2\leq C_\star^{-1}\log(\epsilon^{-1})^{-C_\star}$, with $C^\star>0$ given by Proposition \ref{prop:concentration3}. We give estimates on $(R_j)_{j=0,...,3}$ recalling that the case $R_0$ corresponds to the case $k:=\lfloor \beta \rfloor +1= 1$ and the case  $(R_j)_{j=1,2,3}$ corresponds to the case $k\geq 2$. We first focus on the setting $t\in [0,1)$ such that $(1-t^2)^{1/2}\leq \|\theta_1-\theta_2\|$. 

For the terms $(R_j)_{j=1,2,3}$ we use the bound \eqref{align:x1x2grand}. Taking $\xi_m=f(t\theta_m)$ with $m\in \{1,2\}$ in Proposition \ref{prop:concentration3} we obtain for $i\in\{1,2\}$,
\begin{align*}
     \|\int f(y)\big(\nabla^i_x p^{t,\theta_1}(y)-\nabla^i_x p^{t,\theta_2}(y)\big)d\lambda^d(y)\|& \leq \frac{C}{(1-t^2)^{i/2}}\sum_{m=1}^2\left(\int \|f(y)-f(t\theta_m)\|^2dp^{t,\theta_m}(y)\right)^{1/2}\\
     & \leq \frac{C_\epsilon}{(1-t^2)^{i/2}}(1-t^2)^{(\beta-\lfloor \beta \rfloor)/2}\log((1-t^2)^{-1})^{-\theta}\\
     & \leq \frac{C_\epsilon}{(1-t^2)^{i/2}}\|\theta_1-\theta_2\|^{\beta-\lfloor \beta \rfloor}\log((1-t^2)^{-1})^{-\theta}.
\end{align*}
Using this result we get
\begin{align*}
    R_1(t,\theta_1,\theta_2)= & \sum_{l=1}^{k-1} \|\int f^{l}(y)(\nabla^2_x dp^{t,\theta_1}(y)-\nabla^2_x dp^{t,\theta_2}(y))\| \prod_{j=1}^{k-2}\left(\|\int f^{j}(y) dp^{t,\theta_1}(y)\|\vee 1\right)^{k-2}\\
    & \leq \frac{C_\epsilon}{1-t^2}\|\theta_1-\theta_2\|^{\beta-\lfloor \beta \rfloor}\log((1-t^2)^{-1})^{-\theta},
\end{align*}

\begin{align*}
    R_2(t,\theta_1,\theta_2)= &  \sum_{l=1}^{k-1} \|\int f^{l}(y) (\nabla_x dp^{t,\theta_1}(y)-\nabla_x dp^{t,\theta_2}(y))\|\\
    & \times \sum_{j=1}^{k-2}\|\int f^{j}(y) \nabla_x dp^{t,\theta_1}(y)\|\prod_{i=1}^{k-2}\left(\|\int f^{i}(y)  dp^{t,\theta_1}(y)\|\nonumber\vee 1 \right)^{k-2}\\
    \leq & \frac{C_\epsilon}{(1-t^2)^{1/2}}\|\theta_1-\theta_2\|^{\beta-\lfloor \beta \rfloor}\log((1-t^2)^{-1})^{-\theta}\frac{1}{1-t^2}\sum_{j=1}^{k-2}\|\int f^{j}(y) \left(y-\int z dp^{t,\theta_1}(z)\right)dp^{t,\theta_1}(y)\|\\
    \leq & \frac{C_\epsilon}{1-t^2}\|\theta_1-\theta_2\|^{\beta-\lfloor \beta \rfloor}\log((1-t^2)^{-1})^{-\theta}
\end{align*}
and using Proposition \ref{prop:lipptx3} and \eqref{align:unautreboundparmit},
\begin{align*}
    R_3(s,\theta_1,\theta_2)= & \sum_{j=1}^{k-2} \|\int f^{j}(y)  (dp^{s,\theta_1}(y)-dp^{s,\theta_2}(y))\|\prod_{i=1}^{k-3}\left(\|\int f^{i}(y)  dp^{s,X_s(x_1)}(y)\|\vee 1\right)^{k-3}\\
    & \times\Biggl(\sum_{l=1}^{k-1}  \|\int f^{l}(y)\nabla^2_x dp^{s,\theta_2}(y)\|+\|\int f^{l}(y)\nabla_x dp^{s,X_s(x_2)}(y)\|\sum_{r=1}^{k-1}\|\int f^{r}(y)\nabla_x dp^{s,\theta_2}(y)\|\Biggl)\\
    \leq & C_\epsilon\|\theta_1-\theta_2\|^{\beta-\lfloor \beta \rfloor}\log((1-t^2)^{-1})^{-\theta}\frac{C_\epsilon}{1-t^2}.
\end{align*}

For the term $R_0$, we use the bound \eqref{eq:roboundtpetit2} and the decomposition on $\nabla V$ derived in \eqref{align:decompnabla1}. We have
\begin{align*}
    \|\nabla V(t,\theta_i)\| \leq & \frac{t}{(1-t^2)^2}\|\int\left(y-t\theta_i\right)^{\otimes2}\left( \frac{r(y)}{\int r(z)\varphi^{t,\theta_i}(z)d\lambda^d(z)}-1\right)\varphi^{t,\theta_i}(y)d\lambda^d(y)\|\nonumber\\
         &+\frac{t}{(1-t^2)^2}\|\int(y-t\theta_i)d p^{t,\theta_i}(y)\|^2.
\end{align*}
On one hand we have
\begin{align*}
    \|\int(y-t\theta_i)d p^{t,\theta_i}(y)\|^2 = &\|\int(y-t\theta_i)\left(\frac{r(y)}{Q_tr(x)}-\frac{r(t\theta_i)}{Q_tr(x)}\right)\varphi^{t,\theta_i}(y)d\lambda^d(y)\|^2\\
    = & \frac{1}{Q_tr(x)^2}\|\int(y-t\theta_i)\left(r(y)-r(\theta_i)\right)\varphi^{t,\theta_i}(y)\mathds{1}_{\{\|y-t\theta_i\|\leq \delta\}}d\lambda^d(y)\|^2
    \\
     & +\frac{1}{Q_tr(x)^2}\|\int(y-t\theta_i)\left(r(y)-r(\theta_i)\right)\varphi^{t,\theta_i}(y)\mathds{1}_{\{\|y-t\theta_i\|> \delta\}}d\lambda^d(y)\|^2,
\end{align*}
with $\delta=(1-t^2)^{1/2}\log((1-t^2)^6)^{1/2}$. Furthermore,

\begin{align*}
    |r(y)-r(t\theta_i)|&=|\frac{p(y)}{\gamma_d(y)}-\frac{p(t\theta_i)}{\gamma_d(t\theta_i)}|=|\frac{p(y)-p(t\theta_i)}{\gamma_d(y)}+p(\theta_i)(\frac{1}{\gamma_d(y)}-\frac{1}{\gamma_d(t\theta_i)})|\\
    & \leq C |p(y)-p(t\theta_i)|+p(\theta_i)\|y-t\theta_i\|.
\end{align*}
Then, from Proposition \ref{prop:loca3}, in the case
\begin{equation*}
1-t^2\leq C_\star^{-1}\log(\epsilon^{-1})^{-C_\star} \text{ and } \|x\|^2\leq \log(\epsilon^{-1}),
\end{equation*} 
we have 
$$\|t\theta_i\|\leq 1-C^{-1}\log(\epsilon^{-1}+1)^{-K},$$
so for $y\in B^d(0,1)$ such that $\|y-t\theta_i\|\leq \delta$, we have
\begin{align*}
    |p(y)-p(t\theta_i)|= & \exp(-u(\|y\|^2)+a(y))-\exp(-u(\|t\theta_i)\|^2)+a(t\theta_i)))\\
    \leq & C|\exp(-u(\|y\|^2))-\exp(-u(\|t\theta_i)\|^2)|+\exp(-u(\|t\theta_i)\|^2)|a(t\theta_i)-a(y)|\\
    \leq & C\log(\epsilon^{-1})^{C_2}\exp(-u(\|t\theta_i)\|^2)\|y-t\theta_i\|\\
     &+C\exp(-u(\|t\theta_i)\|^2)\|y-t\theta_i\|^{\beta-\lfloor \beta \rfloor}\log(\|y-t\theta_i\|^{-1})^{-\theta}\\
    \leq & \exp(-u(\|t\theta_i)\|^2)(C\log(\epsilon^{-1})^{C_2}\delta+C\delta^{\beta-\lfloor \beta \rfloor}\log(\delta^{-1})^{-\theta}).
\end{align*}
Then, we deduce that
\begin{align*}
    \frac{1}{Q_tr(x)^2}&\|\int(y-t\theta_i)\left(r(y)-r(\theta_i)\right)\varphi^{t,\theta_i}(y)\mathds{1}_{\{\|y-t\theta_i\|\leq \delta\}}d\lambda^d(y)\|^2 \\\leq  & C\log(\epsilon^{-1})^{C_2} \left(\int \delta^{\beta-\lfloor \beta \rfloor+1}\log(\delta^{-1})^{-\theta})\varphi^{t,\theta_i}(y)\mathds{1}_{\{\|y-t\theta_i\|\leq \delta\}}d\lambda^d(y)\right)^2\\
    \leq & C\log(\epsilon^{-1})^{C_2} \delta^{2\beta-\lfloor \beta \rfloor+1}\log(\delta^{-1})^{-\theta}= C\log(\epsilon^{-1})^{C_2} (1-t^2)^{\beta-\lfloor \beta \rfloor+1}\log((1-t^2)^{-1})^{-\theta+1}\\
     \leq & C\log(\epsilon^{-1})^{C_2} (1-t^2)^{1+(\beta-\lfloor \beta \rfloor)/2}\log((1-t^2)^{-1})^{-\theta}.
\end{align*}
On the other hand, we have
\begin{align*}
    \frac{1}{Q_tr(x)^2}&\|\int(y-t\theta_i)\left(r(y)-r(\theta_i)\right)\varphi^{t,\theta_i}(y)\mathds{1}_{\{\|y-t\theta_i\|> \delta\}}d\lambda^d(y)\|^2\\
    \leq &C\log(\epsilon^{-1})^{C_2} \int\|y-t\theta_i\|\varphi^{t,\theta_i}(y)\mathds{1}_{\{\|y-t\theta_i\|> \delta\}}d\lambda^d(y)\\
     \leq & C\log(\epsilon^{-1})^{C_2} (1-t^2)^{3/2}.
\end{align*}
The same way, we can show that
\begin{align*}
    \|&\int\left(y-t\theta_i\right)^{\otimes2}\left( \frac{r(y)}{\int r(z)\varphi^{t,\theta_i}(z)d\lambda^d(z)}-1\right)\varphi^{t,\theta_i}(y)d\lambda^d(y)\|\\
    &\leq C\log(\epsilon^{-1})^{C_2}(1-t^2)^{1+(\beta-\lfloor \beta \rfloor)/2}\log((1-t^2)^{-1})^{-\theta}.
\end{align*}
We then conclude that
\begin{align}\label{align:estimatenablatgrand}
    \|\nabla V(t,\theta_i)\|\leq C\log(\epsilon^{-1})^{C_2} \frac{t}{(1-t^2)^{1-(\beta-\lfloor \beta \rfloor)/2}}\log\Big(\big((1-t^2)\wedge 1/2\big)^{-1}\Big)^{-\theta}.
\end{align}
Therefore, using \eqref{eq:roboundtpetit2} we obtain
\begin{align*}
    R_0&(t,\theta_1,\theta_2)\leq C_\epsilon (1-t^2)\|\theta_1-\theta_2\|^{\beta-\lfloor \beta \rfloor}\log((1-t^2)^{-1})^{-\theta}.
\end{align*}

Let us now treat the case $\|\theta_1-\theta_2\|< (1-t^2)^{1/2}$. First, for the term $R_0$, we get from \eqref{align:roboundtgrand}  
and Proposition \ref{prop:lipptx3} that
\begin{align*}
    R_0(t,\theta_1,\theta_2)     \leq & \frac{t}{1-t^2}\|\int H(y,\theta_1) dp^{t,\theta_1}(y)\|\left(\|\theta_1-\theta_2\|+\|\int w (p^{t,\theta_1}(w)-p^{t,\theta_2}(w))d\lambda^d(w)\|\right)\\
    & + \frac{t}{1-t^2}\|\int (y-t\theta_1)H(y,\theta_2)(p^{t,\theta_1}(y)-p^{t,\theta_2}(y))d\lambda^d(y)\|\\
    \leq & C_\epsilon(1-t^2)^{1/2+(\beta-\lfloor \beta\rfloor)/2}(\|\theta_1-\theta_2\|+\|\theta_1-\theta_2\|^{\beta-\lfloor \beta\rfloor}\log(\|\theta_1-\theta_2\|^{-1})^{-\theta}(1-t^2)^{1/2})\\
    &+ C\|\theta_1-\theta_2\|(1-t^2)^{1/2+(\beta-\lfloor \beta\rfloor)/2}\\
    & +C_\epsilon(1-t^2)^{1/2}\|\theta_1-\theta_2\|^{\beta-\lfloor \beta\rfloor}\log(\|\theta_1-\theta_2\|^{-1})^{-\theta}(1-t^2)^{1/2}\\
    \leq & C_\epsilon(1-t^2)\|\theta_1-\theta_2\|^{\beta-\lfloor \beta\rfloor}\log(\|\theta_1-\theta_2\|^{-1})^{-\theta}.
\end{align*}
Let us now focus on $(R_{j})_{j=1,2,3}$.
From \eqref{eq:premiertermlipgrand}, \eqref{align:concegenephase2} and Proposition \ref{prop:lipptx3} we have that
\begin{align*}
    \|A_1\|\leq& C(1-t^2)^{1/2}\left(\int \|f(y)-f(t\theta_1)\|^2dp^{t,\theta_1}(y)\right)^{1/2} \|\int z(p^{t,\theta_1}(z)-p^{t,\theta_2}(z))d\lambda^d(z)\|\\
    \leq & C_\epsilon (1-t^2)^{\frac{1+\beta-\lfloor \beta \rfloor}{2}}\log((1-t^2)^{-1})^{-\theta} \|\theta_1-\theta_2\|\\
    \leq & C_\epsilon (1-t^2)^{\frac{1+\beta-\lfloor \beta \rfloor}{2}}\log((1-t^2)^{-1})^{-\theta} \|\theta_1-\theta_2\|^{\beta-\lfloor \beta \rfloor}(1-t^2)^{\frac{1-(\beta-\lfloor \beta \rfloor)}{2}}\\
    \leq & C_\epsilon(1-t^2)\log((1-t^2)^{-1})^{-\theta}\|\theta_1-\theta_2\|^{\beta-\lfloor \beta \rfloor}.
\end{align*}
From \eqref{align:B1} and \eqref{align:unautreboundparmit} we have
\begin{align*}
    \|B_1\|  \leq &C (1-t^2)\left(\int \|f(y)-f(y+t(\theta_2-\theta_1))\|^2dp^{t,\theta_1}(y)\right)^{1/2}\nonumber\\
     &+ C (1-t^2)^{1/2}\left(\int \|f(y)-f(y+t(\theta_2-\theta_1))\|^2dp^{t,\theta_1}(y)\right)^{1/2}\|\int w (p^{t,\theta_1}(w)-p^{t,\theta_2}(w))d\lambda^d(w)\|\\
    \leq &C_\epsilon\|\theta_2-\theta_1\|^{\beta-\lfloor \beta \rfloor}\log(\|\theta_2-\theta_1\|^{-1})^{-2} (1-t^2)\\
    &+ C_\epsilon\|\theta_2-\theta_1\|^{\beta-\lfloor \beta \rfloor}\log(\|\theta_2-\theta_1\|^{-1})^{-2} (1-t^2)^{1/2}\|\theta_2-\theta_1\|\\
    \leq &C_\epsilon\|\theta_2-\theta_1\|^{\beta-\lfloor \beta \rfloor}\log((1-t^2)^{-1})^{-2} (1-t^2).
\end{align*}

From \eqref{align:D1} and \eqref{align:concegenephase2} we have 
\begin{align*}
    \|D_1\|&\leq C\|\theta_2-\theta_1\|(1-t^2)^{1/2}\left(\int \|f(y+t(\theta_2-\theta_1))-f(t\theta_1)\|^2dp^{t,\theta_1}(y)\right)^{1/2}\\
    &\leq C_\epsilon\|\theta_2-\theta_1\|(1-t^2)^{(1+\beta-\lfloor \beta \rfloor)/2}\log((1-t^2)^{-1})^{-\theta}\\
    &\leq C_\epsilon\|\theta_2-\theta_1\|^{\beta-\lfloor \beta \rfloor}(1-t^2)\log((1-t^2)^{-1})^{-\theta}.
\end{align*}

From \eqref{align:D2}, \eqref{align:concegenephase2} and Proposition \ref{prop:lipptx3}  we have 
\begin{align*}
    \|D_2\|
    \leq & \left(\int \|f(y+t(\theta_2-\theta_1))-f(t\theta_1)\|^2dp^{t,\theta_1}(y)\right)^{1/2}\|\int H(w,\theta_2)^{\otimes 2}(dp^{t,\theta_1}(w)- dp^{t,\theta_1}(w))\|\\
    \leq & C_\epsilon(1-t^2)^{(\beta-\lfloor \beta \rfloor)/2}\log((1-t^2)^{-1})^{-\theta}\|\int \left(w-\int zdp^{t,\theta_1}(z)\right) H(w,\theta_2)(dp^{t,\theta_1}(w)- dp^{t,\theta_1}(w))\|\\
    \leq & C_\epsilon(1-t^2)^{(\beta-\lfloor \beta \rfloor)/2}\log((1-t^2)^{-1})^{-\theta}\|\theta_2-\theta_1\|(1-t^2)^{1/2}
    \\
    \leq & C_\epsilon(1-t^2)\log((1-t^2)^{-1})^{-\theta}\|\theta_2-\theta_1\|^{\beta-\lfloor \beta \rfloor}.
\end{align*}

From \eqref{align:D3} and Proposition \ref{prop:generalboundlhp} we have
\begin{align*}
    \|D_3\|\leq &C \|\int (f(y)-f(t\theta_1))\left(H(y,\theta_2)^{\otimes 2}-\int H(w,\theta_2)^{\otimes 2} dp^{t,\theta_2}(w)\right)(p^{t,\theta_1}(y)-p^{t,\theta_2}(y))d\lambda^d(y)\|\\
    \leq & C \sum_{i=1}^2\left(\int\|H(y,\theta_2)^{\otimes 2}-\int H(w,\theta_2)^{\otimes 2} dp^{t,\theta_2}(w)\|^2dp^{t,\theta_i}(y) \right)^{1/2}\\
    & \times \Biggl(\left(\int\|l(y)-l(y+t(\theta_2-\theta_1))\|^2dp^{t,\theta_1}(y)\right)^{1/2}+ |\frac{Q_t(\theta_2)}{Q_t(\theta_1)}-1|\left(\int \|f(y)-f(t\theta_1)\|^2dp^{t,\theta_2}(y)\right)^{1/2}\Biggl)\\
   + & \|\int (f(y)-\zeta)\left(H(y-t(\theta_2-\theta_1),\theta_2)^{\otimes 2}-\int H(w,\theta_2)^{\otimes 2} dp^{t,\theta_2}(w)\right)\frac{r(y-t(\theta_2-\theta_1))-r(y)}{r(y)}dp^{t,\theta_2}(y)\|\\
    & + C\left(\int \|f(y)-f(t\theta_1)\|^2 dp^{t,\theta_2}(y)\right)^{1/2}\left(\int\|H(y,\theta_2)^{\otimes 2}-H(y-t(\theta_2-\theta_1),\theta_2)^{\otimes 2}\|^2p^{t,\theta_2}(y) \right)^{1/2}.
\end{align*}
Now applying Branscamp-Lieb inequality we get
\begin{align*}
    & \sum_{i=1}^2\left(\int\|H(y,\theta_2)^{\otimes 2}-\int H(w,\theta_2)^{\otimes 2} dp^{t,\theta_2}(w)\|^2dp^{t,\theta_i}(y) \right)^{1/2}\\
    \leq & C(1-t^2)+C(1-t^2)^{1/2}\left(\int\|H(y,\theta_2)\|^2dp^{t,\theta_1}(y) \right)^{1/2}
    \\
    \leq & C(1-t^2)+C(1-t^2)^{1/2}\|\int w (p^{t,\theta_1}(w)-p^{t,\theta_2}(w))d\lambda^d(w)\|\\
    \leq & C(1-t^2)+C(1-t^2)^{1/2}\|\theta_1-\theta_2\|\\
    \leq & C(1-t^2).
\end{align*}
Furthermore, using \eqref{align:H2lipplust} we get
\begin{align*}
& \left(\int\|H(y,\theta_2)^{\otimes 2}-H(y-t(\theta_2-\theta_1),\theta_2)^{\otimes 2}\|^2p^{t,\theta_2}(y) \right)^{1/2}\\
& \leq C\|\theta_2-\theta_1\|\left(\|y-\int w dp^{t,\theta_2}(w)\|+\|\theta_2-\theta_1\|\right)\\
&\leq C(1-t^2)^{1/2}\|\theta_2-\theta_1\|.
\end{align*}
Then using Proposition \ref{prop:lipptx3} and \eqref{align:unautreboundparmit} we obtain 
\begin{align*}
    \|D_3\|\leq & (1-t^2)\Big(C_\epsilon\|\theta_2-\theta_1\|^{\beta-\lfloor \beta \rfloor}\log(\|\theta_2-\theta_1\|^{-1})^{-2}+ C_\epsilon\|\theta_2-\theta_1\|^{\beta-\lfloor \beta \rfloor}\log((1-t^2)^{-1})^{-2}\Big)\\
    & +C_\epsilon\|\theta_2-\theta_1\|^{\beta-\lfloor \beta \rfloor}\log((1-t^2)^{-1})^{-2}(1-t^2)\\
    & +C_\epsilon(1-t^2)^{(1+\beta-\lfloor \beta \rfloor)/2}\|\theta_2-\theta_1\|\log((1-t^2)^{-1})^{-2}\\
    & \leq C_\epsilon(1-t^2)\log((1-t^2)^{-1})^{-2}\|\theta_2-\theta_1\|^{\beta-\lfloor \beta \rfloor}.
\end{align*}
Therefore, we can conclude that 
\begin{align*}
    R_1(t,\theta_1,\theta_2) & = \sum_{l=1}^{k}\|\int f^{l}(y)(\nabla^2_x dp^{t,\theta_1}(y)-\nabla^2_x dp^{t,\theta_2}(y))\| \prod_{j=0}^{k+1}\left(\|\int f^{j}(y) dp^{t,\theta_1}(y)\|\vee 1\right)\\
    & \leq  C_\epsilon(1-t^2)^{-1}\log((1-t^2)^{-1})^{-\theta}\|\theta_2-\theta_1\|^{\beta-\lfloor \beta \rfloor}.
\end{align*}
From \eqref{align:boundR2}, Propositions \ref{prop:lipptx3} and \ref{prop:firstboundongradsth2} we have 
\begin{align*}
     R_2(t,\theta_1,\theta_2)  = & \sum_{l=1}^{k} \|\int f^{l}(y) (\nabla_x dp^{t,\theta_1}(y)-\nabla_x dp^{t,\theta_2})\|\sum_{j=1}^{k+1}\|\int f^{j}(y) \nabla_x dp^{t,\theta_1}(y)\|\prod_{i=0}^{k+1}\left(\|\int f^{i}(y)  dp^{t,\theta_1}(y)\|\vee 1\right)\\
      \leq  &C (1-t^2)^{-3/2}\log((1-t^2)^{-1})^{-2}\sum_{l=1}^{k}\left(\int \|f^l(y)-f^l(t\theta_1)\|^2dp^{t,\theta_1}(y)\right)^{1/2}\|\int w (dp^{t,\theta_1}(w)-dp^{t,\theta_2}(w))\|\\
    & + \|\int (f^l(y)-f^l(t\theta_1))H(y,\theta_2)(p^{t,\theta_1}(y)-p^{t,\theta_2}(y))d\lambda^d(y)\|\\
    &\leq  C_\epsilon(1-t^2)^{-3/2}\log((1-t^2)^{-1})^{-2}(1-t^2)^{(\beta-\lfloor \beta \rfloor)/2} \|\theta_1-\theta_2\|\\
    & +  C_\epsilon(1-t^2)^{-3/2}\log((1-t^2)^{-1})^{-2}\Big((1-t^2)^{1/2} \|\theta_1-\theta_2\|^{\beta-\lfloor \beta \rfloor}+(1-t^2)^{(\beta-\lfloor \beta \rfloor)/2} \|\theta_1-\theta_2\|\Big)\\
    & \leq  C_\epsilon(1-t^2)^{-1}\log((1-t^2)^{-1})^{-2}\|\theta_2-\theta_1\|^{\beta-\lfloor \beta \rfloor}.
\end{align*}

Finally using Propositions \ref{prop:firstboundongradsth2} and \ref{prop:lipptx3} we obtain
\begin{align*}
     R_3(t,\theta_1,\theta_2)  = &\sum_{j=0}^{k} \|\int f^{j}(y)  (dp^{t,\theta_1}(y)-dp^{t,\theta_2}(y))\|\prod_{i=0}^{k+1}\left(\|\int f^{i}(y)  dp^{t,\theta_1}(y)\|\vee 1\right)\\
    & \times\Biggl(\sum_{l=1}^{k}  \|\int f^{l}(y)\nabla^2_x dp^{t,\theta_2}(y)\|+\|\int f^{l}(y)\nabla_x dp^{t,\theta_2}(y)\|\sum_{i=1}^{k}\|\int f^{i}(y)\nabla_x dp^{t,\theta_2}(y)\|\Biggl)\\
    \leq & C_\epsilon \|\theta_2-\theta_1\|^{\beta-\lfloor \beta \rfloor}\log(\|\theta_2-\theta_1\|^{-1})^{-\theta}(1-t^2)^{-1}
    \\
    \leq & C_\epsilon \|\theta_2-\theta_1\|^{\beta-\lfloor \beta \rfloor}\log((1-t^2)^{-1})^{-\theta}(1-t^2)^{-1}.
\end{align*}

\subsubsection{Putting everything together}
Let us take $C^\star>0$ large enough so that  Propositions \ref{prop:loca3}, \ref{prop:concentHtheo2} and \ref{prop:firstboundongradsth2} hold. We first treat the case  $1-t^2\geq C^{-1}_\star\log(\epsilon^{-1})^{-C_\star}$.
Writing
$$ F_{|B|}(z):=(1-t^2)^{|B|}\exp(\frac{\|z\|^2}{1-t^2})\nabla^{|B|} \exp(-\frac{\|z\|^2}{1-t^2})$$
like in Section \ref{sec:evertogether3}, we have 
\begin{align*}
   & \|\nabla^{k} V(t,X_t(x_1))-\nabla^{k} V(t,X_t(x_2))\|\\
     = &\|\frac{t^{k}}{(1-t^2)^{k+1}}\sum_{\pi \in \Pi_{k+1}} (-1)^{|\pi|+1} \left(\prod_{B\in \pi} \int F_{|B|}(y-tX_t(x_1))dp^{t,X_t(x_1)}(y)-\prod_{B\in \pi} \int F_{|B|}(y-tX_t(x_2))dp^{t,X_t(x_2)}(y)\right)\|\\
     \leq& C_\epsilon \sum_{l=1}^{k+1}\|\int F_{l}(y-tX_t(x_1))dp^{t,X_t(x_1)}(y)-\int F_{l}(y-tX_t(x_2))dp^{t,X_t(x_2)}(y)\|\\
     & \times \sum_{r=1}^2\prod_{j=1}^{k+1} \left(\|\int F_{j}(y-tX_t(x_r))dp^{t,X_t(x_r)}(y)\|\vee 1\right)^{k+1}\\
     & \leq C_\epsilon \sum_{l=1}^{k+1}\|\int F_{l}(y-tX_t(x_1))dp^{t,X_t(x_1)}(y)-\int F_{l}(y-tX_t(x_2))dp^{t,X_t(x_2)}(y)\|\\
     & \times \sum_{r=1}^2\prod_{j=1}^{k+1}\left(\int \sum_{i=1}^{j}(\|y\|^i+\|X_t(x_r)\|^i)dp^{t,X_t(x_r)}(y)\vee 1\right)^{k+1}\\
     &\leq C_\epsilon \sum_{l=1}^{k+1}\|\int F_{l}(y-tX_t(x_1))dp^{t,X_t(x_1)}(y)-\int F_{l}(y-tX_t(x_2))dp^{t,X_t(x_2)}(y)\|,
\end{align*}
as from Proposition \ref{prop:localfortsmall} we have that $\|X_t(x_1)\|+\|X_t(x_2)\|\leq C_\epsilon$. Furthermore, from Proposition \ref{prop:lipptx3} we obtain
\begin{align*}
    &\|\int F_{l}(y-tX_t(x_1))dp^{t,X_t(x_1)}(y)-\int F_{l}(y-tX_t(x_2))dp^{t,X_t(x_2)}(y)\|\\
    & \leq  \|\int (F_{l}(y-tX_t(x_1))-F_{l}(y-tX_t(x_2)))dp^{t,X_t(x_1)}(y)\| +\|\int F_{l}(y-tX_t(x_2))(dp^{t,X_t(x_1)}(y)-dp^{t,X_t(x_2)}(y))\|\\
    & \leq C_\epsilon \|x_1-x_2\|.
\end{align*}
Let us now consider the case $1-t^2\leq C^{-1}_\star\log(\epsilon^{-1})^{-C_\star}$. From Proposition \ref{prop:loca3} we have $\|tX_t(x_1)\|^2+\|tX_t(x_2)\|^2 \leq 1-C^{-1}\log(\epsilon^{-1}+1)^{-K}$ so putting together the bounds on $(R_{j})_{j=0,...,3}$ from Section \ref{sec:bounondiff3} we finally obtain for all $t\in [0,1)$, 
$$\sum_{j=1}^3 R_j(t,X_t(x_1),X_t(x_2))\leq \frac{C_\epsilon}{(1-t^2)\log((1-t^2)^{-1})^\theta} \|X_t(x_2)-X_t(x_1)\|^{\beta-\lfloor \beta \rfloor}$$
and 
$$\frac{t}{1-t^2}R_0(t,X_t(x_1),X_t(x_2))\leq C_\epsilon\frac{\|X_t(x_2)-X_t(x_1)\|^{\beta-\lfloor \beta \rfloor}}{\log((1-t^2)^{-1})^\theta},$$
so from Propositions \ref{prop:rodebut} and \ref{prop:boundrjdebut}, this concludes the proof of Theorem \ref{theo:thetheo2}.

\subsection{Proof of the bound on $\|\nabla^{\lfloor \beta \rfloor+1} X_1\|_{\mathcal{H}^{\beta-\lfloor \beta \rfloor}}$ for Theorem \ref{theo:thetheo}}
This section follows the same steps as Section \ref{sec:castheo1append}.
\subsubsection{Integration against $p^{t,\theta_1}-p^{t,\theta_2}$ and mass concentration}
Let us first derive a general bound on the integration against $p^{t,\theta_1}-p^{t,\theta_2}$.
\begin{proposition}\label{prop:lipptx33}
    Let $\delta_1,\delta_2\in [0,1)$, $\lambda_1\geq 0$ and $l\in \mathcal{H}_1^{\delta_1,\lambda_1}(\mathbb{R}^d)$, $h\in \mathcal{H}_1^{\delta_2}(\mathbb{R}^d)$, $z\in \mathbb{R}^d$. For $t\in [0,1)$ and $\theta_1,\theta_2\in \mathbb{R}^d$, in the case $\beta\geq 1$ we have 
    \begin{align*}
        \|&\int (l(y)-l(z))h(y)(p^{t,\theta_1}(y)-p^{t,\theta_2}(y))d\lambda^d(y)\|\\
        & \leq C \sum_{i=1}^2\left(\|\theta_1-\theta_2\|^{\delta_1}\log(\|\theta_1-\theta_2\|^{-1})^{-\lambda_1}\Biggl(\int\|h(y)\|^2p^{t,\theta_i}(y) \right)^{1/2}+\|\theta_1-\theta_2\|^{\delta_2}\left(\int\|l(y)-l(z)\|^2p^{t,\theta_i}(y) \right)^{1/2}\Biggl).
    \end{align*}
    and in the case $\beta<1$, $l(y)-l(z)=y-t\theta_1$, $h(y)=H(y,\theta_2)$ we have
    \begin{align*}
        &\|\int (y-t\theta_1)H(y,\theta_2)(p^{t,\theta_1}(y)-p^{t,\theta_2}(y))d\lambda^d(y)\|\\
         \leq &C\|\theta_1-\theta_2\|\left(\|\int H(y,\theta_2)dp^{t,\theta_2}(y) \|+\|\int \left(y-t\theta_1 \right)dp^{t,\theta_2}(y)\|\right)\\
        & +C\frac{\|\theta_1-\theta_2\|^{\beta}}{\log(\|\theta_1-\theta_2\|^{-1})^{\theta}}\left(\int\|y-t\theta_2\|^2dp^{t,\theta_2}(y) \right)^{1/2}\left(\int\|H(y,\theta_2)\|^2p^{t,\theta_1}(y) \right)^{1/2}.
    \end{align*}
\end{proposition}
\begin{proof}
We are going to bound the different terms of Proposition \ref{prop:generalboundlhp}. Let $\chi = \mathds{1}_{\{\beta<1\}}$, we have 
\begin{align*}
    |&Q_{t}r(\theta_2)-Q_{t}
    r(\theta_1)|\\
    &=|\int \left(r(t\theta_1+\sqrt{1-t^2}z)-r(t\theta_2+\sqrt{1-t^2}z)\right)\gamma_d(z)d\lambda^d(z)|\\
     & \leq C\|\theta_2-\theta_1\|^{\chi\beta+1-\chi},
\end{align*}
which gives 
\begin{align*}
    |\frac{Q_{t}r(\theta_2)}{Q_{t}
    r(\theta_1)}-1|
    & \leq C\|\theta_2-\theta_1\|^{\chi\beta+1-\chi}.
\end{align*}
Likewise, we have
\begin{align*}
    |\frac{r(y-t(\theta_2-\theta_1))}{r(y)}-1| & \leq C\|\theta_2-\theta_1\|^{\chi\beta+1-\chi},
\end{align*}
so using Proposition \ref{prop:generalboundlhp} and Cauchy-Schwarz we get the results.
\end{proof}
Let us now bound the terms
$\int\|f(y)-f(y+t(\theta_2-\theta_1))\|^2dp^{t,\theta_1}(y)$ and $\int\|f(y)-f(z)\|^2dp^{t,\theta_2}(y)$ for $f=f^l$ with $l\in \{1,...,\lfloor \beta \rfloor\}$ and $z=t\theta_2$. We have 
\begin{align}\label{align:allerundernierp}
    \left(\int \|f^l(y)-f^l(y+t(\theta_2-\theta_1))\|^2dp^{t,\theta_1}(y)\right)^{1/2}\leq C \|\theta_2-\theta_1\|^{\beta-\lfloor \beta \rfloor}\log(\|\theta_2-\theta_1\|)^{-\theta}
\end{align}
and 
\begin{align}\label{align:concentrationcas1easy}
    &\left(\int \|f^l(y)-f^l(t\theta_2)\|^2dp^{t,\theta_2}(y)\right)^{1/2}\nonumber\\ & \leq \left(\int \|y-t\theta_2\|^{2(\beta-\lfloor \beta \rfloor)}\log(\|y-t\theta_2\|^{-1})^{-2\theta}dp^{t,\theta_2}(y)\right)^{1/2}\nonumber\\
    & \leq C (1-t^2)^{\beta-\lfloor \beta \rfloor}\log((1-t^2)^{-1})^{-\theta}.
\end{align}

\subsubsection{Bounds on the different terms}\label{sec:bounddifftermi}
Let $f:=f^{l}$ with $l\in \{0,...,\lfloor \beta \rfloor\}$. We give estimates on $(R_j)_{j=0,...,3}$ recalling that the case $R_0$ corresponds to the case $\lfloor \beta \rfloor +1= 1$ and the case  $(R_j)_{j=1,2,3}$ corresponds to the case $\lfloor \beta \rfloor +1\geq 2$. We first focus on the case $t\in [0,1)$ such that $(1-t^2)^{1/2}\leq \|\theta_1-\theta_2\|$. 

For the terms $(R_j)_{j=1,2,3}$ we use the bound \eqref{align:x1x2grand}. Taking $\xi_m=f(t\theta_m)$ with $m\in \{1,2\}$ in Proposition \ref{prop:concentration1} we obtain for $i\in\{1,2\}$,
\begin{align*}
     \|\int f(y)\big(\nabla^i_x p^{t,\theta_1}(y)-\nabla^i_x p^{t,\theta_2}(y)\big)d\lambda^d(y)\|& \leq \frac{C}{(1-t^2)^{i/2}}\sum_{m=1}^2\left(\int \|f(y)-f(\xi_m)\|^2dp^{t,\theta_m}(y)\right)^{1/2}\\
     & \leq \frac{C}{(1-t^2)^{i/2}}(1-t^2)^{(\beta-\lfloor \beta \rfloor)/2}\log((1-t^2)^{-1})^{-\theta}\\
     & \leq \frac{C}{(1-t^2)^{i/2}}\|\theta_1-\theta_2\|^{\beta-\lfloor \beta \rfloor}\log((1-t^2)^{-1})^{-\theta}.
\end{align*}
Using this result we get
\begin{align*}
    R_1(t,\theta_1,\theta_2)= & \sum_{l=1}^{k-1} \|\int f^{l}(y)(\nabla^2_x dp^{t,\theta_1}(y)-\nabla^2_x dp^{t,\theta_2}(y))\| \prod_{j=1}^{k-2}\left(\|\int f^{j}(y) dp^{t,\theta_1}(y)\|\vee 1\right)^{k-2}\\
    & \leq \frac{C}{1-t^2}\|\theta_1-\theta_2\|^{\beta-\lfloor \beta \rfloor}\log((1-t^2)^{-1})^{-\theta},
\end{align*}

\begin{align*}
    R_2(t,\theta_1,\theta_2)= &  \sum_{l=1}^{k-1} \|\int f^{l}(y) (\nabla_x dp^{t,\theta_1}(y)-\nabla_x dp^{t,\theta_2}(y))\|\\
    & \times \sum_{j=1}^{k-2}\|\int f^{j}(y) \nabla_x dp^{t,\theta_1}(y)\|\prod_{i=1}^{k-2}\left(\|\int f^{i}(y)  dp^{t,\theta_1}(y)\|\nonumber\vee 1 \right)^{k-2}\\
    \leq & \frac{C}{(1-t^2)^{1/2}}\|\theta_1-\theta_2\|^{\beta-\lfloor \beta \rfloor}\log((1-t^2)^{-1})^{-\theta}\frac{1}{1-t^2}\sum_{j=1}^{k-2}\|\int f^{j}(y) \left(y-\int z dp^{t,\theta_1}(z)\right)dp^{t,\theta_1}(y)\|\\
    \leq & \frac{C}{1-t^2}\|\theta_1-\theta_2\|^{\beta-\lfloor \beta \rfloor}\log((1-t^2)^{-1})^{-\theta}
\end{align*}
and using Proposition \ref{prop:lipptx33} and \eqref{align:concentrationcas1easy},
\begin{align*}
    R_3(s,\theta_1,\theta_2)= & \sum_{j=1}^{k-2} \|\int f^{j}(y)  (dp^{s,\theta_1}(y)-dp^{s,\theta_2}(y))\|\prod_{i=1}^{k-3}\left(\|\int f^{i}(y)  dp^{s,X_s(x_1)}(y)\|\vee 1\right)^{k-3}\\
    & \times\Biggl(\sum_{l=1}^{k-1}  \|\int f^{l}(y)\nabla^2_x dp^{s,\theta_2}(y)\|+\|\int f^{l}(y)\nabla_x dp^{s,X_s(x_2)}(y)\|\sum_{r=1}^{k-1}\|\int f^{r}(y)\nabla_x dp^{s,\theta_2}(y)\|\Biggl)\\
    \leq & \frac{C}{1-t^2}\|\theta_1-\theta_2\|^{\beta-\lfloor \beta \rfloor}\log((1-t^2)^{-1})^{-\theta}.
\end{align*}

For the term $R_0$, we use the bound \eqref{eq:roboundtpetit}. On one hand, we have 
\begin{align*}
\|\int(y-t\theta_i)d p^{t,\theta_i}(y)\|^2 & \leq C\|\int(y-t\theta_i)(r(y)-r(t\theta_i))\varphi^{t,\theta_i}(y)d\lambda^d(y)\|^2\\
     \leq &C \left(\int\|y-t\theta_i\|^{1+\beta}\log(\|y-t\theta_i\|^{-1})^{-\theta}\varphi^{t,\theta_i}(y)d\lambda^d(y)\right)^2\\
     \leq & C(1-t^2)^{1+\beta}\log((1-t^2)^{-1})^{-2\theta}.
\end{align*}
On the other hand, we have 
\begin{align*}
\|&\int\left(y-t\theta_i\right)^{\otimes2}\left( \frac{r(y)}{\int r(z)\varphi^{t,\theta_i}(z)d\lambda^d(z)}-1\right)\varphi^{t,\theta_i}(y)d\lambda^d(y)\|\\
      \leq & C \|\int\left(y-t\theta_i\right)^{\otimes2}\left( r(y)-r(t\theta_i)\right)\varphi^{t,\theta_i}(y)d\lambda^d(y)\|\nonumber\\
      & + C \| \int (r(t\theta_i)-r(z))\varphi^{t,\theta_i}(z)d\lambda^d(z)\| \|\int\left(y-t\theta_i\right)^{\otimes2}\varphi^{t,\theta_i}(y)d\lambda^d(y)\|
      \nonumber\\
    \leq & C \|\int\|y-t\theta_i\|^{2+\beta}\log(\|y-t\theta_i\|^{-1})^{-\theta}\varphi^{t,\theta_i}(y)d\lambda^d(y)\|\\
    & +C\| \int \|y-t\theta_i\|^\beta\log(\|y-t\theta_i\|^{-1})^{-\theta}\varphi^{t,\theta_i}(z)d\lambda^d(z)\| \|\int\|y-t\theta_i\|^2\varphi^{t,\theta_i}(y)d\lambda^d(y)\|\\
    \leq & C(1-t^2)^{1+\beta/2}\log((1-t^2)^{-1})^{-\theta}.
\end{align*}

We can then conclude that
\begin{align}\label{align:r0boundreu}
R_0(t,\theta_1,\theta_2)&\leq C(1-t^2)^{1+\beta/2}\log((1-t^2)^{-1})^{-\theta}\\
    & \leq C(1-t^2)\|\theta_1-\theta_2\|^{\beta}\log((1-t^2)^{-1})^{-\theta}\nonumber.
\end{align}

Let us now treat the case $\|\theta_1-\theta_2\|< (1-t^2)^{1/2}$. For the term $R_0$, we have using \eqref{align:roboundtgrand} and Proposition \ref{prop:lipptx33} that
\begin{align*}
R_0(t,\theta_1,\theta_2)\leq & \frac{t}{1-t^2}\sum_{i=1}^2\|\int H(y,\theta_i) dp^{t,\theta_i}(y)\|\left(\|\theta_1-\theta_2\|+\|\int w (p^{t,\theta_1}(w)-p^{t,\theta_2}(w))d\lambda^d(w)\|\right)\nonumber\\
    & + \frac{t}{1-t^2}\|\int (y-t\theta_2)H(y,\theta_2)(p^{t,\theta_1}(y)-p^{t,\theta_2}(y))d\lambda^d(y)\|\\
    \leq & C(1-t^2)^{1/2+\beta/2}(\|\theta_1-\theta_2\|+\|\theta_1-\theta_2\|^{\beta-\lfloor \beta\rfloor}\log(\|\theta_1-\theta_2\|^{-1})^{-\theta}(1-t^2)^{1/2})\\
    &+ C\|\theta_1-\theta_2\|(1-t^2)^{1/2+\beta/2} +C(1-t^2)^{1/2}\|\theta_1-\theta_2\|^{\beta-\lfloor \beta\rfloor}\log(\|\theta_1-\theta_2\|^{-1})^{-\theta}(1-t^2)^{1/2}\\
    \leq & C(1-t^2)\|\theta_1-\theta_2\|^{\beta-\lfloor \beta\rfloor}\log(\|\theta_1-\theta_2\|^{-1})^{-\theta}.
\end{align*}

Let us now focus on $(R_{j})_{j=1,2,3}$. From \eqref{eq:premiertermlipgrand}, \eqref{align:concentrationcas1easy} and Proposition \ref{prop:lipptx33} we have that
\begin{align*}
    \|A_1\|\leq& C(1-t^2)^{1/2}\left(\int \|f(y)-f(t\theta_1)\|^2dp^{t,\theta_1}(y)\right)^{1/2} \|\int z(p^{t,\theta_1}(z)-p^{t,\theta_2}(z))d\lambda^d(z)\|\\
    \leq & C(1-t^2)\log((1-t^2)^{-1})^{-\theta}\|\theta_1-\theta_2\|^{\beta-\lfloor \beta \rfloor}.
\end{align*}
From \eqref{align:B1} we have
\begin{align*}
    \|B_1\| \leq & C (1-t^2)\left(\int \|f(y)-f(y+t(\theta_2-\theta_1))\|^2dp^{t,\theta_1}(y)\right)^{1/2}\nonumber\\
     &+ C (1-t^2)^{1/2}\left(\int \|f(y)-f(y+t(\theta_2-\theta_1))\|^2dp^{t,\theta_1}(y)\right)^{1/2}\|\int w (p^{t,\theta_1}(w)-p^{t,\theta_2}(w))d\lambda^d(w)\|\\
    \leq & C\|\theta_2-\theta_1\|^{\beta-\lfloor \beta \rfloor}\log((1-t^2)^{-1})^{-\theta} (1-t^2).
\end{align*}

From \eqref{align:D1} and \eqref{align:allerundernierp} we have 
\begin{align*}
    \|D_1\|&\leq C\|\theta_2-\theta_1\|(1-t^2)^{1/2}\left(\int \|f(y+t(\theta_2-\theta_1))-f(t\theta_1)\|^2dp^{t,\theta_1}(y)\right)^{1/2}\\
    &\leq C\|\theta_2-\theta_1\|^{\beta-\lfloor \beta \rfloor}(1-t^2)\log((1-t^2)^{-1})^{-\theta}.
\end{align*}

From \eqref{align:D2} and Proposition \ref{prop:lipptx33} we have 
\begin{align*}
    \|D_2\|
    \leq & \left(\int \|f(y+t(\theta_2-\theta_1))-f(t\theta_1)\|^2dp^{t,\theta_1}(y)\right)^{1/2}\|\int H(w,\theta_2)^{\otimes 2}(dp^{t,\theta_1}(w)- dp^{t,\theta_2}(w))\|\\
    \leq & C(1-t^2)\log((1-t^2)^{-1})^{-\theta}\|\theta_2-\theta_1\|^{\beta-\lfloor \beta \rfloor}.
\end{align*}

From \eqref{align:D3} and Proposition \ref{prop:lipptx33} we have
\begin{align*}
    \|D_3\|\leq &C \|\int (f(y)-f(t\theta_1)\left(H(y,\theta_2)^{\otimes 2}-\int H(w,\theta_2)^{\otimes 2} dp^{t,\theta_2}(w)\right)(p^{t,\theta_1}(y)-p^{t,\theta_2}(y))d\lambda^d(y)\|\\
    \leq & C(1-t^2)\log((1-t^2)^{-1})^{-\theta}\|\theta_2-\theta_1\|^{\beta-\lfloor \beta \rfloor}.
\end{align*}
Therefore, we can conclude that 
\begin{align*}
    R_1(t,\theta_1,\theta_2) & = \sum_{l=1}^{k}\|\int f^{l}(y)(\nabla^2_x dp^{t,\theta_1}(y)-\nabla^2_x dp^{t,\theta_2}(y))\| \prod_{j=0}^{k+1}\left(\|\int f^{j}(y) dp^{t,\theta_1}(y)\|\vee 1\right)\\
    & \leq  C(1-t^2)^{-1}\log((1-t^2)^{-1})^{-\theta}\|\theta_2-\theta_1\|^{\beta-\lfloor \beta \rfloor}.
\end{align*}
From \eqref{align:boundR2} we have 
\begin{align*}
     R_2(t,\theta_1,\theta_2)  = & \sum_{l=1}^{k} \|\int f^{l}(y) (\nabla_x dp^{t,\theta_1}(y)-\nabla_x dp^{t,\theta_2})\|\sum_{j=1}^{k+1}\|\int f^{j}(y) \nabla_x dp^{t,\theta_1}(y)\|\prod_{i=0}^{k+1}\left(\|\int f^{i}(y)  dp^{t,\theta_1}(y)\|\vee 1\right)\\
      \leq  & \left(\int \|f(y)-f(t\theta_1)\|^2dp^{t,\theta_1}(y)\right)^2\|\int w (p^{t,\theta_1}(w)-p^{t,\theta_2}(w))d\lambda^d(w)\|(1-t^2)^{1/2}\log((1-t^2)^{-1})^{-\theta}\\
    & + \|\int (f(y)-f(t\theta_1))H(y,\theta_2)(p^{t,\theta_1}(y)-p^{t,\theta_2}(y))d\lambda^d(y)\|(1-t^2)^{1/2}\log((1-t^2)^{-1})^{-\theta}\\
    &\leq  C(1-t^2)^{-1/2}\log((1-t^2)^{-1})^{-\theta}(1-t^2)^{(\beta-\lfloor \beta \rfloor)/2} \|\theta_1-\theta_2\|\\
    & \leq  C(1-t^2)^{-1}\log((1-t^2)^{-1})^{-\theta}\|\theta_2-\theta_1\|^{\beta-\lfloor \beta \rfloor}.
\end{align*}

Finally
\begin{align*}
     R_3(t,\theta_1,\theta_2)  = &\sum_{j=0}^{k} \|\int f^{j}(y)  (dp^{t,\theta_1}(y)-dp^{t,\theta_2}(y))\|\prod_{i=0}^{k+1}\left(\|\int f^{i}(y)  dp^{t,\theta_1}(y)\|\vee 1\right)\\
    & \times\Biggl(\sum_{l=1}^{k}  \|\int f^{l}(y)\nabla^2_x dp^{t,\theta_2}(y)\|+\|\int f^{l}(y)\nabla_x dp^{t,\theta_2}(y)\|\sum_{i=1}^{k}\|\int f^{i}(y)\nabla_x dp^{t,\theta_2}(y)\|\Biggl)\\
    \leq & C \|\theta_2-\theta_1\|^{\beta-\lfloor \beta \rfloor}\log((1-t^2)^{-1})^{-\theta}(1-t^2)^{-1}.
\end{align*}

Putting together the bounds on $R_1,R_2,R_3$ we finally obtain for all $t\in [0,1)$ 
$$\sum_{j=1}^3 R_j(t,X_t(x_1),X_t(x_2))\leq \frac{C}{(1-t^2)\log((1-t^2)^{-1})^\theta} \|X_t(x_2)-X_t(x_1)\|^{\beta-\lfloor \beta \rfloor}$$
and 
$$\frac{t}{1-t^2}R_0(t,X_t(x_1),X_t(x_2))\leq C\frac{\|X_t(x_2)-X_t(x_1)\|^{\beta-\lfloor \beta \rfloor}}{\log((1-t^2)^{-1})^\theta},$$
so from Propositions \ref{prop:rodebut} and \ref{prop:boundrjdebut}, this concludes the proof of Theorem \ref{theo:thetheo}.

\subsection{Additional proofs}
\subsubsection{$C^{\beta+1}$ diffeomorphism between deformations of the Gaussian}\label{sec:diffeog}
This section is dedicated to the proof of Corollary \ref{theo:diffeo}.

\begin{proof}Let us take $p(x):=\exp\left(-\frac{\|x\|^2}{2}+a(x)\right)$. We first show that the inverse of the Langevin transport map from the Gaussian to $p$ is also $(\beta+1)$-Hölder. Let $(X_t)_{t\in [0,1]}$ be the flow of diffeomorphism (solution to \eqref{eq:PDE}) such that $(X_1)_{\# \gamma_d =p}$. 
Let $w\in \mathbb{S}^{d-1}$ and $\alpha_t=\|\nabla X_t(x) w\|^2$, we have
\begin{align*}
    \partial_t \alpha_t &= 2\left\langle \nabla X_t(x)w, \nabla V(t,X_t(x)) \nabla X_t(x)w\right\rangle\\
    & \geq 2 \lambda(t)\alpha_t,
\end{align*}
with $\lambda(t) = \min_{u\in \mathbb{S}^{d-1}} \left\langle  u, \nabla V(t,X_t(x)) u\right\rangle$. Then
\begin{align*}
    (\partial_t \alpha_t + (-2 \lambda(t))\alpha_t)\exp(-\int_0^t 2\lambda(s)ds)\geq 0
\end{align*}
so 
\begin{align*}
    \partial_t (\alpha_t\exp(-\int_0^t 2\lambda(s)ds))\geq 0
\end{align*}
which gives
\begin{align*}
    \alpha_t\exp(-\int_0^t 2\lambda(s)ds)\geq \alpha_0,
\end{align*}
so finally as $\alpha_0=1$ we have 
\begin{align*}
    \alpha_t\geq \exp(\int_0^t 2\lambda(s)ds).
\end{align*}
Now, from \eqref{align:r0boundreu} we have that in the case $\beta\in (0,1)$,
\begin{align*}
    \|\nabla V(t,x)\|&\leq C(1-t^2)^{\beta/2-1}\log((1-t^2)^{-1})^{-\theta},
\end{align*}
and in the case $\beta\geq 1$, using Proposition \ref{prop:concentHtheo1} we get
\begin{align*}
    \|\nabla V(t,x)\| & = \frac{t}{1-t^2}\|\int \nabla a(y)\left(y-\int z dp^{t,x}(z)\right)dp^{t,x}(y)\|\\
    & \leq C\frac{1}{1-t^2}\left(\int \|y-\int z dp^{t,x}(z)\|^2dp^{t,x}(y)\right)^{1/2}\\
    & \leq C\frac{1}{(1-t^2)^{1/2}}.
\end{align*}
We then deduce that
$$\alpha_t\geq \exp(-C\int_0^t \left(\frac{\mathds{1}_{\{\beta\geq1\}}}{(1-s^2)^{1/2}}+\frac{\mathds{1}_{\{\beta\in (0,1)\}}}{(1-s^2)^{1-\beta/2}}\right)ds) \geq C^{-1}.$$
Therefore, we obtain that the map $T$ from Theorem \ref{theo:thetheo} satisfies $\inf_{x\in \mathbb{R}^{d}}\min_{w\in \mathbb{S}^{d-1}} \|\nabla T^{-1}(x)w \|\geq C^{-1}$ so by the Faa Di Bruno formula we obtain that $\nabla T^{-1}\in \mathcal{H}^{\beta}_C   (\mathbb{R}^d)$.

Let us now take two density probabilities
$$p(x)=\exp\left(-\frac{1}{2}\|\Sigma_p^{-1/2}x\|^2+a(x)\right) \quad \text{ and } \quad q(x)=\exp\left(-\frac{1}{2}\|\Sigma_q^{-1/2}x\|^2+b(x)\right).$$

Then for $p^{'}=(\Sigma_p^{-1/2})_{\#} p$ and $q^{'}=(\Sigma_q^{-1/2})_{\#} q$, we have that $p^{'}$ and $q^{'}$ satisfy the assumptions of Theorem \ref{theo:thetheo} so there exist $T_p,T_q$ such that $\nabla T_p,\nabla T_p^{-1},\nabla T_q,\nabla T_q^{-1}\in \mathcal{H}^{\beta}_C   (\mathbb{R}^d)$ and $(T_p)_{\#}\gamma_d =p^{'}$ and $(T_q)_{\#}\gamma_d =q^{'}$.

Therefore, the map $T=\Sigma_q^{1/2}\circ T_q\circ T^{-1}_p\circ \Sigma_p^{-1/2}$ satisfies $\nabla T,\nabla T^{-1}\in  \mathcal{H}^{\beta}_C  (\mathbb{R}^d)$ and $T_{\# }p=q$. The general case for non-centered Gaussian distributions follows the same way.
\end{proof}

\end{document}